\documentclass{amsart}[11pt]
\usepackage[margin=1.0in]{geometry}

\usepackage{inputenc}
\usepackage{amsmath}
\usepackage{amssymb, amsfonts}
\usepackage[mathscr]{eucal}
\usepackage{graphicx,epstopdf,bm}
\usepackage{amsaddr}
\usepackage{wrapfig}
\usepackage{subfig}
\usepackage{placeins}
\usepackage{enumitem,kantlipsum}

\usepackage{dcolumn}%
\usepackage[title]{appendix}
\usepackage[colorlinks,bookmarks,urlcolor=blue,citecolor=blue,linkcolor=blue]{hyperref}
\usepackage[capitalize]{cleveref}
\usepackage{bbm}
\usepackage{mathtools}
\usepackage{multirow}
\usepackage{xcolor}

\usepackage{algorithm}
\usepackage{algpseudocode}

\usepackage{amsthm}
\numberwithin{equation}{section}
\theoremstyle{plain}
\newtheorem{theorem}{Theorem}[section]
\newtheorem{definition}{Definition}[section]

\newtheorem{lemma}[theorem]{Lemma}
\newtheorem{assumption}[theorem]{Assumption}

\newtheorem{remark}[theorem]{Remark}
\allowdisplaybreaks

\usepackage{bbm}

\renewcommand{\vec}[1]{\boldsymbol{#1}}

\newcommand{\paren}[1]{\left(#1\right)}
\newcommand{\brac}[1]{\left[#1\right]}

\newcommand{\lap}[1]{\Delta#1}

\newcommand{\abs}[1]{\left|#1\right|}
\newcommand{\norm}[1]{\Vert#1\Vert}

\newcommand{\DA}{D^{\textrm{A}}}
\newcommand{\DB}{D^{\textrm{B}}}
\newcommand{\DC}{D^{\textrm{C}}}

\newcommand{\vx}{\vec{x}}
\newcommand{\vy}{\vec{y}}

\newcommand{\vQ}{\vec{Q}}

\def\R{\mathbb{R}}

\newcommand{\Kd}{K_{\textrm{d}}}

\newcommand{\la}{\left \langle}
\newcommand{\ra}{\right\rangle}

\newcommand{\X}{\mathbb{X}}
\newcommand{\Y}{\mathbb{Y}}
\newcommand{\vecrho}{\vec{\rho}}

\newcommand{\RN}[1]{
  \textup{\uppercase\expandafter{\romannumeral#1}}%
}

\newcommand{\commentab}[1]{{\color{black}{#1}}}

\makeatletter
\def\mathcolor#1#{\@mathcolor{#1}}
\def\@mathcolor#1#2#3{%
  \protect\leavevmode
  \begingroup
    \color#1{#2}#3%
  \endgroup
}
\makeatother

\begin{document}
\title{How reaction-diffusion PDEs approximate the large-population limit of stochastic particle models.}

\author{S. Isaacson, J. Ma and K. Spiliopoulos}
\footnote{Boston University, Department of Mathematics and Statistics, 111 Cummington Mall, Boston, 02215.
E-mails: isaacson@math.bu.edu, majw@bu.edu, kspiliop@math.bu.edu}
\thanks{K.S. was partially supported by the NSF DMS-1550918, Simons Foundation Award 672441 and ARO W911NF-20-1-0244. J. M. and S. A. I. were partially supported by NSF DMS-1902854 and ARO W911NF-20-1-0244.}

\begin{abstract}
Reaction-diffusion PDEs and particle-based stochastic reaction-diffusion (PBSRD) models are commonly-used approaches for modeling the spatial dynamics of chemical and biological systems. Standard reaction-diffusion PDE models ignore the underlying stochasticity of spatial transport and reactions, and are often described as appropriate in regimes where there are large numbers of particles in a system. Recent studies have proven the rigorous large-population limit of PBSRD models, showing the resulting mean-field models (MFM) correspond to non-local systems of partial-integro differential equations. In this work we explore the rigorous relationship between standard reaction-diffusion PDE models and the derived MFM. We prove that the former can be interpreted as an asymptotic approximation to the later in the limit that bimolecular reaction kernels are short-range and averaging. As the reactive interaction length scale approaches zero, we prove the MFMs converge at second order to standard reaction-diffusion PDE models. In proving this result we also establish local well-posedness of the MFM model in time for general systems, and global well-posedness for specific reaction systems and kernels. \commentab{Finally, we illustrate the agreement and disagreement between the MFM, SM and the underlying particle model for several numerical examples.}

\end{abstract}
\date{\today}

\maketitle


\section{Introduction}
Reaction-diffusion partial differential equations (PDEs) are often used to model the average or large-population dynamics of systems of reacting and diffusing particles at a macroscopic level. In biological contexts, such models may describe signaling pathways inside cells, interactions within populations (of cells, animals or people), or chemical signals within tissues. Standard mass-action-based reaction-diffusion PDE models (subsequently \textbf{SM} for standard model) in the form of \cref{Eq:density_formula_PDEs} in Section \ref{S:GeneralSetup},  are extensively used in modeling studies, for example \cite{ECM:2007, MNKS:2009, NTS:2008}. They can be formally obtained from classical mass-action ODE models for chemical reactions \cite{G:2000} by adding a diffusion operator (or more generally a second order elliptic operator) to model spatial transport. However, such a formulation does not capture the detailed spatial reaction mechanisms described in more microscopic particle-based stochastic reaction-diffusion (PBSRD) models~\cite{TeramotoDoiModel1967,DoiSecondQuantA,DoiSecondQuantB}.  The latter explicitly resolve the diffusion of, and reactions between, individual molecules, which are often represented as point particles.

Due to their mathematical complexity and high dimensionality, PBSRD models are typically studied by Monte Carlo simulations approximating the underlying stochastic process of molecules diffusing and reacting. This is often computationally expensive, particularly in systems with large populations for which the dynamics are well-approximated as deterministic. In order to deal with this issue, coarse-grained models were developed in \cite{LLN:2019,IMS:2020} by proving the large population mean-field limit of the measure-valued stochastic process (MVSP) tracing the empirical distribution of the particles. In this limit the empirical distribution can be rigorously shown to satisfy an evolution equation for which the corresponding limiting density satisfies a system of partial-integro differential equations (PIDEs), e.q.~\cref{Eq:density_formula_PIDEs} in Section~\ref{S:GeneralSetup}. We subsequently call these PIDEs the mean-field model (\textbf{MFM}). In contrast to the \textbf{SM}, the nonlocal  \textbf{MFM} retains the detailed spatial reaction mechanisms of the underlying PBSRD model.

\commentab{The new contributions of this paper} are to study the relationship between the formal \textbf{SM}  and the coarse-grained and rigorously derived \textbf{MFM} for PBSRD systems. We show that the \textbf{SM} can be interpreted as an asymptotic approximation to the more general \textbf{MFM} when the underlying reactive interaction kernels are averaging and short-range. In particular, we prove that the solution to the \textbf{MFM} converges to the solution of the \textbf{SM} for such kernels as the kernel's width approaches zero. The convergence is proven to be second order in the approximation parameter that captures the kernel width, see Theorem~\ref{thm:convTwoModels}. We also demonstrate by numerical simulations the convergence as the kernel width approaches zero of the \textbf{MFM} to the \textbf{SM}, show how well the \textbf{MFM} and \textbf{SM} agree with the average concentration fields of PBSRD models, and demonstrate that when the reactive interaction distance between two particles is not sufficiently "small" the \textbf{SM} may provide a poor approximation to either the \textbf{MFM} or mean concentration fields of the PBSRD models. \commentab{To illustrate a biologically-relevant case where the MFM is needed, we conclude with a simple model motivated by our recent study of T-cell receptor signaling~\cite{ZI:2020}. There it is not immediately clear how to even define an appropriate SM. In contrast, we demonstrate that the (well-defined) MFM correctly reproduces the qualitative steady-state behavior of the underlying particle model observed in~\cite{ZI:2020}.}

\commentab{As our main result is presented for systems with general reaction networks, and relatively notationally heavy, we now give a brief summary of it in the special case of the multi-particle $\textrm{A} + \textrm{B} \leftrightarrows \textrm{C}$ reaction.} We assume all molecules move by Brownian motion in $\R^{d}$ with species-dependent diffusivities, $\DA$, $\DB$ and $\DC$. Assume an \textrm{A} molecule at $x$ and \textrm{B} molecule at $y$ may react with probability per time $\hat{K}_{1}^{\epsilon}(x-y)/\gamma$. Here $\gamma$ denotes a large system size parameter, for example Avogadro's number in $\R^{d}$~\cite{LLN:2019,IMS:2020}. $\epsilon$ denotes the length scale of the kernel, related to our assumption that it has an averaging form, i.e. that $\hat{K}_{1}^{\epsilon}(x-y) \to \kappa_{1} \delta(x-y)$ as $\epsilon \to 0$. Common choices for $\hat{K}_1^{\epsilon}/\gamma$ are given in \cref{Remark:kernels}, and include the popular Doi model in which the two reactants may react with a fixed probability per time when within $\epsilon$~\cite{TeramotoDoiModel1967,DoiSecondQuantA,DoiSecondQuantB}. We denote by $m_{1}(z \vert x,y)$ the probability density that when an \textrm{A} molecule at $x$ and a \textrm{B} molecule at $y$ react they produce at \textrm{C} molecule at $z$. $m_{1}$ is commonly chosen to place $z$ on the line connecting $x$ and $y$, see~\cref{Assume:measureTwo2One}. For the reverse reaction we assume $K_{2}(z)=\kappa_{2}$ denotes the constant probability per time an individual \textrm{C} molecule splits into \textrm{A} and \textrm{B} molecules, with $m_{2}^{\epsilon}(x,y\vert z)$ denoting the probability density that a \textrm{C} at $z$ which dissociates produces an \textrm{A} molecule at $x$ and a \textrm{B} molecule at $y$. The dependence of $m_{2}^{\epsilon}$ on $\epsilon$ follows under the assumption of detailed balance for the reversible reaction, see \cref{Assume:measureOne2Two} and \cref{Assume:measureMrho}.

With these choices, the large population limit for the volume reactivity PBSRD model of the particles diffusing and reacting, \commentab{a special case of the limit we proved for general systems in~~\cite{IMS:2020}}, shows that the molar concentration fields for each species converge in a weak sense to the solution of a nonlocal system of PIDEs. More concretely, denote by $A(t)$ the stochastic process for the number of species \textrm{A} molecules at time $t$, and label the position of $i$th molecule of species \textrm{A} at time $t$ by the stochastic process $\vQ_i^{A(t)} \subset \R^d$. The random measure
\begin{equation*}
A^{\gamma}_{\epsilon}(x,t) = \frac{1}{\gamma} \sum_{i=1}^{A(t)} \delta\paren{x - \vQ_i^{A(t)}}
\end{equation*}
corresponds to the stochastic process for the molar concentration of species \textrm{A} at $x$ at time $t$. We can similarly define $B^{\gamma}_{\epsilon}(x,t)$ and $C^{\gamma}_{\epsilon}(x,t)$. The large population (thermodynamic) limit where $\gamma \to \infty$ and $(A^{\gamma}_{\epsilon}(x,0),B^{\gamma}_{\epsilon}(x,0),C^{\gamma}_{\epsilon}(x,0))$ converge to well-defined limiting molar concentration fields gives (in a weak sense) that
\begin{align*}
 \paren{A^{\gamma}_{\epsilon}(x,t),B^{\gamma}_{\epsilon}(x,t),C^{\gamma}_{\epsilon}(x,t)} \to \paren{A_{\epsilon}(x,t), B_{\epsilon}(x,t), C_{\epsilon}(x,t)},
\end{align*}
which satisfy the \textbf{MFM}
\begin{align}
\partial_t A_\epsilon(x, t) &=  D_1\lap_x A_\epsilon(x, t) -  \left(\int_{\R^d} K_{1}^\epsilon(x - y) B_\epsilon(y, t) \, dy\right) A_\epsilon(x, t) + \int_{\R^d}K_{2}(z)\left( \int_{\R^d} m_{2}^\epsilon(x,y|z) dy \right)C_\epsilon(z, t)\, dz \nonumber\\
\partial_t B_\epsilon(y, t)&=  D_2\lap_y B_\epsilon(y, t) - \left(\int_{\R^d} K_{1}^\epsilon(x - y) A_\epsilon(x, t) \, dx\right) B_\epsilon(y, t)+ \int_{\R^d}K_{2}(z)\left( \int_{\R^d} m_{2}^\epsilon(x,y|z) dx \right)C_\epsilon(z, t)\, dz \nonumber\\
\partial_t C_\epsilon(z, t)&= D_3\lap_z C_\epsilon(z, t) -  K_{2}(z)C_\epsilon(z, t) + \int_{\R^{2d}} K_{1}(x, y)m_1(z | x, y) A_\epsilon(x, t)B_\epsilon(y, t) \, dx \, dy.\nonumber\\
\label{Eq:density_formula_reversible_PIDE}
\end{align}

The commonly used \textbf{SM} is a more formally derived system of reaction-diffusion PDEs. Denote by $A(x, t), B(y, t), C(z, t)$ the molar concentration fields for species \textrm{A}, \textrm{B}, \textrm{C} in the \textbf{SM}, which satisfy
\begin{align}
\partial_tA(x, t) &=  D_1\lap_x A(x, t) - \kappa_1A(x, t)  B(x, t) + \kappa_2 C(x, t), \nonumber\\
\partial_t B(y, t)&=  D_2 \lap_y  B(y, t) - \kappa_1 A(y, t) B(y, t)+ \kappa_2 C(z, t), \nonumber\\
\partial_t C(z, t)&=  D_3 \lap_z C(z, t) -  \kappa_2 C(z, t) + \kappa_1 A(z, t) B(z, t) .
\label{Eq:density_formula_reversible_PDE}
\end{align}
\commentab{The main theoretical result of this paper}, \cref{thm:convTwoModels}, proves that when the kernels and placement measures have averaging forms,
\begin{equation*}
\paren{A_{\epsilon}(x,t), B_{\epsilon}(x,t), C_{\epsilon}(x,t)} = (A(x,t),B(x,y),C(x,t)) + O(\epsilon^{2}), \quad \epsilon \to 0.
\end{equation*}
This demonstrates that we may interpret the \textbf{SM} as an approximation to the rigorous \textbf{MFM} when reactive interactions are short-range.

Nonlocal reaction-diffusion PIDE models related to our \textbf{MFM} have been introduced and studied for a variety of physical applications, including for models of population dynamics, evolutionary dynamics, and neuronal dynamics~\cite{NTY:2017, CFF:2019, SVA:2013, JG:1989, SSS:2020, PGB:2019, GB:1996}. These works focus on the effects of different nonlocal reaction kernels, and how the stability of steady states is influenced by the nonlocal interactions. To our knowledge, there is no rigorous analysis on the closeness between such nonlocal PIDE models and corresponding local reaction-diffusion models analogous to our \textbf{SM}. Our work is also distinguished in the generality of the considered chemical reaction systems, the range of spatial reaction rate kernels, and the variety of reaction product placement measures that are allowed.

As we completed this work, we became aware of the recent preprint~\cite{K:2020}. The authors provide an intuitive, formal argument for going from the Doi PBSRD model to the \textbf{SM} in the large-population limit for systems with bimolecular reactions. This involves discretizing the forward equation for the PBSRD model to a spatially-discrete convergent reaction-diffusion master equation (CRDME) model~\cite{I:2013,IZ:2018}, assuming that the CRDME's second moments can be approximated as products of first moments in the large-population limit, and assuming that interaction length scales are sufficiently small that convolution sums representing bimolecular interactions between molecules at nearby lattice locations can be approximated by point interactions. Under these assumptions, a PDE corresponding to the \textbf{SM} for a bimolecular reaction is obtained when the lattice spacing within an approximating mean-field model for the CRDME is taken to zero. The present work is distinguished from such studies through our rigorous and direct proof that the \textbf{SM} is the short (bimolecular) interaction-range limit of the \textbf{MFM} (which was previously proven to be the rigorous large-population limit of the PBSRD model~\cite{LLN:2019,IMS:2020}). It is also distinguished in providing an error bound that demonstrates the rate of convergence of the \textbf{MFM} to the \textbf{SM} as the bimolecular interaction-distance is decreased.

\commentab{Finally, we reiterate that in \cite{IMS:2020} we proved that the \textbf{MFM}  is the correct large population (thermodynamic) limit as $\gamma\rightarrow\infty$ of a wide-class of particle-based stochastic reaction diffusion models. In contrast, this paper explores the rigorous relationship between the \textbf{SM} and the \textbf{MFM} as $\epsilon$ is varied as explained above.  The rest of the paper is organized as follows. In \cref{S:GeneralSetup} we introduce the general setup, notation and our assumptions. In \cref{S:MainThm} we present our main result, \cref{thm:convTwoModels}, on the approximation of the \textbf{MFM} by the \textbf{SM} as the bimolecular reactive interaction distance $\epsilon \to 0$. In addition, we also present a number of numerical studies that demonstrate the (rate of) convergence of the \textbf{MFM} to the \textbf{SM} as $\epsilon \to 0$ \commentab{(when both exist)}, and illustrate how well the \textbf{MFM} and \textbf{SM} models agree with the underlying PBSRD model for varying values of $\epsilon$. Our numerical results first demonstrate in several simple systems the theoretical findings that the \textbf{SM} is the short (bimolecular) interaction-range limit of the \textbf{MFM}, see Section \ref{SS:Numerics}. We then explore several biologically-motivated examples, see Section \ref{SS:NumericsBeyond}, where the \textbf{MFM} correctly captures essential behavior of the underlying particle-based stochastic system, but the \textbf{SM} sometimes fails to do so (or it is not immediately clear how to even define the \textbf{SM}). These illustrate the utility in using a rigorously derived-MFM for both model formulation when the \textbf{SM} is not clear, and as a means of assessing the accuracy of the \textbf{SM}. \cref{S:Local} proves local well-posedness of the \textbf{MFM} and the \textbf{SM} models. \cref{S:Global} discusses global well-posedness of the \textbf{MFM} for specific choices of the reaction kernels. Appendices~\ref{A:appendixLemDiff} and~\ref{A:appendixLemDerivative} give proofs of technical results used throughout the paper.}

\section{General Setup and Main Assumptions}\label{S:GeneralSetup}

For simplicity, in our rigorous studies we assume that molecules diffuse freely in $\R^d$, with $D_j$ denoting the diffusion coefficient for species $S_j$, $ j = 1,\cdots, J$. Let $L$ be the number of possible reactions, each labelled by $\mathscr{R}_1, \cdots , \mathscr{R}_L$. The change in the number of particles induced by the $\mathscr{R}_\ell$th reaction,  $\ell \in \{1,\dots,L\}$, is given by the chemical equation
\begin{equation*}
\sum_{j = 1}^{J} \alpha_{\ell j}S_j \rightarrow \sum_{j = 1}^{J} \beta_{\ell j}S_j,
\end{equation*}
where we assume the stoichiometric coefficients $\{\alpha_{\ell j}\}_{j=1}^{J}$ and $\{\beta_{\ell j}\}_{j=1}^J$ are non-negative integers. Let $\vec{\alpha}^{(\ell)} = (\alpha_{\ell 1}, \alpha_{\ell 2}, \cdots, \alpha_{\ell J})$ and $\vec{\beta}^{(\ell)} = (\beta_{\ell 1}, \beta_{\ell 2}, \cdots, \beta_{\ell J})$ be multi-index vectors collecting the coefficients of the $\ell$th reaction. We denote the reactant and product orders of the reaction by $|\vec{\alpha}^{(\ell)}|\doteq\sum_{i = 1}^{J} \alpha_{\ell i} \leq 2$ and $|\vec{\beta}^{(\ell)}|\doteq\sum_{j = 1}^{J} \beta_{\ell j}\leq 2$, assuming that at most two reactants and two products participate in any reaction. We therefore implicitly assume all reactions are at most second order.

We introduce two notations to encode reactant and product particle positions, that are subsequently needed to specify the \textbf{MFM}. \commentab{Note, these definitions are the same under which we derived the MFM in~\cite{IMS:2020}.}
\begin{definition}\label{def:reacPosSpace}
For reaction $\mathscr{R}_{\ell}$ we define the reactant position space
\begin{equation*}
\mathbb{X}^{(\ell)} = \{ \vec{x} = (x^{(1)}_1, \cdots, x^{(1)}_{\alpha_{\ell 1}}, \cdots, x^{(J)}_1, \cdots, x^{(J)}_{\alpha_{\ell J}}) \, | \,  x_r^{(j)} \in \R^d, \text{ for all } 1\leq j\leq J, 1\leq r\leq \alpha_{\ell j} \}  = \left(\R^d\right)^{|\vec{\alpha}^{(\ell)}|}.
\end{equation*}
Here when $\alpha_{\ell j} = 0$ species $j$ is not a reactant for the $\ell$th reaction, and hence there will be no indices for particles of species $j$ within the reactant position space. In the underlying PBSRD model, $\vec{x}\in \mathbb{X}^{(\ell)}$ represents one possible configuration for the (sampled) positions of individual reactant particles that might undergo an $\mathscr{R}_{\ell}$ reaction. $x_r^{(j)}$ then labels the sampled position for the $r$th reactant particle of species $j$ involved in this specific instance of the reaction. We let $d\vec{x} = \left( \bigwedge_{j = 1}^J (\bigwedge_{r = 1}^{ \alpha_{\ell j}} d x_r^{(j)}) \right) $ be the corresponding volume form on $\mathbb{X}^{(\ell)}$.
\end{definition}

\begin{definition}\label{def:prodPosSpace}
For reaction $\mathscr{R}_{\ell}$ with $1\leq \ell\leq L$ i.e. having at least one product particle we define the product position space
\begin{equation*}
\mathbb{Y}^{(\ell)} = \{ \vec{y} = (y^{(1)}_1, \cdots, y^{(1)}_{\beta_{\ell 1}}, \cdots, y^{(J)}_1, \cdots, y^{(J)}_{\beta_{\ell J}}) \, | \,  y_r^{(j)} \in \R^d, \text{ for all } 1\leq j\leq J, 1\leq r\leq \beta_{\ell j} \} = \left(\R^d\right)^{|\vec{\beta}^{(\ell)}|}.
\end{equation*}
Here when $\beta_{\ell j} = 0$ species $j$ is not a product for the $\ell$th reaction, and hence there will be no indices for particles of species $j$ within the product position space. In the underlying PBSRD model, $\vec{y}\in \mathbb{Y}^{(\ell)}$ represents one possible configuration for the (sampled) positions of individual product particles that might be produced by an $\mathscr{R}_{\ell}$ reaction. $y_r^{(j)}$ then labels the sampled position for the $r$th product particle of species $j$ involved in this specific instantance of the reaction. Let $d\vec{y} = \left( \bigwedge_{j = 1}^J (\bigwedge_{r = 1}^{ \beta_{\ell j}} d y_r^{(j)}) \right) $ be the corresponding volume form on $\mathbb{Y}^{(\ell)}$.
\end{definition}

\begin{remark}
In the case of no product for reaction $\mathscr{R}_{\ell}$ with $1\leq \ell\leq L$, for example, $A\to\emptyset$, $A+B\to\emptyset$, etc, we naturally set the product position space $\mathbb{Y}^{(\ell)} $ defined in \cref{def:prodPosSpace} to be the empty set.

\end{remark}

We denote by $K_\ell(\vec{x})$ the reaction rate kernel (i.e. probability per time) that reactant particles with positions $\vec{x}\in \mathbb{X}^{(\ell)}$ undergo reaction $\mathscr{R}_{\ell}$.  $m_\ell( \vec{y} \, | \,  \vec{x})$ will represent the reaction's placement measure, giving the probability density that when the reaction occurs reactants at positions $\vec{x}\in \mathbb{X}^{(\ell)}$ generate products at positions $\vec{y}\in \mathbb{Y}^{(\ell)}$.

We formulate the dynamics of the \textbf{MFM} in terms of the time evolution of the spatial molar concentration field for species $j$ at time $t$, denoted by $\rho_j(x, t)$, $j = 1, 2, \cdots, J$. As  summarized in the introduction, in~\cite{IMS:2020} we proved the coarse-grained large-population limit of the PBSRD model of particles diffusing and reacting is given by a coarse-grained system of PIDEs corresponding to our \textbf{MFM}. In contrast to the \textbf{SM}, the latter accounts for spatially distributed chemical interactions between particles in a manner that is consistent with the PBSRD model. For a general system of reacting and diffusing particles with the notation defined above, the derived coarse-grained \textbf{MFM} is given by the coupled system of (non-local) reaction-diffusion PIDEs
\begin{align}
\partial_t\rho_j(x, t)
&
= D_{j} \lap_x \rho_j(x, t) - \sum_{\ell = 1}^L \paren{
\frac{1}{\vec{\alpha}^{(\ell)}!}  \sum_{r = 1}^{\alpha_{\ell j}} \int_{\tilde{\vec{x}} \in \mathbb{X}^{(\ell)}}   \delta_{x}(\tilde{x}_r^{(j)})  K_\ell(\tilde{\vec{x}}) \, \left( \Pi_{k = 1}^{J} \Pi_{s = 1}^{\alpha_{\ell k}}  \rho_{k}(\tilde{x}_{s}^{(k)}, t)\right) \, d\tilde{\vec{x}}
}.  \nonumber\\
&
\quad+\sum_{\ell = 1}^L \paren{  \frac{1}{\vec{\alpha}^{(\ell)}!} \sum_{r = 1}^{\beta_{\ell j}}  \int_{\tilde{\vec{x}} \in\mathbb{X}^{(\ell)}}  K_\ell(\tilde{\vec{x}}) \left( \int_{\vy \in \mathbb{Y}^{(\ell)}}   \delta_{x}(y_r^{(j)}) m_\ell(\vec{y}\, | \,\tilde{\vec{x}}) \,d \vec{y} \right) \left( \Pi_{k = 1}^J \Pi_{s = 1}^{\alpha_{\ell k}}  \rho_{k}(\tilde{x}_s^{(k)}, t)\right) \, d\tilde{\vec{x}}
},
\label{Eq:density_formula_PIDEs}
\end{align}
where $j = 1,\dots,J$.

\begin{remark}
In \cref{Eq:density_formula_PIDEs} the expressions $\int  \delta_{x}(\tilde{x}_r^{(j)}) \cdots \, d\tilde{x}_r^{(j)}$ and $\int  \delta_{x}(y_r^{(j)}) \cdots \, dy_r^{(j)}$ are used as convenient and systematic notations to represent replacing $\tilde{x}_r^{(j)}$ and $y_r^{(j)}$ with $x$ through the formal action of the $\delta$-function. That is
\begin{equation*}
\int_{{\R^{d}}}  \delta_{x}(\tilde{x}_r^{(j)}) f(\tilde{x}_r^{(j)}) \, d\tilde{x}_r^{(j)} \doteq f(x).
\end{equation*}
\end{remark}

In contrast to the \textbf{MFM}, the standard reaction-diffusion PDE model (\textbf{SM}) used for modeling the time evolution of chemical concentration fields extends spatially-homogeneous law of mass action-based ODE models by adding Laplacians to model molecular diffusion~\cite{ECM:2007, MNKS:2009, NTS:2008}. The \textbf{SM} involves only local chemical interactions, a major simplification from more-detailed PBSRD models, with reactions occurring based on a spatially uniform reaction rate constant inherited from the underlying mass-action ODE model. For reaction $\mathscr{R}_\ell$ we denote this constant by $\kappa_\ell$. The \textbf{SM} is then the coupled system of PDEs

\begin{align}
\partial_t\rho_j(x, t)
&
= D_{j} \lap_x \rho_j(x, t) + \sum_{\ell = 1}^L \kappa_\ell (\beta_{\ell j} - \alpha_{\ell j}) \paren{ \Pi_{k = 1}^J \rho_{k}(x, t)^{\alpha_{\ell k}} }, \quad j = 1,\dots,J.
\label{Eq:density_formula_PDEs}
\end{align}

\subsection{Assumptions for  \textbf{MFM}}\label{Ass:MainAssumptions}

To discuss the connection between the \textbf{MFM} and \textbf{SM} models,  we introduce a reactive interaction scale parameter $\epsilon$ in the following assumptions on the reaction kernels and placement measures for the \textbf{MFM}. In the remainder, we will exploit that in many applications $\epsilon$ is a small compared to macroscopic length-scales over which we are interested in the dynamics of the concentration fields.

\begin{assumption}\label{Assume:kernelTwo}
\commentab{We assume that for bimolecular reactions the rate kernels $K_\ell(x, y)$, $x, y\in \R^d$ depend only on the separation distance, $\abs{x - y}$, of the two particles. In particular, we shall write $K_\ell(x, y) = K^\epsilon_\ell(x, y) = \hat{K}_\ell^\epsilon(x-y) = \hat{K}_\ell^\epsilon(y-x) $, where $\hat{K}$ has domain $\R^d$, to indicate this more explicitly, and to indicate the length scale over which the reaction can occur. We also assume that $||\hat{K}_\ell^{\epsilon}(w)||_{L^1} = k_\ell$ and $\int_{\R^d} \hat{K}^\epsilon_\ell(w)  |w|^2 \, dw = O(\epsilon^2)$ for all $\epsilon$. This last assumption is key for determining the rate of convergence of the \textbf{MFM} to the \textbf{SM} we establish in~\cref{thm:convTwoModels} as $\epsilon \to 0$. Finally, as $\epsilon\to 0$ we assume that $\hat{K}_\ell^\epsilon(w)\to k_\ell\delta(w)$ in distribution.}
\end{assumption}

\begin{remark} \label{Remark:kernels}
We previously showed how such $\epsilon$-scalings arise naturally when calibrating parameters in the PBSRD model to recover known/measured well-mixed parameters in the fast diffusion limit~\cite{IMS:2020}. For example, with such calibrations \cref{Assume:kernelTwo} is satisfied for the commonly used Doi kernel
\begin{equation*}
\hat{K}^\epsilon_\ell(w) = \frac{k_\ell}{|B_\epsilon(0)|}\mathbbm{1}_{\{|w|\leq \epsilon\}},
\end{equation*}
where $|B_\epsilon(0)|$ is the volume of d-dimensional ball with radius $\epsilon$. The scaling can also be derived from more microscopic polymer models of tethered interactions between membrane bound proteins, which give the coarse-grained Gaussian kernel
\begin{equation*}
\hat{K}^\epsilon_\ell(w) = \frac{k_\ell}{\paren{\sqrt{2\pi\epsilon^2}}^d} e^{-\frac{|w|^2}{2\epsilon^2}},
 \end{equation*}
see~\cite{ZI:2020} for details and references.

\commentab{Both these kernels are examples with a mollification-type scaling that $\hat{K}_\ell^\epsilon(w) = \frac{1}{\epsilon^d}\hat{K}_\ell^1(\frac{w}{\epsilon})$, with $||\hat{K}_\ell^1(w)||_{L^1} = k_\ell$ and a finite second moment (i.e. $||\hat{K}_\ell^1(w)|w|^2||_{L^1} < \infty$). We then have that $\int_{\R^d} \hat{K}^\epsilon_\ell(w)  |w|^2 \, dw = o(\epsilon^2)$ as $\epsilon \to 0$, and $\hat{K}_\ell^\epsilon(w)\to k_\ell\delta(w)$ in distribution as $\epsilon\to 0$, motivating our choices in~\cref{Assume:kernelTwo}.}
\end{remark}

\begin{assumption}\label{Assume:kernelOne}
We assume that the unimolecular reaction rate function $K_\ell(x)$, $x\in\R^d$, is a constant i.e. $K_\ell(x) = k_\ell$. \commentab{ We note that there is no $\epsilon$ dependence in this case.}
\end{assumption}

\begin{assumption}\label{Assume:measureP}
We assume that for any  $1\leq \ell \leq L$, $\vec{y}\in \mathbb{Y}^{(\ell)}$ and $\vec{x}\in \mathbb{X}^{(\ell)}$,  the placement measure  $m_{\ell}(\vec{y} \, | \, \vec{x})$ is a probability measure, i.e. $\int_{\mathbb{Y}^{(\ell)}} m_{\ell}(\vec{y} \, | \, \vec{x})\, d\vec{y}= 1$. Notice that in the case of a reaction with no products, $\mathbb{Y}^{(\ell)} = \emptyset$, without loss of generality, let us assume that $\int_{\mathbb{Y}^{(\ell)}} m_{\ell}(\vec{y} \, | \, \vec{x})\, d\vec{y}= 1$ still holds.
\end{assumption}

If $\mathscr{R}_{\ell}$ is a zeroth order reaction (birth reaction) of the form $\emptyset \rightarrow S_i$, the reaction rate function is typically assumed to be a constant, $K_\ell = k_\ell$. In~\cite{IMS:2020} (Remark 5.7), we explained how the mean-field large-population limit holds if we assume that the placement measure $m(y), y\in\R^d$ for such a birth reaction has compact support. In biological applications such reactions typically occur within a compact region of $\R^3$, for example the interior of a cell, and as such $K_\ell$ and the placement measure should be zero outside of the region of interest. To avoid having a spatially varying birth rate, in the remainder we exclude zero'th order birth processes. We note, however, that this choice is made to simplify notation and subsequent calculations; nothing in our analysis fundamentally precludes the incorporation of zeroth order reactions.

\begin{assumption}\label{Assume:measureOne2One}
If $\mathscr{R}_{\ell}$ is a first order reaction of the form $S_i \rightarrow S_j$, we assume that the placement measure  $m_{\ell}(y\,|\, x)$ takes the form
$$m_{\ell}(y\,|\, x) = \delta_x(y).$$
This describes that the newly created $S_j$ particle is placed  at the position of the reactant $S_i$ particle.
\end{assumption}

\begin{assumption}\label{Assume:measureTwo2One}
If $\mathscr{R}_{\ell}$ is a second order reaction of the form $S_i + S_k \rightarrow S_j$, we assume that the placement measure  $m_{\ell}(z\vert x,y)$ takes the form
 $$m_{\ell}(z\vert x,y) = \sum_{i=1}^{I}p_i \times \delta\left(z-(\alpha_i x +(1-\alpha_i)y)\right),$$
  where $I$ is a fixed finite integer and $\sum_i p_i = 1$. This describes that the creation of particle $S_j$ is always on the segment connecting the reactant $S_i$ and reactant $S_k$ particles, but allows some random choice of position. A special case would be $I = 2$, $p_i = \tfrac{1}{2}$, $\alpha_1 = 0$ and $\alpha_2 = 1$, which corresponds to placing the particle randomly at the position of one of the two reactants. One common choice is taking $I = 1$, $p_1 = 1$ and choosing $\alpha_1$ to be the diffusion weighted center of mass~\cite{IZ:2018}.
\end{assumption}

\begin{assumption}\label{Assume:measureTwo2Two}
If $\mathscr{R}_{\ell}$ is a second order reaction of the form $S_i + S_k \rightarrow S_j + S_r$, we assume that the placement measure  $m_\ell(z, w \, | \, x, y)$ takes the form
$$m_\ell(z, w \, | \, x, y) = p\times\delta_{(x, y)}\left((z, w)\right)  + (1-p)\times\delta_{(x, y)}\left((w, z)\right).$$
This describes that newly created product $S_j$ and $S_r$ particles are always at the positions of the reactant $S_i$ and $S_k$ particles. $p$ is typically either $0$ or $1$, depending on the underlying physics of the reaction.
\end{assumption}

\begin{assumption}\label{Assume:measureOne2Two}
If $\mathscr{R}_{\ell}$ is a first order reaction of the form $S_i \rightarrow S_j + S_k$, we assume the placement measure depends on the separation scale parameter $\epsilon$ and is in the form
$$m^\epsilon_{\ell}(x,y\,|\,z) = \rho^\epsilon(|x-y|) \sum_{i=1}^{I}p_i\times\delta\left(z-(\alpha_i x +(1-\alpha_i)y)\right),$$
with $\sum_i p_i = 1$. Here we assume the relative separation of the product $S_j$ and $S_k$ particles, $|x -y|$, is sampled from the probability density $\rho^\epsilon(|x-y|)$. Their (weighted) center of mass is sampled from the density encoded by the sum of $\delta$ functions. Such forms are common for detailed balance preserving reversible bimolecular reactions~\cite{IZ:2018}, from which $\rho^{\epsilon}$ obtains the explicit $\epsilon$ dependence.
\end{assumption}

We further assume some regularity of the separation placement density, $\rho^\epsilon(|w|)$, $w\in\R^d$, introduced in \cref{Assume:measureOne2Two}:

\begin{assumption}\label{Assume:measureMrho}
For \cref{Assume:measureP} to be true, we require $||\rho^\epsilon||_{L^1(\R^d)} = 1$ for all $\epsilon$. When $\rho^\epsilon(|w|)$ comes from a reversible bimolecular reaction that satisfies detailed-balance, it will have a similar functional form to the bimolecular reaction kernel $\hat{K}^\epsilon_\ell(w)$ in \cref{Assume:kernelTwo}. For this reason, as $\epsilon\to 0$ we assume that $\rho^\epsilon(|w|)\to \delta(w)$ in distribution and $\int_{\R^d} \rho^\epsilon(|w|)|w|^2\, dw = O(\epsilon^2)$.
\end{assumption}

Note, with the preceding assumptions the placement measure only depends on $\epsilon$  for dissociation reactions of the form $S_{i} \to S_{j} + S_{k}$.

In what follows we rewrite the \textbf{MFM} to make explicit the $\epsilon$-dependence giving
\begin{align}
\partial_t\rho^\epsilon_j(x, t)
&
= D_{j} \lap_x \rho^\epsilon_j(x, t) - \sum_{\ell = 1}^L \paren{
\frac{1}{\vec{\alpha}^{(\ell)}!}  \sum_{r = 1}^{\alpha_{\ell j}} \int_{\tilde{\vec{x}} \in \mathbb{X}^{(\ell)}}   \delta_{x}(\tilde{x}_r^{(j)})  K^\epsilon_\ell(\tilde{\vec{x}}) \, \left( \Pi_{k = 1}^{J} \Pi_{s = 1}^{\alpha_{\ell k}}  \rho^\epsilon_{k}(\tilde{x}_{s}^{(k)}, t)\right) \, d\tilde{\vec{x}}
}  \nonumber\\
&
\quad+\sum_{\ell = 1}^L \paren{  \frac{1}{\vec{\alpha}^{(\ell)}!} \sum_{r = 1}^{\beta_{\ell j}}  \int_{\tilde{\vec{x}} \in\mathbb{X}^{(\ell)}}  K^\epsilon_\ell(\tilde{\vec{x}}) \left( \int_{\vy \in \mathbb{Y}^{(\ell)}}   \delta_{x}(y_r^{(j)}) m^\epsilon_\ell(\vec{y}\, | \,\tilde{\vec{x}}) \,d \vec{y} \right) \left( \Pi_{k = 1}^J \Pi_{s = 1}^{\alpha_{\ell k}}  \rho^\epsilon_{k}(\tilde{x}_s^{(k)}, t)\right) \, d\tilde{\vec{x}}
}
\label{Eq:density_formula_PIDEs_epsilon}
\end{align}
for $j = 1,\dots,J$.


\section{Main Results}\label{S:MainThm}
With the preceding assumptions, we now prove the rigorous relationship between the \textbf{MFM} and \textbf{SM} models as the reactive interaction scale $\epsilon\to 0$. Our main theoretical result is given by \cref{thm:convTwoModels} on the approximation of the \textbf{MFM}, \cref{Eq:density_formula_PIDEs_epsilon}, by the \textbf{SM}, \cref{Eq:density_formula_PDEs} as $\epsilon\rightarrow 0$. In addition, we present a series of numerical studies in one and two dimensions to demonstrate the relationship between the two models and the underlying PBSRD model.

Before proceeding with \cref{thm:convTwoModels}, let us briefly discuss the issue of well-posedness of models  \cref{Eq:density_formula_PIDEs} (equivalently \cref{Eq:density_formula_PIDEs_epsilon}) and \cref{Eq:density_formula_PDEs}. The mild solution to both models is of the form
\begin{equation}\label{eq:mildsoluPDE}
\vecrho(t) = S(t)\vecrho(0) + \int_0^t S(t-s)N[\vecrho](s)\, ds,
\end{equation}
where $\vecrho(t) = (\rho_1(\cdot, t), \rho_2(\cdot, t), \cdots,  \rho_J(\cdot, t))^T$,  $S(t)$ is the semigroup generated by the linear diffusion operator Diag($D_1\lap_x, D_2\lap_x, \cdots, D_J\lap_x$), and $N[\vecrho] = \paren{ \paren{N[\vecrho]}_1, \paren{N[\vecrho]}_2, \cdots, \paren{N[\vecrho]}_J }^T$ represents the nonlinear reaction term.

Depending on the properties of the nonlinear reaction term $N[\vecrho]$ one obtains local in time well-posedness, i.e. in some interval $[0,T_0]$, or one obtains global in time well-posedness. In \cref{S:Local} we discuss local in time well posedness and regularity for both equations \cref{Eq:density_formula_PIDEs} and \cref{Eq:density_formula_PDEs}. In \cref{S:Global}, we discuss under which additional assumptions on $N[\vecrho]$ one has global well-posedness. As illustrative examples, in \cref{S:Global} we prove that reaction systems with reactions of the type $A+B \leftrightarrows C+D$ and $A+B \leftrightarrows C$ both satisfy the requirements for global well-posedness (the latter under specific choices for the placement measure).

\subsection{Approximation theorem}\label{S:ApproximationThm}

Let $C_{b, unif}(\R^d)$ denote the space of bounded and uniformly continuous functions on $\R^d$, and denote by $C_b^{\ell}(\R^d)$ the space of functions with continuous and uniformly bounded derivatives on $\R^d$ through order $\ell$. Our main result is the following theorem on the convergence of the \textrm{MFM} to the \textrm{SM} as $\epsilon \to 0$.

\begin{theorem}\label{thm:convTwoModels}
Let $T_0$ be a time such that the solutions to the \textbf{SM} and \textbf{MFM} are uniformly bounded for $(x, t) \in \R^{d} \times \brac{0,T_{0}}$ and $\epsilon$. Let $\rho_j(x, 0) = \rho_j^\epsilon(x, 0)\in C^1_b(\R^d)\cap C_{b, unif}(\R^d)$ for all $\epsilon > 0$ and $j = 1, \cdots, J $, and assume that
\begin{align*}
k_{\ell} = \paren{\vec{\alpha}^{(\ell)}!} \kappa_{\ell}.
\end{align*}
Under the assumptions in  \cref{Ass:MainAssumptions} we have that the solution to  \cref{Eq:density_formula_PIDEs_epsilon} converges uniformly in $(x,t)$ to the solution to \cref{Eq:density_formula_PDEs} at second order as $\epsilon \to 0$, i.e.
\begin{equation*}
  \sup_{t\in [0, T_0]}\max_{j = 1, \cdots, J}||\rho_j(x, t) - \rho_j^\epsilon(x, t)  ||_{L^\infty} = O(\epsilon^2).
\end{equation*}

\end{theorem}

\begin{remark}
One can find such a $T_0$ in \cref{thm:convTwoModels} by a contraction mapping approach as in \cref{S:Local}, or choose any $T_0 < \infty$ for particular reaction systems for which global well-posedness holds with appropriate uniform estimates. Systems of the form $A+B \leftrightarrows C+D$ or $A+B \leftrightarrows C$ are shown to have such global well-posedness estimates in \cref{S:Global}. More generally, whether $T_{0}$ can be chosen arbitrarily large will depend on the global well-posedness of the \textbf{MFM} and the \textbf{SM} for a particular reaction network. As we describe in \cref{S:Global}, the construction of such a global well-posedness theory for general reaction systems is a still an open problem.
\end{remark}

\begin{proof}
By taking the difference of the mild solution to  \cref{Eq:density_formula_PIDEs_epsilon} and \cref{Eq:density_formula_PDEs} for any species $j$, $1\leq j\leq J$,  we have that
\begin{align}
\rho_j(x, t) - \rho_j^\epsilon(x, t)
&
=  e^{tD_j\lap_x } \left( \rho_j(x, 0) - \rho_j^\epsilon(x, 0) \right) -  \sum_{\ell = 1}^L\int_0^t e^{(t-\tau)D_j\lap_x }\sum_{r = 1}^{\alpha_{\ell j}}\bigg[ \kappa_{\ell}\times\Pi_{k = 1}^{J} \Pi_{s = 1}^{\alpha_{\ell k}}  \rho_{k}(x, \tau)\nonumber\\
&
\qquad - \paren{\frac{1}{\vec{\alpha}^{(\ell)}!}   \int_{\tilde{\vec{x}} \in \mathbb{X}^{(\ell)}}   \delta_{x}(\tilde{x}_r^{(j)})  K^\epsilon_\ell(\tilde{\vec{x}}) \, \left( \Pi_{k = 1}^{J} \Pi_{s = 1}^{\alpha_{\ell k}}  \rho^\epsilon_{k}(\tilde{x}_{s}^{(k)}, \tau)\right) \, d\tilde{\vec{x}} } \bigg]\nonumber\\
&
\quad +  \sum_{\ell = 1}^L \int_0^t e^{(t-\tau)D_j\lap_x }\sum_{r = 1}^{\beta_{\ell j}}\bigg[
 \kappa_{\ell}\times\Pi_{k = 1}^{J} \Pi_{s = 1}^{\alpha_{\ell k}}  \rho_{k}(x, \tau)\nonumber\\
&
\qquad - \paren{  \frac{1}{\vec{\alpha}^{(\ell)}!}   \int_{\tilde{\vec{x}} \in\mathbb{X}^{(\ell)}}  K^\epsilon_\ell(\tilde{\vec{x}}) \left( \int_{\vy \in \mathbb{Y}^{(\ell)}}   \delta_{x}(y_r^{(j)}) m^\epsilon_\ell(\vec{y}\, | \,\tilde{\vec{x}}) \,d \vec{y} \right) \left( \Pi_{k = 1}^J \Pi_{s = 1}^{\alpha_{\ell k}}  \rho^\epsilon_{k}(\tilde{x}_s^{(k)}, \tau)\right) \, d\tilde{\vec{x}}} \bigg]\, d\tau.
\end{align}
Using that $||e^{tD_j\lap_x} f||_{L^\infty}\leq ||f||_{L^\infty}$ for all $t\geq 0$ and $f\in L^\infty(\R^d)$, we have that
\begin{equation}\label{eq:L_infty_bound1_init}
||\rho_j(x, t) - \rho_j^\epsilon(x, t)  ||_{L^\infty} \leq  || \rho_j(x, 0) - \rho_j^\epsilon(x, 0) ||_{L^\infty}+  \sum_{\ell = 1}^L  \int_0^t  \Lambda^{\ell, j}(\tau) + \Theta^{\ell, j}(\tau) d\tau,
\end{equation}
where
\begin{equation}
\Lambda^{\ell, j}(\tau)  = \sum_{r = 1}^{\alpha_{\ell j}} \norm{  \paren{\frac{1}{\vec{\alpha}^{(\ell)}!}   \int_{\tilde{\vec{x}} \in \mathbb{X}^{(\ell)}}   \delta_{x}(\tilde{x}_r^{(j)})  K^\epsilon_\ell(\tilde{\vec{x}}) \, \left( \Pi_{k = 1}^{J} \Pi_{s = 1}^{\alpha_{\ell k}}  \rho^\epsilon_{k}(\tilde{x}_{s}^{(k)}, \tau)\right) \, d\tilde{\vec{x}} } -\kappa_{\ell}\times\Pi_{k = 1}^{J} \Pi_{s = 1}^{\alpha_{\ell k}}  \rho_{k}(x, \tau) }_{L^\infty},
\end{equation}
and
\begin{multline}
\Theta^{\ell, j}(\tau) = \sum_{r = 1}^{\beta_{\ell j}} \bigg|\bigg|  \paren{  \frac{1}{\vec{\alpha}^{(\ell)}!}   \int_{\tilde{\vec{x}} \in\mathbb{X}^{(\ell)}}  K^\epsilon_\ell(\tilde{\vec{x}}) \left( \int_{\vy \in \mathbb{Y}^{(\ell)}}   \delta_{x}(y_r^{(j)}) m^\epsilon_\ell(\vec{y}\, | \,\tilde{\vec{x}}) \,d \vec{y} \right) \left( \Pi_{k = 1}^J \Pi_{s = 1}^{\alpha_{\ell k}}  \rho^\epsilon_{k}(\tilde{x}_s^{(k)}, \tau)\right) \, d\tilde{\vec{x}}}\\
- \kappa_{\ell}\times\Pi_{k = 1}^{J} \Pi_{s = 1}^{\alpha_{\ell k}}  \rho_{k}(x, \tau)\bigg|\bigg|_{L^\infty}.
\end{multline}

From now on, and also in \cref{lem:estimates1} and \cref{lem:estimates2}, we will use the following notations for generic constants. Using the assumed uniform boundedness of the \textbf{MFM} and \textbf{SM} solutions, we have
$$C : = \paren{\max_{\ell = 1, \cdots, L} k_\ell}  \max_{j = 1, \cdots, J} \{\sup_{\epsilon > 0} \sup_{\tau\in[0, T_0]}||\rho^\epsilon_{j}(x, \tau)||_{L^\infty} \vee \sup_{\tau\in[0, T_0]} ||\rho_{j}(x, \tau)||_{L^\infty} \} < \infty.$$
\textbf{(Inequality 1)} in \cref{thm:regularityPIDE} and \cref{thm:regularityPDE} then give that
\begin{equation}\label{eq:firstDiffEst}
\max_{j = 1, \cdots, J} \{ \sup_{\epsilon > 0}\sup_{\tau\in[0, T_0]} ||\rho^\epsilon_{j}(x, \tau)||_{C^1_b(\R^d)} \vee \sup_{\tau\in[0, T_0]} ||\rho_{j}(x, \tau)||_{C^1_b(\R^d)}\} \leq C_1,
\end{equation}
for some constant $C_1$ depending only on $C$, $\sup_{j = 1, \cdots, J}||\rho_j(x, 0)||_{C^1_b(\R^d)}$ and $T_0$.
Similarly, \textbf{(Inequality 2)} in \cref{thm:regularityPIDE} and  \cref{thm:regularityPDE} give
\begin{equation}
\max_{j = 1, \cdots, J} \{ \sup_{\epsilon > 0}||\rho^\epsilon_{j}(x, \tau)||_{C^2_b(\R^d)} \vee ||\rho_{j}(x, \tau)||_{C^2_b(\R^d)}\} \leq C_2 + \frac{C_3}{\sqrt{\tau}}.
\end{equation}
Combined with \cref{eq:firstDiffEst}, we have that
\begin{equation}
\max_{j, k = 1, \cdots, J} \{ \sup_{\epsilon > 0}||\rho^\epsilon_{j}(x, \tau)\rho^\epsilon_{k}(x, \tau)||_{C^2_b(\R^d)} \vee ||\rho_{j}(x, \tau)\rho_{k}(x, \tau)||_{C^2_b(\R^d)}\} \leq C_2 + \frac{C_3}{\sqrt{\tau}},
\end{equation}
for any fixed $\tau\in(0, T_0]$ and some (other) constants $C_2$ and $C_3$ only depending on $C$, $\sup_{j = 1, \cdots, J}||\rho_j(x, 0)||_{C^1_b(\R^d)}$ and $T_0$.

Using the estimates from \cref{lem:estimates1} and \cref{lem:estimates2}, \cref{eq:L_infty_bound1_init} becomes
\begin{align}\label{eq:L_infty_bound}
||\rho_j(x, t) - \rho_j^\epsilon(x, t)  ||_{L^\infty} & \leq  || \rho_j(x, 0) - \rho_j^\epsilon(x, 0) ||_{L^\infty}\nonumber\\
&
\qquad+  \sum_{\ell = 1}^L  \int_0^t  6C\paren{\max_{j = 1, \cdots, J}||\rho_j(x, \tau) - \rho_j^\epsilon(x, \tau) ||_{L^\infty} + \paren{C_2 + \frac{C_3}{\sqrt{\tau}}}O(\epsilon^2)} d\tau, \nonumber\\
&
 \leq  || \rho_j(x, 0) - \rho_j^\epsilon(x, 0) ||_{L^\infty} + 6CL \paren{C_2t + 2C_3\sqrt{t}}O(\epsilon^2)\nonumber\\
&
\qquad +   6CL\int_0^t  \paren{\max_{j = 1, \cdots, J}||\rho_j(x, \tau) - \rho_j^\epsilon(x, \tau) ||_{L^\infty}} d\tau.
\end{align}
Applying Gronwall's Lemma we have
$$\max_{j = 1, \cdots, J}||\rho_j(x, t) - \rho_j^\epsilon(x, t)  ||_{L^\infty}\leq \left( 6CL \paren{C_2t + 2C_3\sqrt{t}}O(\epsilon^2)+ \max_{j =1, \cdots, J}|| \rho_j(x, 0) - \rho_j^\epsilon(x, 0) ||_{L^\infty}\right) e^{6CLt}$$
so that
$$\sup_{t\in [0, T_0]}\max_{j = 1, \cdots, J}||\rho_j(x, t) - \rho_j^\epsilon(x, t)  ||_{L^\infty}\leq  6CL \paren{C_2T_0 + 2C_3\sqrt{T_0}}e^{6CLT_0}O(\epsilon^2) .$$

\end{proof}


\subsection{Numerical Comparison for Reversible $A+B\rightleftarrows C$ Reaction}\label{SS:Numerics}

In order to illustrate \cref{thm:convTwoModels} and further investigate the connections between the \textbf{SM} and  the \textbf{MFM}, we numerically solved the reversible $A+B\rightleftarrows C$ reaction using each of the PBSRD model, the \textbf{SM}, and the \textbf{MFM}. The PDEs and PIDEs for the \textbf{SM} and \textbf{MFM} were solved in MATLAB with periodic boundary conditions on both the interval $[0, L]$, and the square $[0, L] \times [0, L]$. The stochastic process associated with the PBSRD model was numerically solved by discretization to a jump process via the Convergent Reaction Diffusion Master Equation (CRDME) \cite{I:2013, IZ:2018}, which was then sampled using the Gibson-Bruck stochastic simulation algorithm~\cite{GB:2000}. Due to the computational expense of the PBSRD model in the large-population limit, we only solved it for the one-dimensional problem with periodic boundary conditions on $[0, L]$.

Let us denote by $A(x, t), B(y, t), C(z, t)$ the concentration fields for species A, B, C respectively in the \textbf{SM}, \cref{Eq:density_formula_reversible_PDE}.

Similarly, we denote by $A_\epsilon(x, t), B_\epsilon(y, t), C_\epsilon(z, t)$ the concentration fields for species A, B, C respectively in the \textbf{MFM} \cref{Eq:density_formula_PIDEs_epsilon}, with separation scale parameter $\epsilon$. The latter then satisfy~\cref{Eq:density_formula_reversible_PIDE}.

In the following we fixed $L = 2\pi$, the diffusivity $D_1 = 1, D_2 = 0.5, D_3 = 0.1$, and assumed a detailed balanced condition on the reversible reactions, i.e.
\begin{equation} \label{eq:dbcondit}
K_d \hat{K}^\epsilon_1(x-y)m_1(z | x, y)  = K_2(z)m_{2}^\epsilon(x,y|z),
\end{equation}
where $K_d = k_2/k_1$ is the equilibrium dissociation constant of the reaction, see~\cite{IZ:2018}. We set $\kappa_1 = k_1 = ||\hat{K}^\epsilon_1||_{L^1(\R^d)} = 1$ and $\kappa_2 = k_2 = K_2(z) = 0.05$ for any $z\in\R^d$. We consider two reaction kernels: the Doi kernel
$$\hat{K}^\epsilon_1(w) = \frac{k_1}{|B_\epsilon(0)|}\mathbbm{1}_{\{|w|\leq \epsilon\}},$$
where $|B_\epsilon(0)|$ is the volume of d-dimensional ball with radius $\epsilon$, and the Gaussian kernel
$$\hat{K}^\epsilon_1(w) = \frac{k_1}{\paren{\sqrt{2\pi\epsilon^2}}^d} e^{-\frac{|w|^2}{2\epsilon^2}},$$ for any $w\in\R^d$.

The PDEs and PIDEs were solved using a Fourier collocation method \cite{HGG:2007}  for the spatial discretization, with collocation points $x_i = \frac{i L}{N}, i = 0, \cdots, N-1$ in 1d and $(x_i, y_j ) = (\frac{i L}{N}, \frac{j L}{N}), i, j = 0, \cdots, N-1$ in 2d. We chose $N = 2^9$ for 1d and $N = 2^8$ for 2d.  We approximated integral terms in the \textbf{MFM} using the midpoint quadrature rule centered at collocation points. For the reaction-diffusion equations the diffusion terms were stiff whereas the reaction terms were non-stiff. We therefore used the Crank-Nicolson Adams-Bashforth (CNAB) implicit-explicit method (IMEX) to discretize in time the spatially discretized system, with a time step of $\Delta t = 1e-3$ (for stability and accuracy reasons). As CNAB is a two-step multistep method, to obtain the numerical solution at time $\Delta t$ we applied the one-step IMEX Forward-Backward Euler method with a time step of $\paren{\Delta t}^2$ until time $\Delta t$.

\FloatBarrier

\subsubsection{One Dimensional Results}
In the one dimensional periodic domain $[0, L]$, we set the initial conditions for both models to be $A(x, 0) = e^{-10(x-1)^2}$, $B(y, 0) = e^{-10(y-2)^2}$ and $C(z, 0) = 0$.

\begin{figure*}[tb]
  \centering
  \subfloat[]{
    \label{fig:mass1d-a}
    {\includegraphics[width=0.35\textwidth]{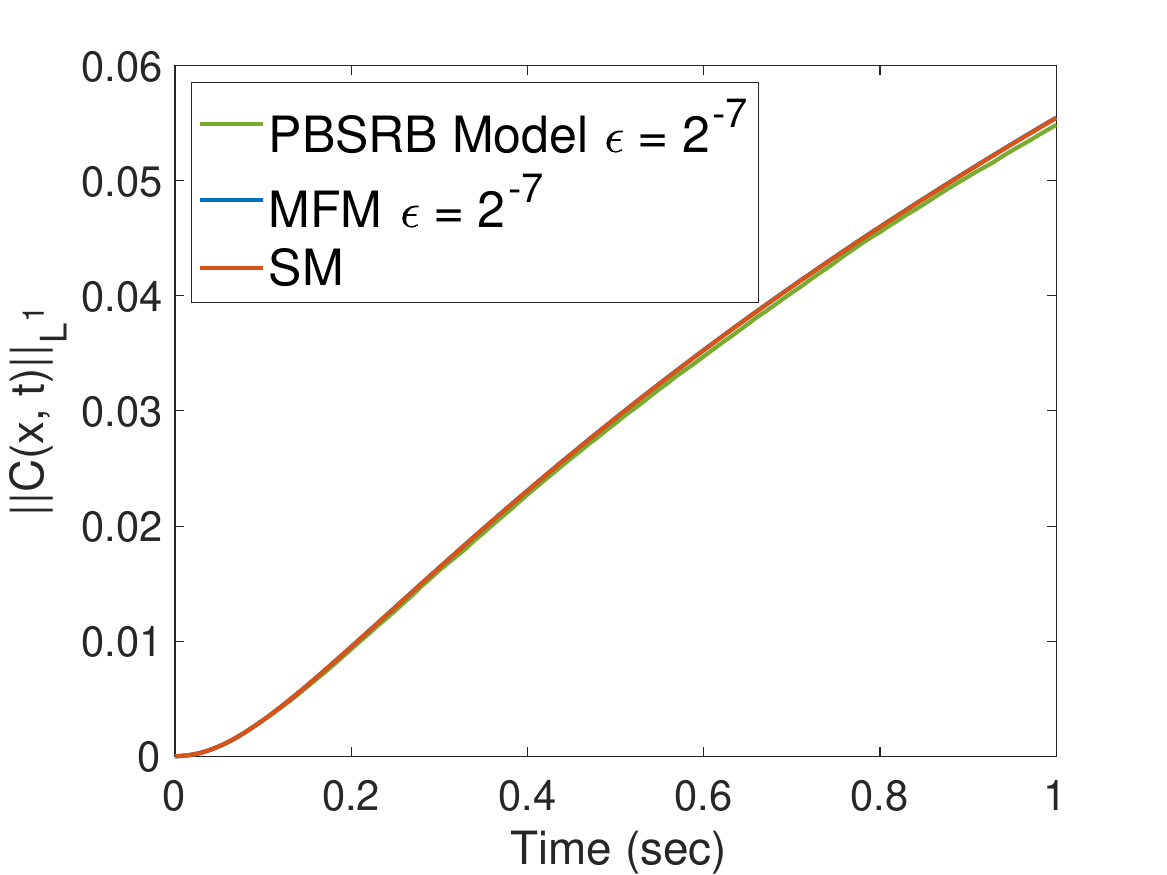}} }
  \subfloat[]{
    \label{fig:mass1d-b}
    {\includegraphics[width=0.35\textwidth]{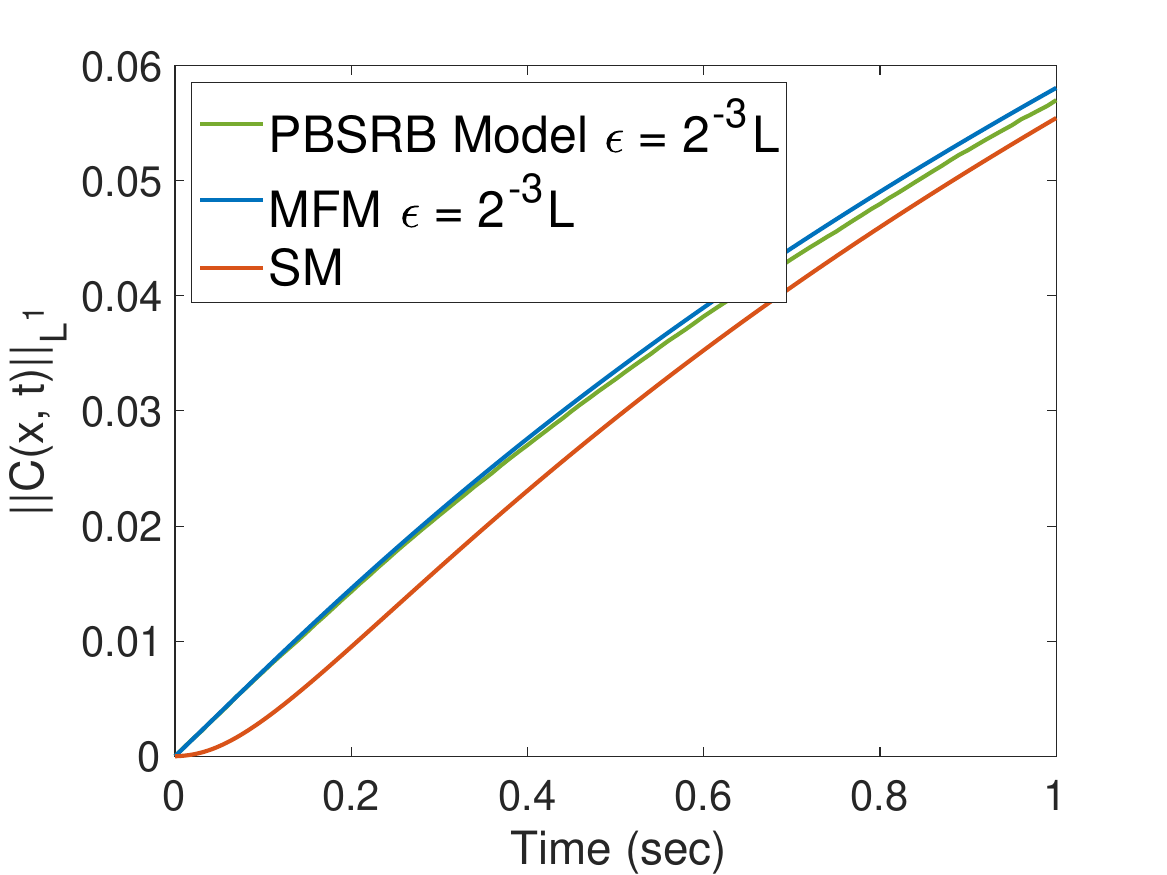}}}
  \caption{ Evolution of the Molar Mass in one dimension estimated by 100 simulations with $\gamma = 10^4$ and   \textbf{(a)}  reaction radius $\epsilon = 2^{-7}$ , \textbf{(b)}   reaction radius $\epsilon = 2^{-4}*L$.}
  \label{fig:mass1d}
\end{figure*}
\begin{figure*}[thb]
  \centering
  \subfloat[]{
    \label{fig:dist1d-a}
    {\includegraphics[width=0.35\textwidth]{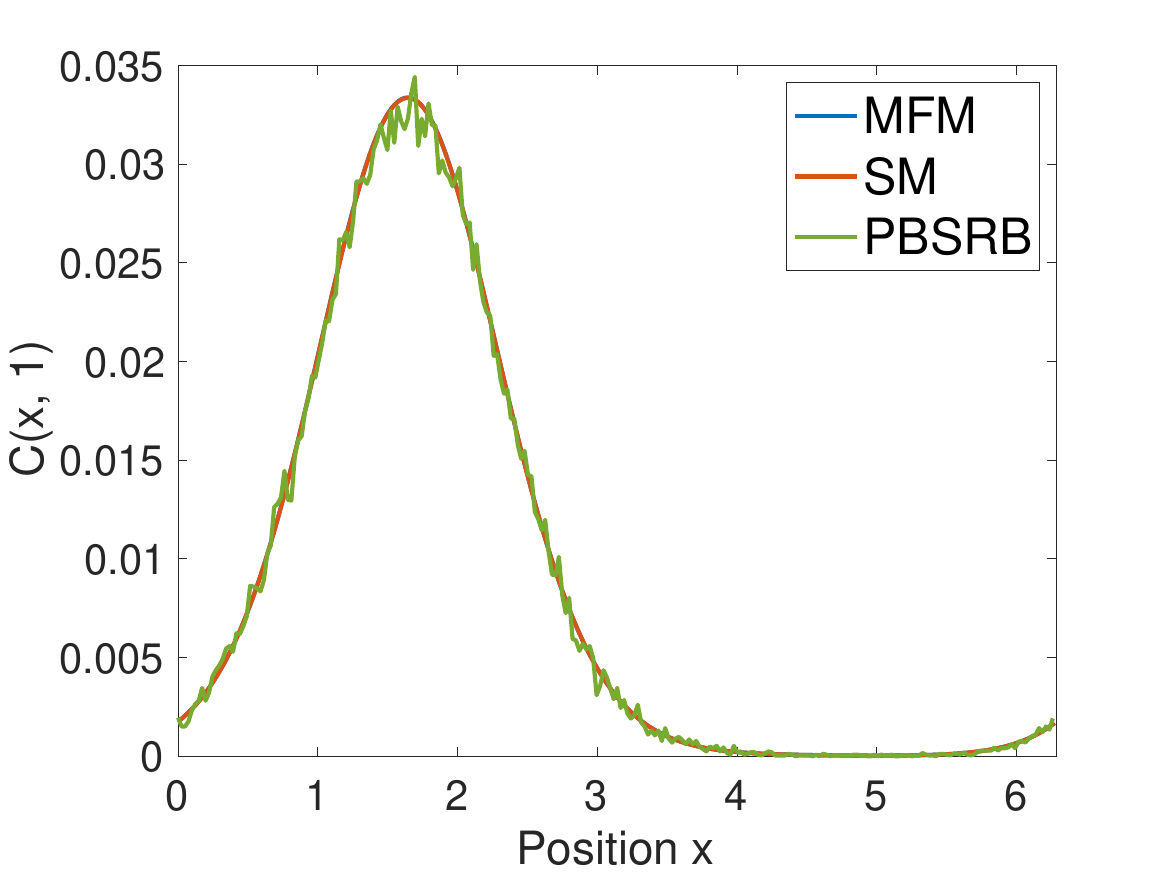}} }
  \subfloat[]{
    \label{fig:dist1d-b}
    {\includegraphics[width=0.35\textwidth]{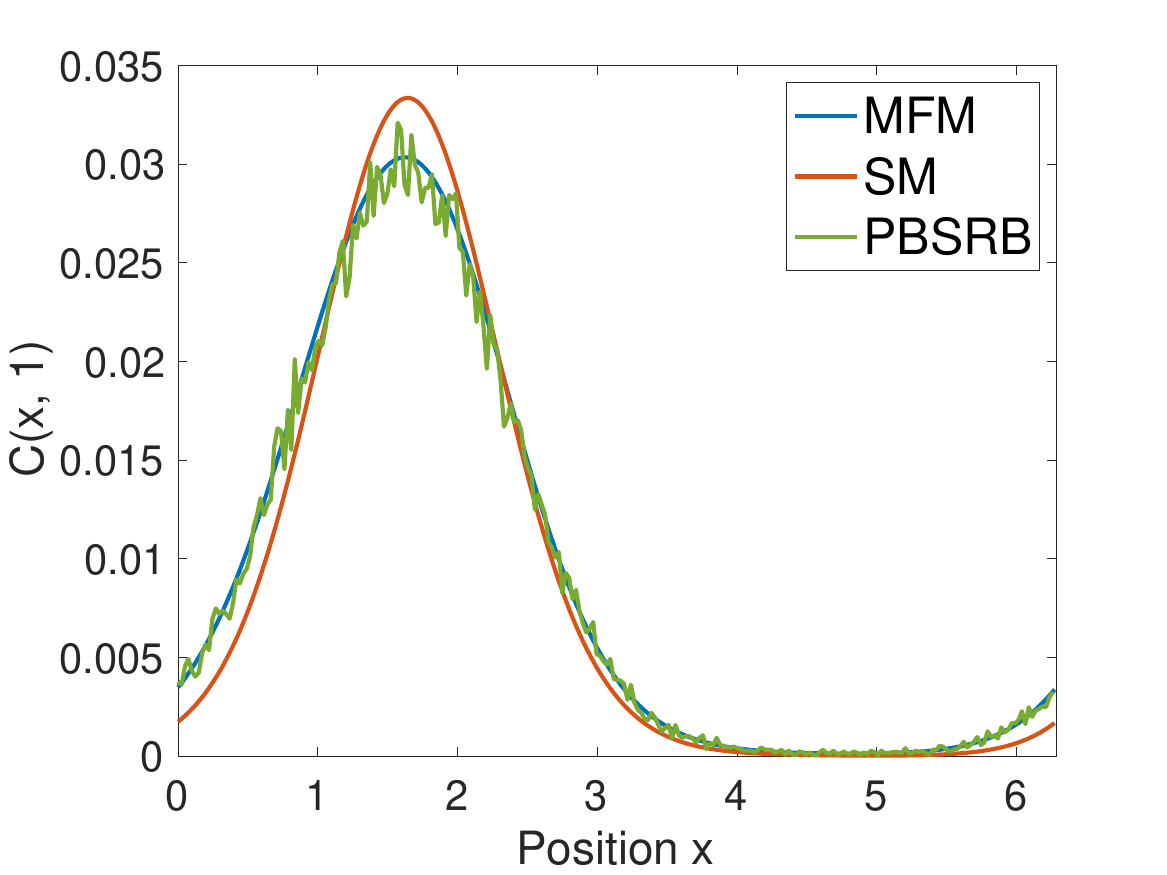}}}
  \caption{ Spatial Molar Concentration at Time 1 in one dimension estimated by 100 simulations \textbf{(a)} with reaction radius $\epsilon = 2^{-7}$  , \textbf{(b)}  with reaction radius $\epsilon = 2^{-4}*L$.}
  \label{fig:dist1d}
\end{figure*}
We first examine the relationship between the PBSRD model, the \textbf{MFM} and the \textbf{SM} as shown in \cref{fig:mass1d,fig:dist1d}. It is clear that the  molar concentration fields and molar masses in the \textbf{MFM} are a good approximation of the PBSRD model for $\gamma$, the large-population limit scaling parameter in the PBSRD model \cite{IMS:2020}, sufficiently large. For $\epsilon$ sufficiently small both the \textbf{SM}  and \textbf{MFM} are good approximations, while for $\epsilon$ sufficiently large, only the \textbf{MFM} is a good approximation to the PBSRD model.

\begin{figure*}[tb]
  \centering
  \subfloat[]{
    \label{fig:convRate1d-a}
    {\includegraphics[width=0.35\textwidth]{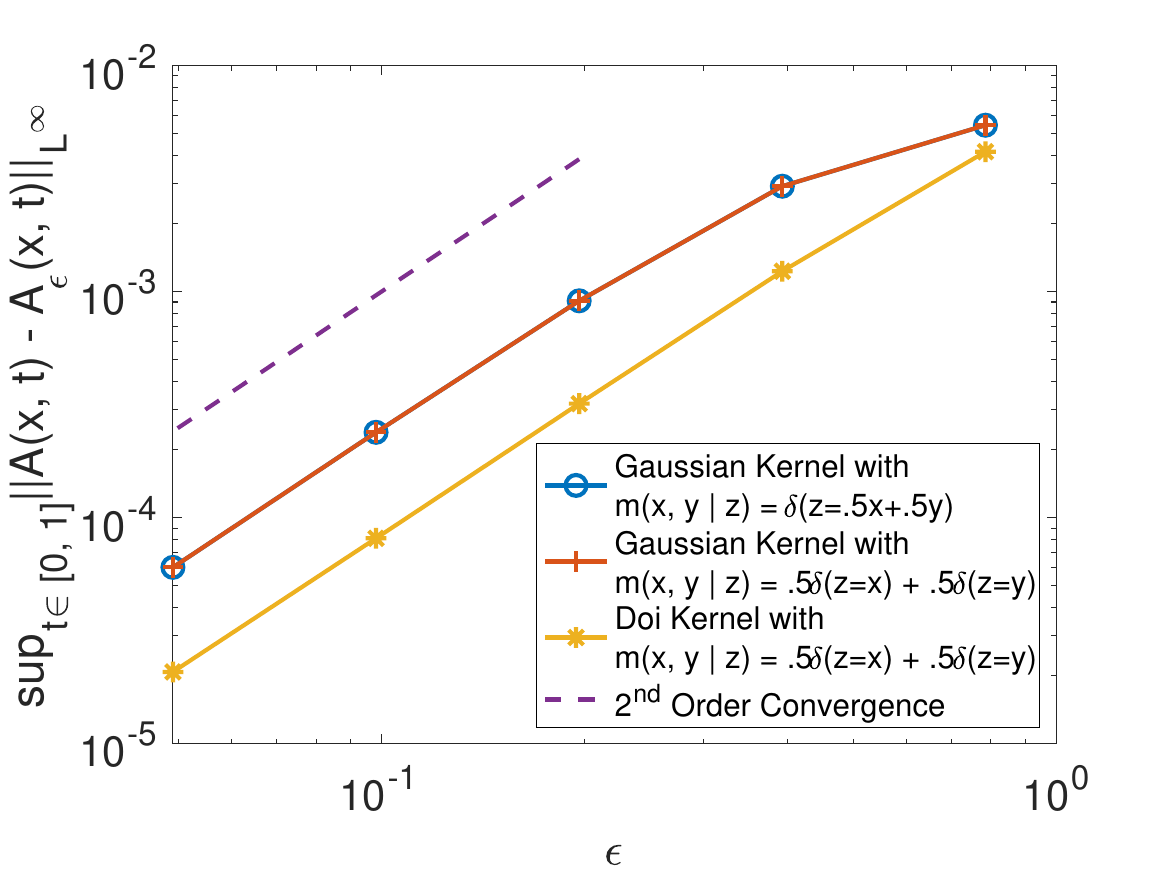}} }
  \subfloat[]{
    \label{fig:convRate1d-b}
    {\includegraphics[width=0.35\textwidth]{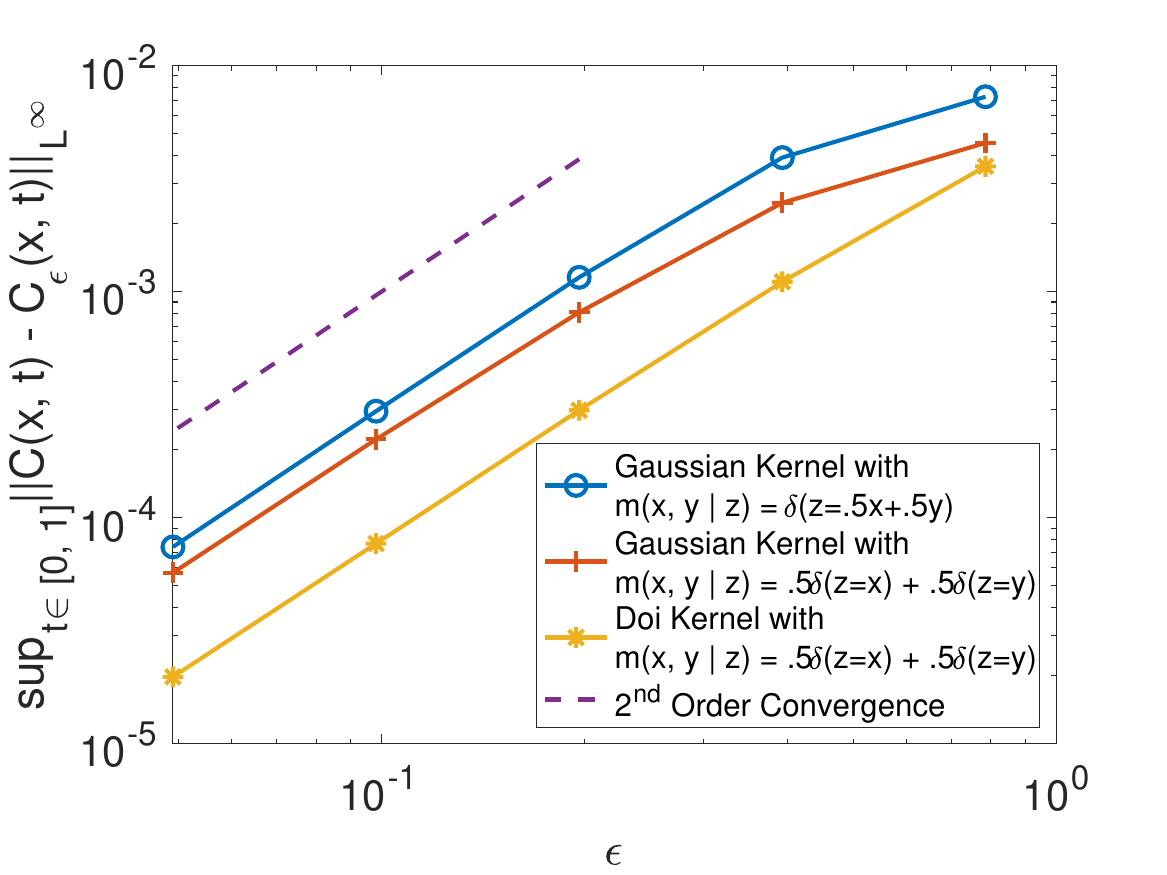}}}
  \caption{ One Dimensional Uniform Convergence of the spatial density in the time interval [0, 1] \textbf{(a)}  for species A , \textbf{(b)} for species C.}
  \label{fig:convRate1d}
\end{figure*}
We next compared the \textbf{SM} to the \textbf{MFM} with various combinations of reaction kernels and displacement measures. In particular, we investigated the \textbf{MFM} with (1) Gaussian kernel and placement measure  $m_1(z | x, y) = \delta(z - (.5x + .5y))$, (2) Gaussian kernel and placement measure $m_1(z | x, y) = .5 \delta(z - x) + .5 \delta(z - y)$,  (3) Doi kernel and placement measure $m_1(z | x, y) = .5 \delta(z - x) + .5 \delta(z - y)$. For all these choices of the reaction kernel and placement measure, the \textbf{MFM} PIDE solution converges to the \textbf{SM} PDE solution at second order in $\epsilon$ as shown in \cref{fig:convRate1d}. This illustrates our rigorous results on the convergence proven in  \cref{thm:convTwoModels}.

\subsubsection{Two Dimensional Results}
We further tested the convergence over a two dimensional periodic domain, comparing the \textbf{SM} PDE solution to the \textbf{MFM} PIDE solution for a representative Gaussian reaction kernel and placement measure $m_1(x, y | z) = .5 \delta(z - x) + .5 \delta(z - y)$. Here we fix the initial solutions for both models to be $A(x, 0) = e^{-12(x_1-1)^2-8(x_2-2)^2}$, $B(x, 0) = e^{-10(x_1-1)^2-5(x_2-2)^2}$ and $C(x, 0) = 0$, where $x = (x_1, x_2) \in \R^2$.

\begin{figure*}[tb]
  \centering
  \subfloat[]{
    \label{fig:2a}
    {\includegraphics[width=0.3\textwidth]{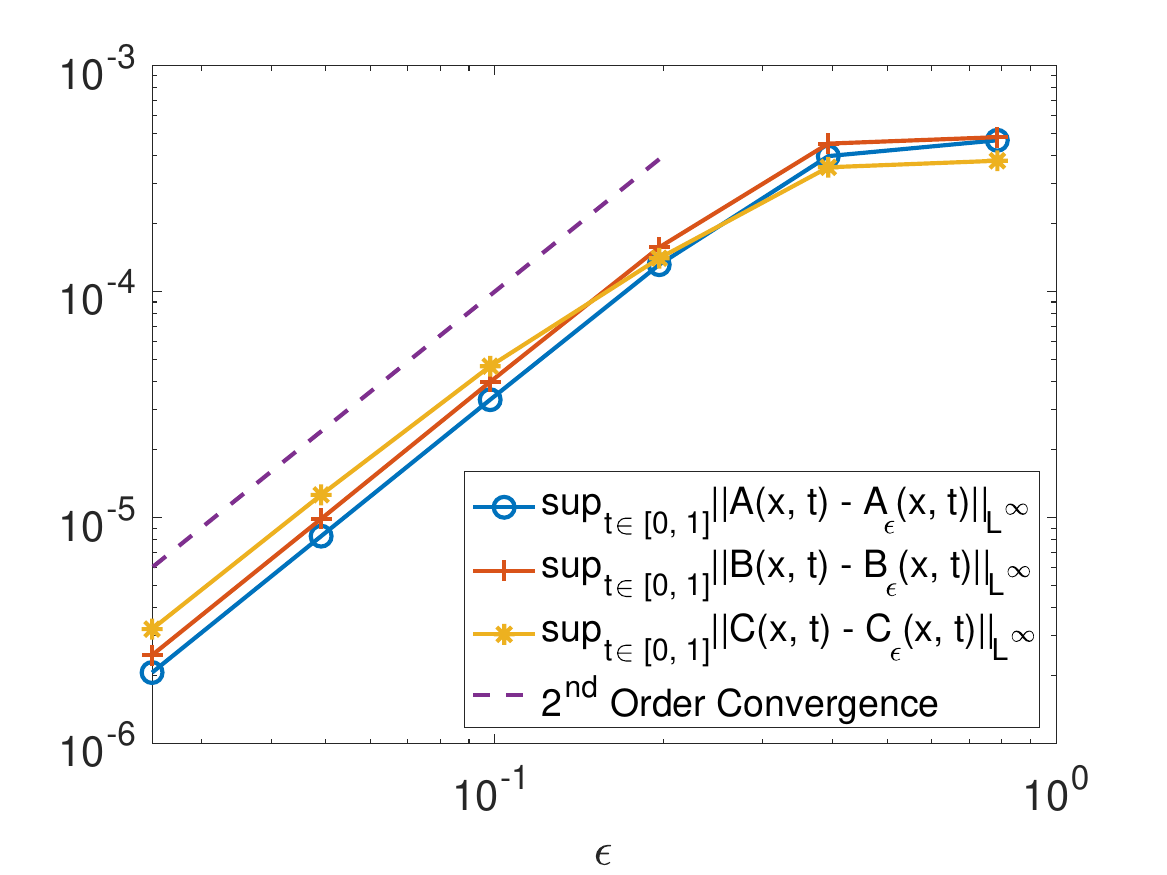}} }
  \subfloat[]{
    \label{fig:2b}
    {\includegraphics[width=0.3\textwidth]{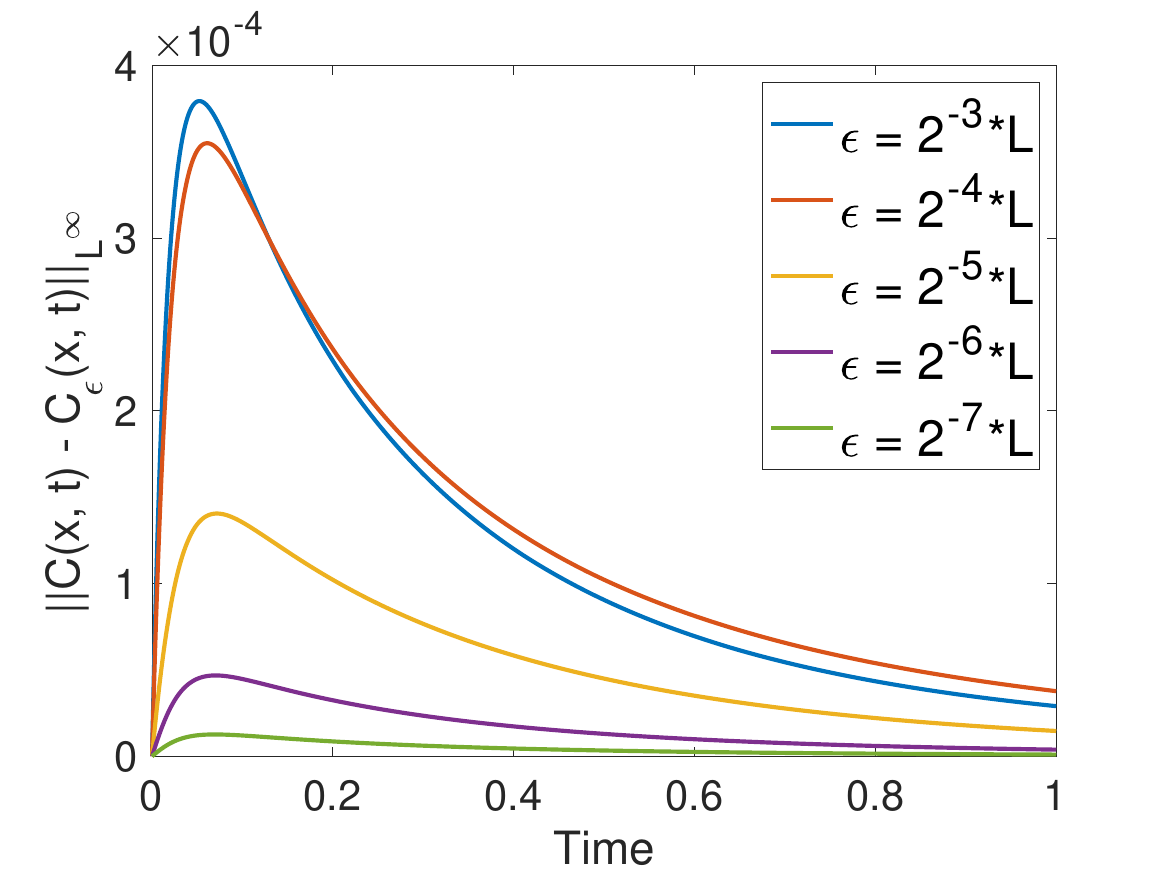}}}
  \subfloat[]{
    \label{fig:2c}
    {\includegraphics[width=0.3\textwidth]{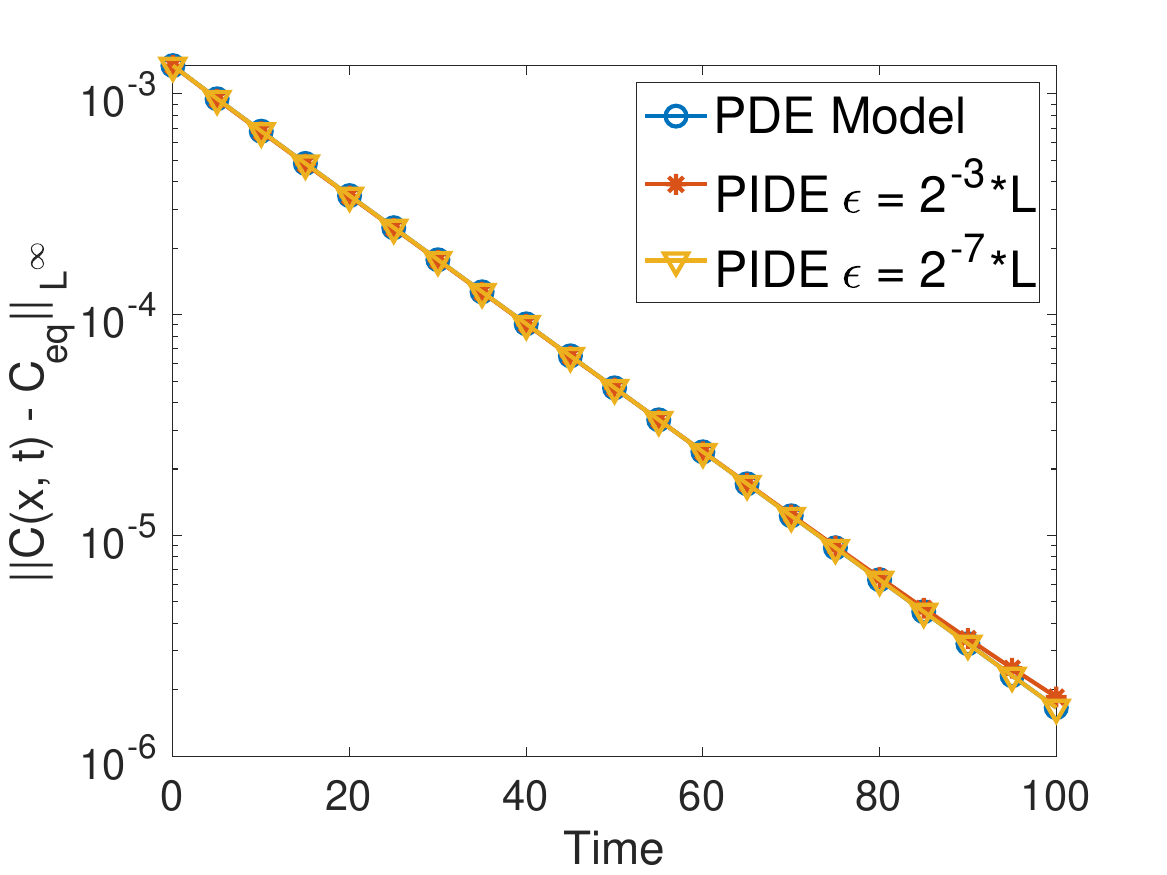}} }
  \caption{ Two Dimensional Reversible Reactions. \textbf{(a)} Convergence of the spatial density uniformly in space and time interval $[0, 1]$ for species A, B, and C. \textbf{(b)} The error between the PDE solution and the PIDE solution versus time for species C. \textbf{(c)} Convergence to the same constant equilibrium denoted as $C_{eq}$ for the PDE and MFM.}
  \label{fig:convRate2d}
\end{figure*}

Uniform second order convergence in space and time is again verified as shown in \cref{fig:2a}. The error between the two models versus time is illustrated in \cref{fig:2b}. We see that the maximum error over space increases to a maximum and then decreases as $t$ increases for each value of $\epsilon$. The smaller $\epsilon$ is, the smaller the error between the solutions of the \textbf{SM} and the \textbf{MFM}. The error decreases for large times because both models converge to the same spatially uniform equilibrium solution (exponentially fast) as $t \to \infty$, as shown in Figure \cref{fig:2c}. Here the equilibrium constant can be calculated from the corresponding mass action ODE model for the reaction using conservation laws for the reversible reaction. Let $A_{eq}, B_{eq}, C_{eq}$ denote the equilibrium concentration for species A, B and C respectively in the ODE model. They satisfy $\Kd C_{eq} = A_{eq}B_{eq}$. Let $A_0, B_0, C_0$ denote the averaged spatial density at $t=0$ in the \textbf{MFM} and \textbf{SM}, and assume these are the initial conditions used in the ODE model.  In the ODE model we have that  $A_{eq} + B_{eq} + 2C_{eq} = A_0 + B_0 + 2C_0 := sum$ and $B_{eq} - A_{eq} = B_0 - A_0 : = diff $. Solving these three equations we  obtain the equilibrium concentration $C_{eq} = \frac{1}{2}(sum + \Kd - \sqrt{(sum+ \Kd)^2 - (sum^2 - diff^2)})$.

\begin{figure*}[tb]
  \centering
  \subfloat[\textbf{SM} at Time 0.1]{
    \label{fig:3a}
    {\includegraphics[width=0.3\textwidth]{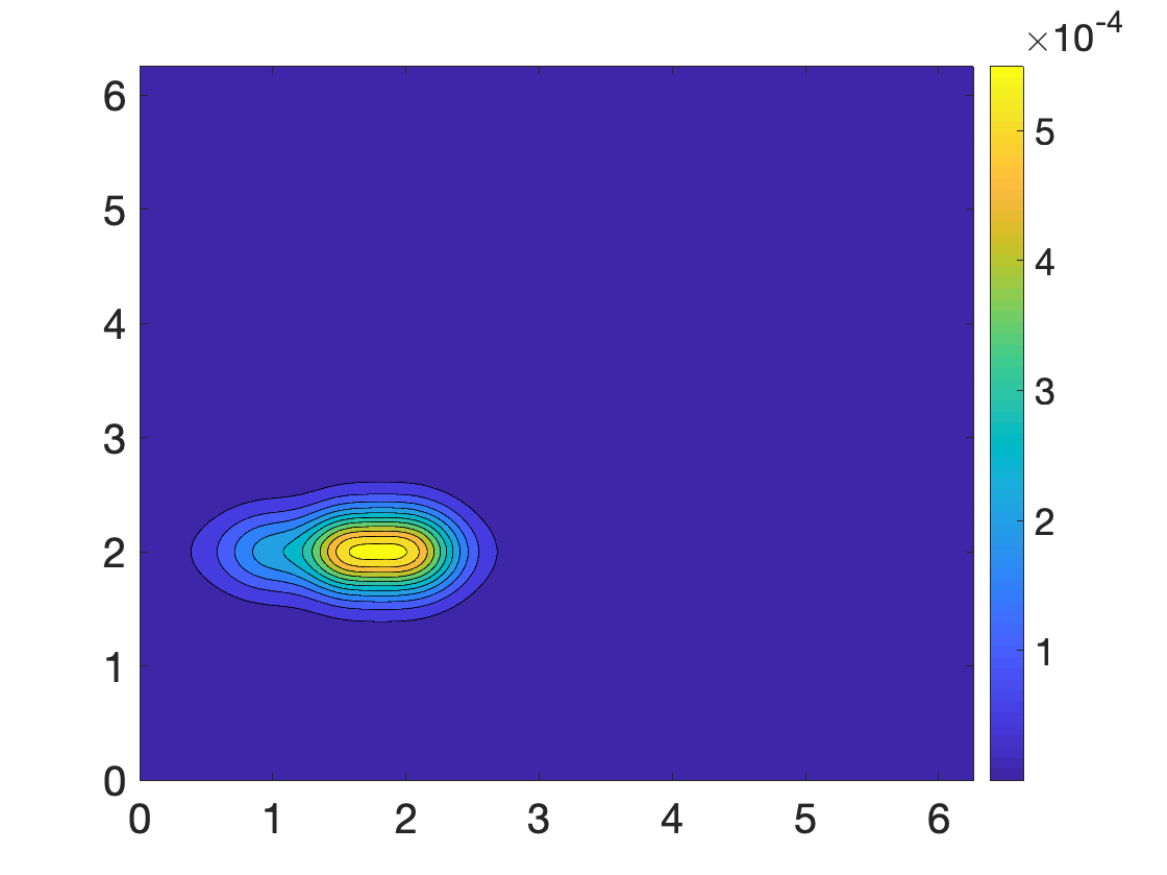}} }
  \subfloat[\textbf{MFM} with $\epsilon = 2^{-3}*L$ at Time 0.1]{
    \label{fig:3b}
    {\includegraphics[width=0.3\textwidth]{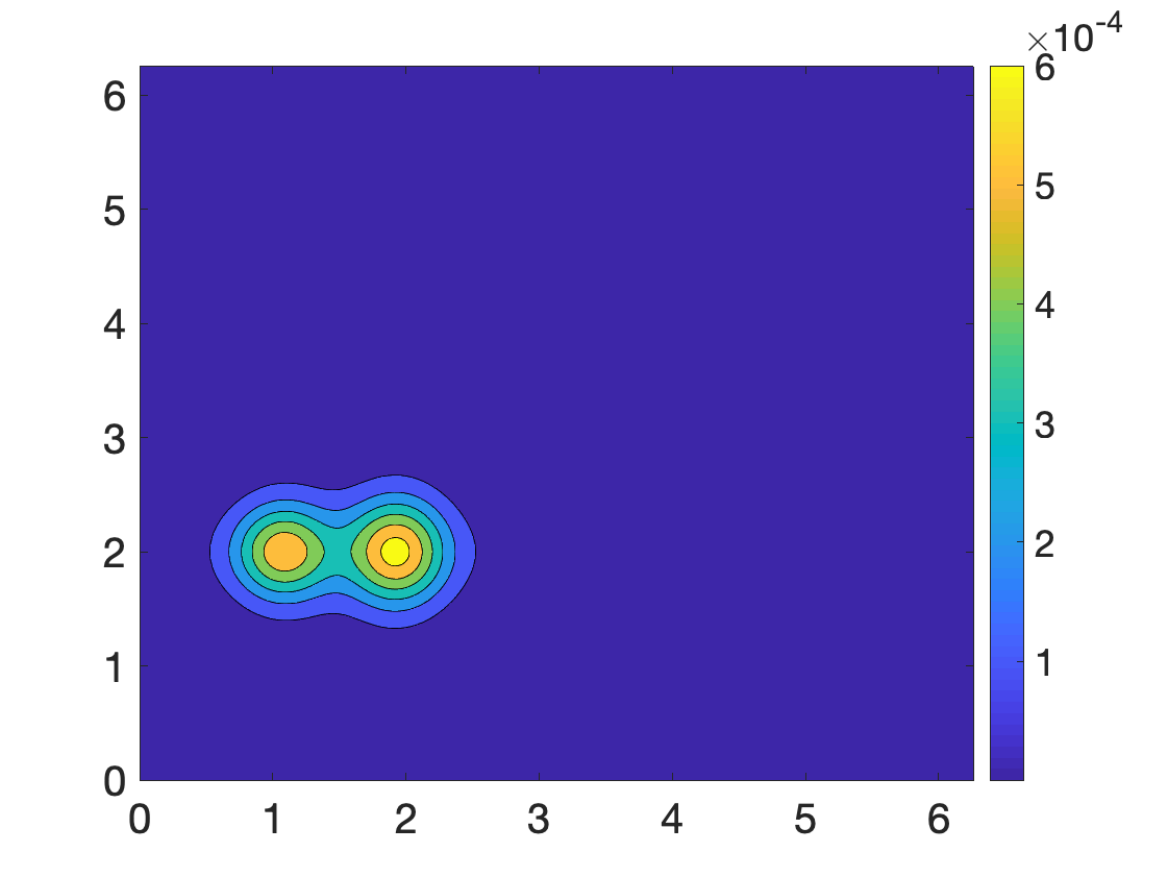}}}
  \subfloat[\textbf{MFM} with $\epsilon = 2^{-7}*L$ at Time 0.1]{
    \label{fig:3c}
    {\includegraphics[width=0.3\textwidth]{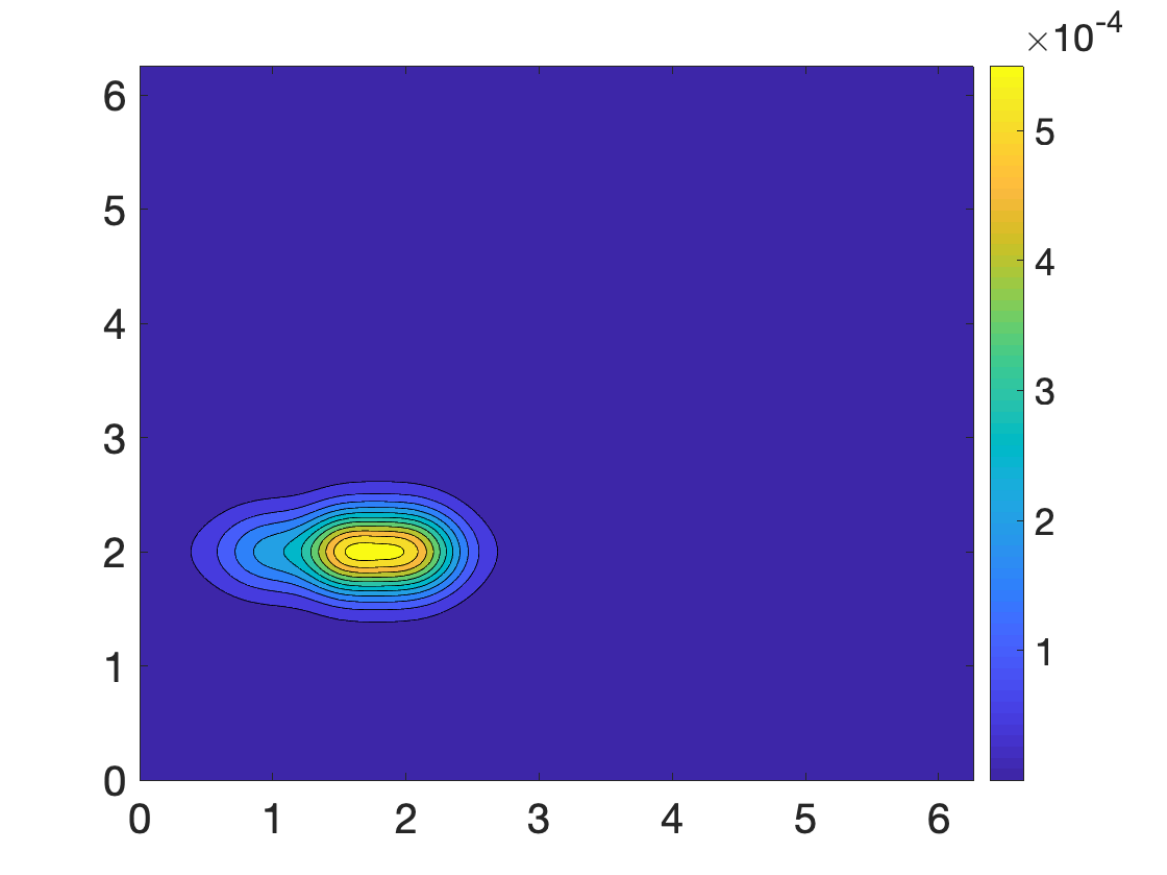}} }

\vskip\baselineskip

    \subfloat[\textbf{SM} at Time 1]{
    \label{fig:3d}
    {\includegraphics[width=0.3\textwidth]{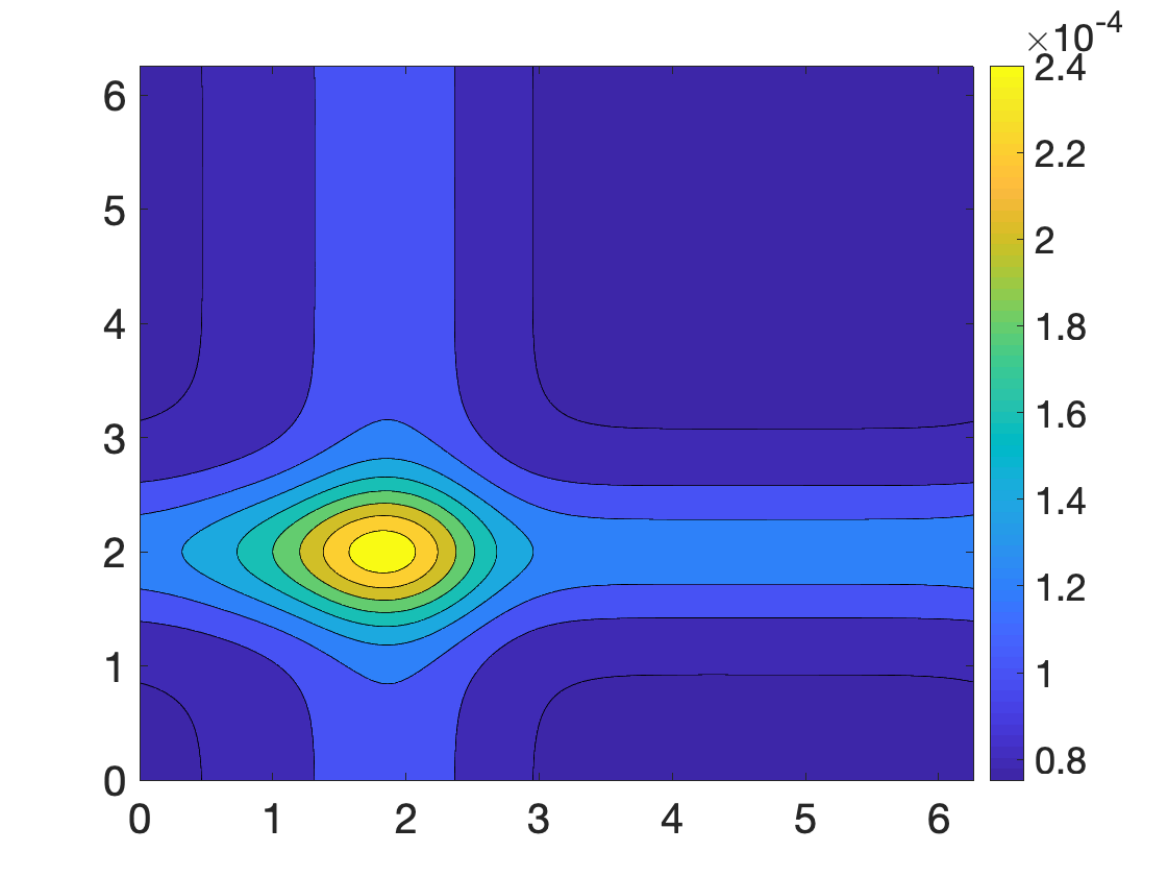}} }
  \subfloat[\textbf{MFM} with $\epsilon = 2^{-3}*L$ at Time 1]{
    \label{fig:3e}
    {\includegraphics[width=0.3\textwidth]{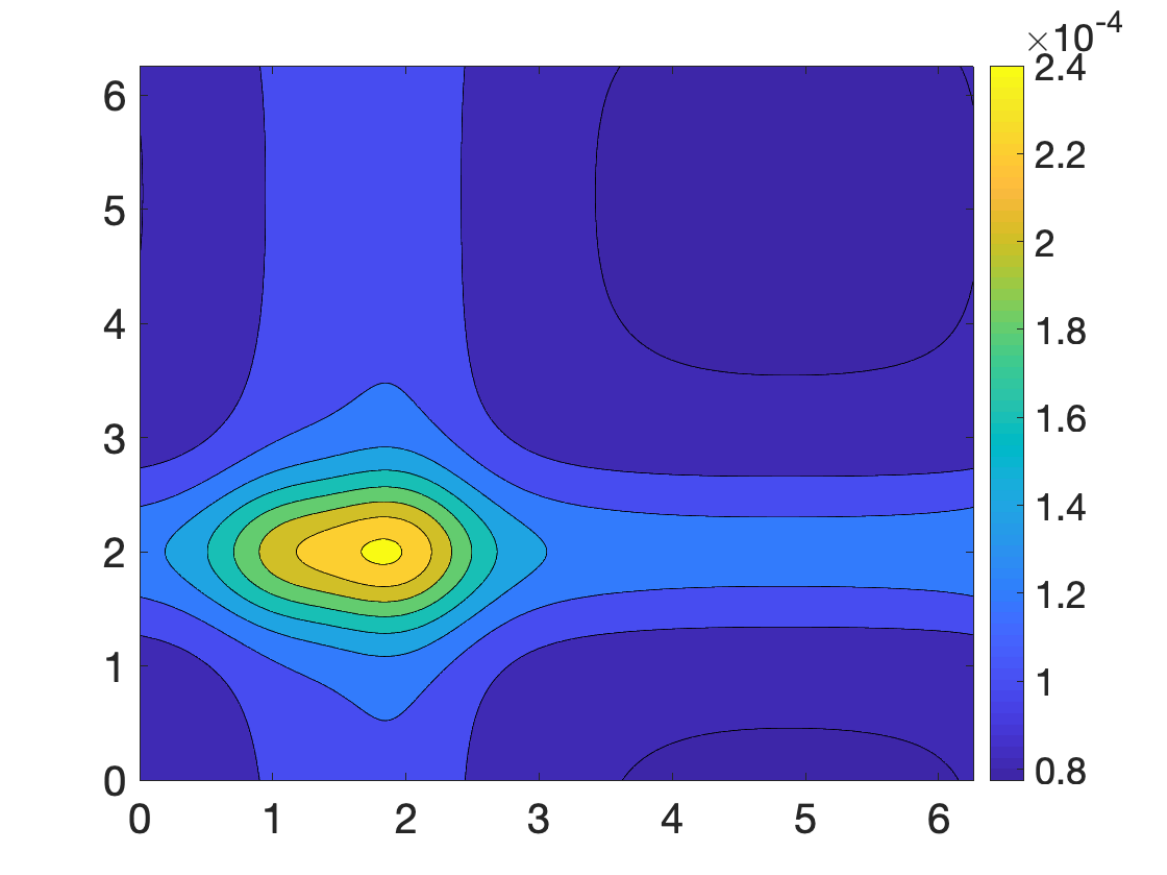}}}
  \subfloat[\textbf{MFM} with $\epsilon = 2^{-7}*L$ at Time 1]{
    \label{fig:3f}
    {\includegraphics[width=0.3\textwidth]{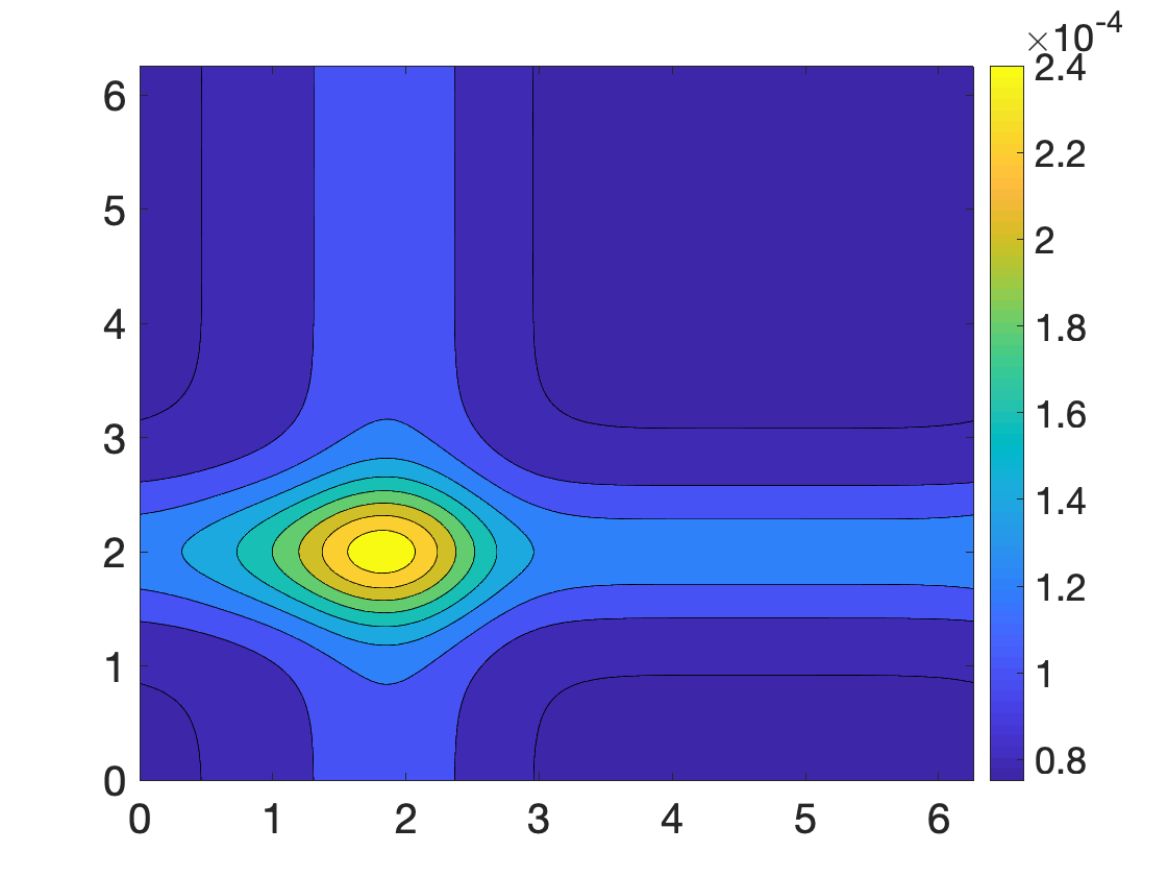}} }

  \caption{Two dimensional spatial profile for species C. }

  \label{fig:spatialProfile}
\end{figure*}
We remark that even though the dynamics of the \textbf{SM} and \textbf{MFM} are different, as long as we choose $\epsilon$ sufficiently small, in particular $\epsilon = 2^{-7}L \approx .8\% L$, there is no apparent visual difference between the two models, as illustrated in \cref{fig:spatialProfile}.

\subsection{Examples of disagreement between the PBSRD, MFM and SM models.}\label{SS:NumericsBeyond}
\commentab{
Our previous examples illustrated our main result, exploring regimes where the MFM converges to the SM as $\epsilon \to 0$. This may suggest that one can always use the SM in applications where a deterministic model is sufficient. We now illustrate two contexts where the SM is problematic as an approximation to average concentration fields in the PBSRD model, while the MFM captures key aspects of their behavior. In \cref{SS:BPM} we consider a pattern formation example, the Baras-Pearson-Mansour (BPM) Model, and show the \textbf{MFM}  is able to approximate the steady state statistics of CRDME simulations of the volume-reactivity PBSRD model. In contrast, the \textbf{SM} appears to converge to different steady-state statistics. In \cref{SS:CD28}, we study a simplified version of a model for regulation of T cell signaling from~\cite{ZI:2020}. We demonstrate that the \textbf{MFM} is able to show the same qualitative switch-like behavior at steady state as CRDME simulations of the volume-reactivity PBSRD model in~\cite{ZI:2020}, whereas it is not immediately clear what an appropriate SM to use for this problem would be. 

\subsubsection{Baras-Pearson-Mansour (BPM) Model}\label{SS:BPM}

We consider a reaction-diffusion system with three species, U, V and W undergoing the following reactions \cref{eq:BPMReactions} called the Baras-Pearson-Mansour (BPM) Model~\cite{baras1996microscopic, baras1990microscopic},
\begin{equation}\label{eq:BPMReactions}
U+W \overset{\eta_1}{\rightarrow} V+W,\qquad  2V \underset{\eta_3}{\overset{\eta_2}{\rightleftarrows}} W,\quad U\underset{\eta_5}{\overset{\eta_4}{\rightleftarrows}} \emptyset, \qquad V \underset{\eta_7}{\overset{\eta_6}{\rightleftarrows}}  \emptyset.
\end{equation}

We use the parameters provided in \cite{kim2017stochastic} for a reaction-limited system, and fix the spatial domain as a $32 \mu m \times32 \mu m$ square with periodic boundary conditions. The diffusivities are $D_V = D_W = D_U/10 = 0.01 \mu m^2 / sec$ with rate constants $\eta_1 = \eta_2 = 2\times 10^{-4} \mu m^2 / sec$, $\eta_3 = 1.0 sec^{-1}$, $\eta_4 = 3.33\times 10^{-3} sec^{-1}$, $\eta_5 = 16.7 \mu m^{-2} sec^{-1}$, $\eta_6 = 3.67\times 10^{-2} sec^{-1}$ and $\eta_7 = 4.44 \mu m^{-2} sec^{-1}$. For the CRDME and \textbf{MFM}, we consider two bimolecular interaction length scales, $\epsilon = 0.05 \mu m$ and $\epsilon = 0.025 \mu m$ (kept the same in the two bimolecular reactions of the system). The corresponding particle-level Gaussian kernels' rates $k_{i}$ ($i = 1,2$), for a reaction-limited system, are calibrated by imposing that 
\begin{equation}
\eta_i = \frac{1}{|\Omega|} \int_{\Omega^{2}}  \frac{k_i}{\paren{\sqrt{2\pi\epsilon^2}}^d} e^{-\frac{|\vx-\vy|^2}{2\epsilon^2}}\, d{\vx}\, d{\vy},
\end{equation}
where $\Omega$ denotes the spatial domain. This corresponds to matching the well-mixed reaction-rate constant in the (formal) infinite diffusivity limit. The product placement rule for the $U+W \rightarrow V+W$ reaction is to place V at the position of U. The placement rule for the $2V \rightarrow W$ reaction is to place W with equal probability at the position of the first V or the second V. For the reverse $W \rightarrow 2V$ reaction one V is placed at the position of W, while the position of the other is determined so as to ensure detailed balance of the reversible reaction, see~\cite{IZ:2018}.

We denote the spatial-average of the average number density fields for species U, V, W at time $t$ as $n_U(t), n_V(t), n_W(t)$ respectively. We initiate the system randomly around a point on the limit cycle, $(n_U(0), n_V(0), n_W(0)) = (1686,\, 534, \, 56) \, \mu m^{-2}$. More precisely,  we generate the spatially inhomogenous initial particle number for species $s$ in voxel $V_i$ from a Poisson distribution with mean $n_s(0)\times|V_i|$ following \cite{kim2017stochastic} and use the same initial number density for all the models considered.  We use the same Fourier collocation method as in \cref{SS:Numerics} to solve for the \textbf{SM} and \textbf{MFM}, choosing a time step of $dt = 0.1\, sec$ and using $N = 100$ points per coordinate axis. The CRDME is used for simulation of the underlying particle model, using the same underlying mesh. 

\cref{fig:bpm_density} demonstrates that for all the models with $\epsilon$ sufficiently small, the short-time spatially-averaged average number densities for species U agree as we proved in \cref{thm:convTwoModels}. At intermediate times, the stochastic reaction mechanism in the particle model facilitates faster relaxation, with smaller reactive length scales giving faster relaxation for both the CRDME and \textbf{MFM}. This is consistent with the observations in Figure 5 of \cite{donev2018efficient}. Neither the \textbf{MFM} or the \textbf{SM} give good approximations to the relaxation timescales in the CRDME simulations, though the disagreement of the \textbf{SM} appears less than the \textbf{MFM} for the two values of $\epsilon$ shown in the figure. In contrast, for very long times the \textbf{MFM} demonstrates better agreement with the limiting steady state value from the CRDME simulations for each value of $\epsilon$, while the \textbf{SM} shows a clear discrepancy (see the right panel of \cref{fig:bpm_density}). We note that it is an open question to study how the long time behaviors of the \textbf{MFM}, \textbf{SM} and CRDME relate to each other in complex biochemical systems. We hope to explore how the bimolecular reaction range and system parameters affect the steady state pattern formations and macroscopic observables from the three models in future work.

\begin{figure*}[tb]
  \centering
  \subfloat[]{
    \label{fig:bpm_a}
    {\includegraphics[width=0.5\textwidth]{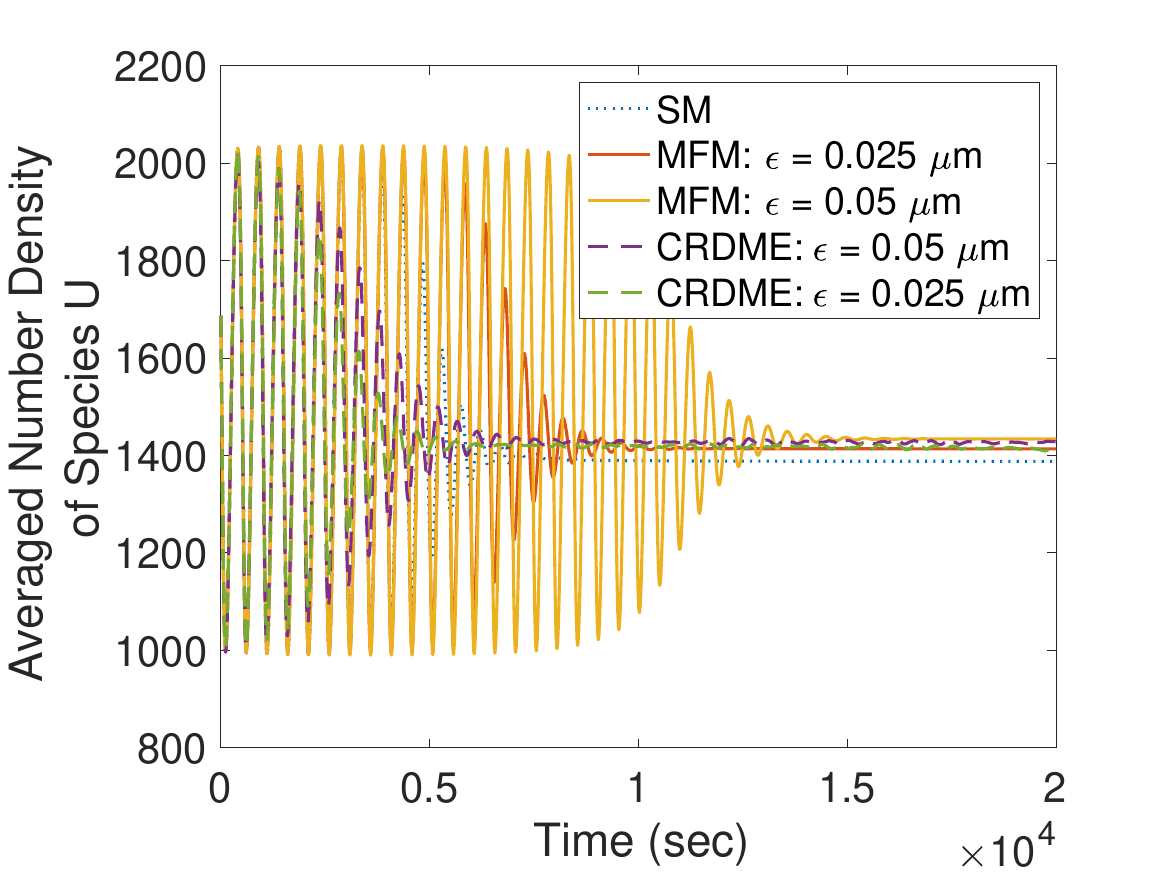}} }
  \subfloat[]{
    \label{fig:bpm_b}
    {\includegraphics[width=0.5\textwidth]{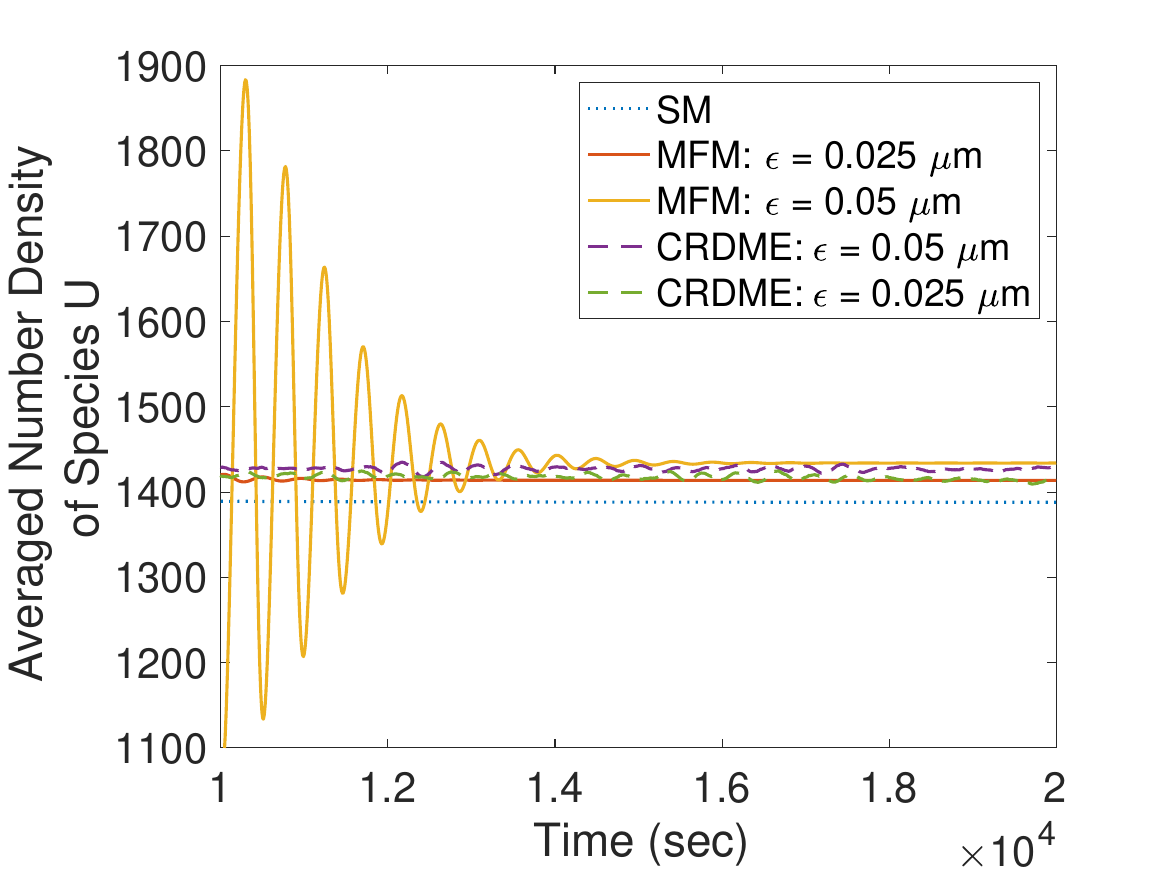}}}
  \caption{ Spatially-averaged average number density of species U versus time (sec) from \textbf{SM} model, \textbf{MFM} model and CRDME model up to time 20000 sec. The results of the CRDME are averaged ovder 10 simulations. (b) is a zoomed in version of (a) focusing on the long-time behavior.}
  \label{fig:bpm_density}
\end{figure*}

\FloatBarrier
\subsubsection{Tethered Surface Receptor Interactions in T Cell Signaling}\label{SS:CD28}
We now consider a simplified model for a tethered surface receptor interactions that occur in T cell signaling (T cells play a key role in the adaptive immune response). The example illustrates a case where the \textbf{MFM} is well-defined, and captures key qualitative behavior of the underlying particle model, but it is not immediately clear how to formulate an appropriate \textbf{SM} (the assumptions of~\cref{thm:convTwoModels} are violated). This demonstrates that while in many contexts the \textbf{SM} and \textbf{MFM} agree well for physically-appropriate parameters, there are cases where the \textbf{MFM} itself serves as a useful model for biological systems.

Surface receptors within the cell membrane often have cytosolic tails that contain docking cites for cytosolic enzymes (i.e. enzymes diffusing within the cell), and regulatory sites that can be modulated by such enzymes. The length and stiffness properties of these tails then define an effective interaction distance for bimolecular reactions involving such receptors, called the molecular reach of the reaction~\cite{ZI:2020}. Enzymes attached to binding sites on tails interact with regulatory sites on nearby tails within the three-dimensional volume proximal to the cell membrane. In contrast, the receptors to which the tails are attached diffuse within the two-dimensional membrane surface. We therefore obtain a reaction-diffusion process of particles (receptors) moving in a two-dimensional domain but reacting through three-dimensional reaction kernels.

In~\cite{ZI:2020}, we investigated a tethered signaling reaction in which surface membrane PD-1 proteins could inhibit activated CD-28 surface receptors, a key component in sustaining T cell signaling responses. We explored how the size of the molecular reach (i.e. bimolecular interaction distance) and diffusivity of the receptors could influence the efficacy of CD-28 inhibition by PD-1. Letting CD28 denote the inactivated (i.e. unphosphorylated) state, and CD28* the activated (i.e. phosphorylated) state, our model had the basic reactions that
\begin{equation}\label{eq:BPMReactions}
\text{CD28} \xrightarrow{\lambda} \text{CD28*},\qquad  \text{CD28*} + \text{PD-1} \xrightarrow{\hat{K}_{2.5D}^{{\epsilon}}(\cdot) \text{ or } \hat{K}_{2D}^{{\epsilon}}(\cdot)} \text{CD28} + \text{PD-1}.
\end{equation}
Here CD28 activation (phosphorylation) follows a first order reaction with rate $\lambda$. Inactivation (dephosphorylation) of CD28* is controlled by PD-1, and modeled by a second order tethered reaction with bimolecular reaction kernel \cref{eq:kernal2_5D}. It depends on the molecular reach $\epsilon$ derived from a polymer model for the cytoplasmic tail of the protein~\cite{ZI:2020}. The bimolecular reaction kernel is given by
\begin{equation}\label{eq:kernal2_5D}
\hat{K}_{2.5D}^{{\epsilon}}(x) = k_{2.5D}\times \paren{\frac{1}{2\pi \epsilon^2}}^{3/2} \exp\paren{-\frac{\abs{x}^2}{2\epsilon^2}}.
\end{equation}
Notice that \cref{eq:kernal2_5D} is a 3D Gaussian kernel, but the molecules will be restricted to diffuse within the two-dimensional membrane surface. We therefore label it the $2.5D$ reactive kernel. To understand how having a three-dimensional interaction for particles moving in two-dimensions changes the reaction efficiacy, in~\cite{ZI:2020} we compared it with a purely 2D bimolecular reaction kernel. The latter is given by the 2D Gaussian kernel
\begin{equation}\label{eq:kernal2D}
\hat{K}^{\epsilon}_{2D}(x) = k_{2D}\times \paren{\frac{1}{2\pi \epsilon^2}} \exp\paren{-\frac{\abs{x}^2}{2\epsilon^2}}.
\end{equation}

We now use our \textbf{MFM} to study the influence of the molecular reach $\epsilon$ and diffusivity, denoted by $D$, on CD28 phosphorylation in a simplified version of the preceding model. Let us denote $A(x, t)$ as the number density of CD28 at position $x$ at time t, and $B(x, t)$ as the number density of CD28* at position $x$ at time t. For illustrative purposes, we use simplified parameters, and assume there is only one (stationary) PD-1 protein in the system. Assume the spatial domain is a $[0, 50 \, nm]\times[0, 50 \, nm]$ square patch of membrane, with periodic boundary conditions. The one PD-1 molecule is placed at the center of the domain, $(25, 25)$, so that the number density of the PD-1 molecule is given by the constant field $\delta_{(25, 25)}(x)$. The \textbf{MFM} for this system is then
\begin{equation} \label{eq:MFM_CD28}
\begin{aligned} 
\frac{\partial}{\partial t} A(x, t) &= D \Delta_x A(x, t) - \lambda A(x, t) + \hat{K}^{\epsilon}(x-(25, 25))B(x, t),\\
\frac{\partial}{\partial t} B(x, t) &= D \Delta_x B(x, t) + \lambda A(x, t) - \hat{K}^{\epsilon}(x-(25, 25))B(x, t),
\end{aligned}
\end{equation}
where $\hat{K}^{\epsilon}$ is $\hat{K}^{\epsilon}_{2.5D}$ in the physiological case and $\hat{K}^{\epsilon}_{2D}$ in the idealized case of purely two-dimensional bimolecular interactions. 
\begin{figure*}[!ht]
  \centering
\includegraphics[width=0.5\textwidth]{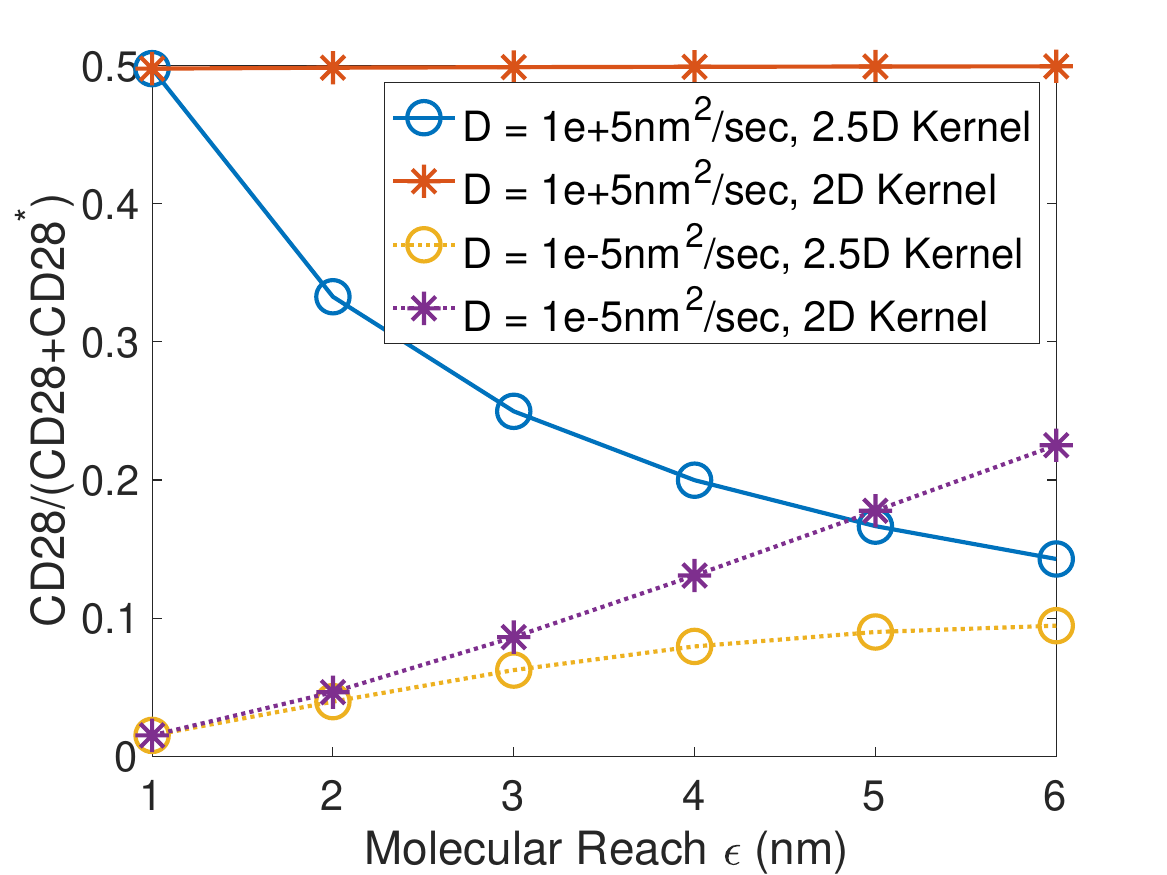}
  \caption{Steady-state fraction of CD28 that is inactivated versus molecular reach $\epsilon$ for different choices of bimolecular reaction kernel and diffusivity. Different diffusivities are labeled by different line styles: $D = 10^5 nm^2/sec$ is a solid line and $D = 10^{-5} nm^2/sec$ is a dotted line. Different reactive kernels are labeled by different markers: the 2D kernel is labeled by star markers and the 2.5D kernel is labeled by circle markers.}
  \label{fig:CD28_fraction}
\end{figure*}

In the following numerical experiments we fix $\lambda = 1 \, sec^{-1}$, $k_{2.5D} = k_{2D} = 2500\,  nm^2/sec$ and the initial density as constant, $A(x, 0) = 10^{-4} \, nm^{-2}$, $B(x, 0) = 0 \, nm^{-2}$.  We use the same numerical methods as \cref{SS:Numerics} for solving the \textbf{MFM}, choosing a time step of $dt = 0.001\, sec$ and the number of spatial points per axis to be $N = 100$. We study the influence of molecular reach $\epsilon$ on the steady-state fraction of CD28 in the inactivated state, illustrating the potency of PD-1. In~\cref{fig:CD28_fraction} we show how this fraction is modulated for each reaction kernel in both the reaction-limited case (i.e. fast diffusion $D = 10^{5} \, nm^2/sec$, solid curve) and the diffusion-limited case (i.e. slow diffusion $D = 10^{-5} \, nm^2/sec$, dotted line). We see that in the case of the (unphysical) 2D kernel (star markers), in both the diffusion-limited and reaction-limited regimes the inactivated fraction is non-decreasing with respect to the reach $\epsilon$. In contrast, for the physiological 2.5D kernel (circle markers), the inactivated fraction decreases in the reaction-limited regime, but increases in the diffusion-limited regime, as the reach $\epsilon$ is increased. This behavior qualitatively reproduces what we observed for a CRDME-based particle model where both CD28 and PD-1 could diffuse in~\cite{ZI:2020}, illustrating how the reach and diffusivity of a receptor can combine to modulate its regulatory efficacy.

While the \textbf{MFM} for the physiological 2.5D model follows from the underlying particle model of~\cite{ZI:2020}, it is not immediately clear what a corresponding \textbf{SM} should be. One could try to just write down such a model, but it would need to capture the qualitative dependency illustrated in~\cref{fig:CD28_fraction}, which arises from the explicit length scale over which spatial interactions can occur in the particle model. If one instead tries to derive the \textbf{SM} as the $\epsilon \to 0$ limit of the \textbf{MFM}, it is also unclear what this limit should be in the 2.5D case, where the bimolecular reaction kernel does not satisfy the assumptions of~\cref{thm:convTwoModels}. In particular, the kernel approaches a three-dimensional delta function, but is being used as a coefficient within a two-dimensional model. In contrast, the \textbf{MFM} is relatively immediate to write down, and as we have demonstrated reproduces the qualitative dependence of the fraction of inactivated receptor on the molecular reach $\epsilon$.
}

\section{Local Well-posedness and Regularity Analysis}\label{S:Local}

Local existence and uniqueness can be derived from the classical contraction mapping argument, see for example  \cite{SU:2017}. For completeness, we present the details for the Mean Field Model (\textbf{MFM}) in this section. We define $\X = [C_{b, unif}(\R^d)]^J$ to be the $J$-vector space of uniformly bounded and uniformly continuous functions on $\R^d$, equipped with the norm
$$||\vecrho||_\X = \sup_{j = 1, \cdots, J} \sup_{x\in\R^d} |\rho_j(x)|,$$
where $\vecrho = (\rho_1, \rho_2, \cdots, \rho_J)^T$ and each $\rho_j\in C_{b, unif}(\R^d)$, $j = 1, \cdots, J$.

Applying the variation of constants formula to the \textbf{MFM} \cref{Eq:density_formula_PIDEs} we have
\begin{equation}\label{eq:mildsolu1}
\vecrho(t) = S(t)\vecrho(0) + \int_0^t S(t-s)N_2[\vecrho](s)\, ds,
\end{equation}
where $\vecrho(t) = (\rho_1(\cdot, t), \rho_2(\cdot, t), \cdots, \rho_J(\cdot, t))^T$,  $S(t)$ is the heat semigroup generated by the linear diffusion operator Diag($D_1\lap_x, D_2\lap_x, \cdots, D_J\lap_x$), and $N_2[\vecrho] = ((N_2[\vecrho])_1, (N_2[\vecrho])_2, \cdots, (N_2[\vecrho])_J )^T$ represents the nonlinear reaction term with
\begin{align}
(N_2[\vecrho])_j(x, t)
&
= - \sum_{\ell = 1}^L \paren{
\frac{1}{\vec{\alpha}^{(\ell)}!}  \sum_{r = 1}^{\alpha_{\ell j}} \int_{\tilde{\vec{x}} \in \mathbb{X}^{(\ell)}}   \delta_{x}(\tilde{x}_r^{(j)})  K_\ell(\tilde{\vec{x}}) \, \left( \Pi_{k = 1}^{J} \Pi_{s = 1}^{\alpha_{\ell k}}  \rho_{k}(\tilde{x}_{s}^{(k)}, t)\right) \, d\tilde{\vec{x}}
}  \nonumber\\
&
\quad+\sum_{\ell = 1}^L \paren{  \frac{1}{\vec{\alpha}^{(\ell)}!} \sum_{r = 1}^{\beta_{\ell j}}  \int_{\tilde{\vec{x}} \in\mathbb{X}^{(\ell)}}  K_\ell(\tilde{\vec{x}}) \left( \int_{\vy \in \mathbb{Y}^{(\ell)}}   \delta_{x}(y_r^{(j)}) m_\ell(\vec{y}\, | \,\tilde{\vec{x}}) \,d \vec{y} \right) \left( \Pi_{k = 1}^J \Pi_{s = 1}^{\alpha_{\ell k}}  \rho_{k}(\tilde{x}_s^{(k)}, t)\right) \, d\tilde{\vec{x}}
},
\end{align}
for any $j = 1, \cdots, J$.
\begin{remark}
$S(t)$ is a $C_0$-semigroup on $\X$ and $||S(t)\vecrho||_\X\leq ||\vecrho||_\X$ for all $t\geq 0$ and $\vecrho\in\X$ . Furthermore, if all the diffusion coefficients are strictly positive, as we will subsequently assume, then $S(t)$ is an analytic semigroup on $\X$.
\end{remark}

Fix $C_1 > 0$, $T_0$ sufficiently small, and define the Banach space $\Y = C([0, T_0], \X)$ equipped with the norm
$$||\vecrho||_{\Y} = \sup_{t\in[0, T_0]} ||\vecrho(t)||_\X.$$
Let
\begin{align}
M &=  \{\vecrho \in\Y : ||\vecrho(\cdot) - S(\cdot)\vecrho_0||_\Y \leq C_1\},\label{Eq:Mspace}
\end{align}
adapted with the same norm
$$||\vecrho||_M = ||\vecrho||_{\Y}.$$
Note that by the contraction property of heat semigroup, we would have $C_2:=\sup_{ \vecrho\in M } ||\vecrho||_M \leq C_1 + C_0$, where $C_0 := \norm{\vecrho_0}_{M}$.

We first show $N_2[\vecrho]$ is a smooth and locally Lipschitz function w.r.t. $\vecrho$.
\begin{lemma}\label{lem:reactionPIDE}

Assume the reaction kernels and placement measures are of the form given in~\cref{Ass:MainAssumptions}, with the assumed allowable types of first and second order reactions. For any $\vecrho, \tilde{\vecrho} \in M$
we have that the following hold:
\begin{description}
\item[(P1) {[Boundedness]}]  $$||N_2[\vecrho]||_M \leq 3 \paren{C_2^2 \vee 1} \paren{\sum_{\ell = 1}^L k_\ell} $$

\item[(P2) {[Locally Lipschitz]}] $$||N_2[\vecrho] - N_2[\tilde{\vecrho}]   ||_M \leq 3 \paren{C_2 \vee 1} \paren{\sum_{\ell = 1}^L k_\ell}  ||\vecrho - \tilde{\vecrho}||_M,$$
\end{description}
where the constants $k_{\ell}$ are defined in \cref{Assume:kernelTwo} and \cref{Assume:kernelOne}.
\end{lemma}
\begin{proof}
\textbf{Proof of (P1):}
We have the following two estimates for the components of $N_2[\vecrho]$.
\begin{align*}
&\abs{\frac{1}{\vec{\alpha}^{(\ell)}!}  \sum_{r = 1}^{\alpha_{\ell j}} \int_{\tilde{\vec{x}} \in \mathbb{X}^{(\ell)}}   \delta_{x}(\tilde{x}_r^{(j)})  K_\ell(\tilde{\vec{x}}) \, \left( \Pi_{k = 1}^{J} \Pi_{s = 1}^{\alpha_{\ell k}}  \rho_{k}(\tilde{x}_{s}^{(k)}, t)\right) \, d\tilde{\vec{x}}}\\
&\leq \abs{\int_{\tilde{\vec{x}} \in \mathbb{X}^{(\ell)}}   \delta_{x}(\tilde{x}_1^{(j)})  K_\ell(\tilde{\vec{x}}) \, \left( \Pi_{k = 1}^{J} \Pi_{s = 1}^{\alpha_{\ell k}}  \rho_{k}(\tilde{x}_{s}^{(k)}, t)\right)d\tilde{\vec{x}}}\\
&\leq \abs{\int_{\tilde{\vec{x}} \in \mathbb{X}^{(\ell)}}   \delta_{x}(\tilde{x}_1^{(j)})  K_\ell(\tilde{\vec{x}}) \, d\tilde{\vec{x}}} \times ||\vecrho||_M^{|\vec{\alpha}^{(\ell)}|} \leq  k_\ell  \paren{C_2^2 \vee 1},
\end{align*}
and
\begin{align*}
&\abs{ \frac{1}{\vec{\alpha}^{(\ell)}!} \sum_{r = 1}^{\beta_{\ell j}}  \int_{\tilde{\vec{x}} \in\mathbb{X}^{(\ell)}}  K_\ell(\tilde{\vec{x}}) \left( \int_{\vy \in \mathbb{Y}^{(\ell)}}   \delta_{x}(y_r^{(j)}) m_\ell(\vec{y}\, | \,\tilde{\vec{x}}) \,d \vec{y} \right) \left( \Pi_{k = 1}^J \Pi_{s = 1}^{\alpha_{\ell k}}  \rho_{k}(\tilde{x}_s^{(k)}, t)\right) \, d\tilde{\vec{x}}}\\
& \leq  \frac{1}{\vec{\alpha}^{(\ell)}!}\beta_{\ell j}||\vecrho||_M^{|\vec{\alpha}^{(\ell)}|} \times  \abs{\int_{\tilde{\vec{x}} \in\mathbb{X}^{(\ell)}}  K_\ell(\tilde{\vec{x}}) \left( \int_{\vy \in \mathbb{Y}^{(\ell)}}   \delta_{x}(y_1^{(j)}) m_\ell(\vec{y}\, | \,\tilde{\vec{x}}) \,d \vec{y} \right)  \, d\tilde{\vec{x}}}\\
&\leq 2 k_\ell  \paren{C_2^2 \vee 1}.
\end{align*}
Then
\begin{align*}
||N_2[\vecrho]||_M & \leq 3\paren{C_2^2 \vee 1}\paren{\sum_{\ell = 1}^L k_\ell}.
\end{align*}

\textbf{Proof of (P2)}
For first and second order reactions we have
\begin{align*}
&|\Pi_{k = 1}^{J} \Pi_{s = 1}^{\alpha_{\ell k}}  \rho_{k}(x_{s}^{(k)}, t) - \Pi_{k = 1}^{J} \Pi_{s = 1}^{\alpha_{\ell k}}  \tilde{\rho}_{k}(x_{s}^{(k)}, t)| \leq \paren{2||\vecrho||_M\vee 1}||\vecrho - \tilde{\vecrho}||_M.
\end{align*}
This implies the following two estimates for the components of $N_2[\vecrho] - N_2[\tilde{\vecrho}]$,
\begin{align*}
&\abs{\frac{1}{\vec{\alpha}^{(\ell)}!}  \sum_{r = 1}^{\alpha_{\ell j}} \int_{\tilde{\vec{x}} \in \mathbb{X}^{(\ell)}}   \delta_{x}(\tilde{x}_r^{(j)})  K_\ell(\tilde{\vec{x}}) \, \left( \Pi_{k = 1}^{J} \Pi_{s = 1}^{\alpha_{\ell k}}  \rho_{k}(\tilde{x}_{s}^{(k)}, t)- \Pi_{k = 1}^{J} \Pi_{s = 1}^{\alpha_{\ell k}}  \tilde{\rho}_{k}(x_{s}^{(k)}, t)\right) \, d\tilde{\vec{x}}}\\
&\leq \abs{\int_{\tilde{\vec{x}} \in \mathbb{X}^{(\ell)}}   \delta_{x}(\tilde{x}_1^{(j)})  K_\ell(\tilde{\vec{x}}) \, \left( \Pi_{k = 1}^{J} \Pi_{s = 1}^{\alpha_{\ell k}}  \rho_{k}(\tilde{x}_{s}^{(k)}, t)- \Pi_{k = 1}^{J} \Pi_{s = 1}^{\alpha_{\ell k}}  \tilde{\rho}_{k}(x_{s}^{(k)}, t)\right)d\tilde{\vec{x}}}\\
&\leq \abs{\int_{\tilde{\vec{x}} \in \mathbb{X}^{(\ell)}}   \delta_{x}(\tilde{x}_1^{(j)})  K_\ell(\tilde{\vec{x}}) \, d\tilde{\vec{x}}} \times\paren{2||\vecrho||_M\vee 1}||\vecrho - \tilde{\vecrho}||_M \leq k_\ell \paren{2C_2\vee 1}||\vecrho - \tilde{\vecrho}||_M,
\end{align*}
and
\begin{align*}
&\abs{ \frac{1}{\vec{\alpha}^{(\ell)}!} \sum_{r = 1}^{\beta_{\ell j}}  \int_{\tilde{\vec{x}} \in\mathbb{X}^{(\ell)}}  K_\ell(\tilde{\vec{x}}) \left( \int_{\vy \in \mathbb{Y}^{(\ell)}}   \delta_{x}(y_r^{(j)}) m_\ell(\vec{y}\, | \,\tilde{\vec{x}}) \,d \vec{y} \right) \left( \Pi_{k = 1}^J \Pi_{s = 1}^{\alpha_{\ell k}}  \rho_{k}(\tilde{x}_s^{(k)}, t)- \Pi_{k = 1}^{J} \Pi_{s = 1}^{\alpha_{\ell k}}  \tilde{\rho}_{k}(x_{s}^{(k)}, t)\right) \, d\tilde{\vec{x}}}\\
& \leq  \frac{1}{\vec{\alpha}^{(\ell)}!}\beta_{\ell j}\paren{2||\vecrho||_M\vee 1}||\vecrho - \tilde{\vecrho}||_M\times \abs{ \int_{\tilde{\vec{x}} \in\mathbb{X}^{(\ell)}}  K_\ell(\tilde{\vec{x}}) \left( \int_{\vy \in \mathbb{Y}^{(\ell)}}   \delta_{x}(y_1^{(j)}) m_\ell(\vec{y}\, | \,\tilde{\vec{x}}) \,d \vec{y} \right)  \, d\tilde{\vec{x}}}\\
&\leq 2k_\ell\paren{2C_2\vee 1}||\vecrho - \tilde{\vecrho}||_M.
\end{align*}
Therefore,as claimed, we obtain that
\begin{align*}
||N_2[\vecrho] - N_2[\tilde{\vecrho}]   ||_M \leq 3\paren{\sum_{k = 1, \cdots, L}k_\ell}\paren{2C_2\vee 1}||\vecrho - \tilde{\vecrho}||_M.
\end{align*}
\end{proof}

Making use of the boundedness and locally Lipschitz properties of $N_2[\, \cdot \,]$ from \cref{lem:reactionPIDE}, we obtain
\begin{theorem}\label{thm:localEU_PIDE}
Assume the conditions of \cref{lem:reactionPIDE}. For all $C_0 > 0$, there exists a $T_0 > 0$ such that for $\vecrho_0\in\X$, with $||\vecrho_0||_\X \leq C_0$, there exists a unique mild solution $\vecrho\in C([0, T_0], \X)$  to \cref{Eq:density_formula_PIDEs} with $\vecrho(0) = \vecrho_0$.
\end{theorem}
\begin{proof}
Let $F[\vecrho]$ be the righthand side of the variation of constant formula \cref{eq:mildsoluPDE}. 
 We'll show in the following that $F$ maps $M$ to $M$ and is a contraction mapping in $M$ using the contraction property of heat semigroup and the properties of the reaction operator in \cref{lem:reactionPIDE}.
\begin{enumerate}[wide, labelwidth=!, labelindent=0pt]
\item $F$ maps $M$ to $M$ follows from

\begin{align*}
||F[\vecrho] - S(t)\vecrho_0||_M &  = \sup_{t\in [0, T_0]}  ||F[\vecrho](t) - S(t)\vecrho_0||_\X  = \sup_{t\in [0, T_0]}  \norm{ \int_0^t S(t-s)N_2[\vecrho](s)\, ds}_\X\\
&\leq \sup_{t\in [0, T_0]}   \int_0^t ||S(t-s)N_2[\vecrho](s)||_\X\, ds\leq \sup_{t\in [0, T_0]}   \int_0^t ||N_2[\vecrho](s)||_\X\, ds\\
&\leq T_0\paren{3 \paren{C_2^2 \vee 1} \paren{\sum_{\ell = 1}^L  \kappa_\ell}}\leq C_1,
\end{align*}
as long as $T_0\leq \frac{C_1}{3 \paren{C_2^2 \vee 1} \paren{\sum_{\ell = 1}^L  \kappa_\ell}}$.
\item $F$ is a contraction mapping in $M$ follows from
\begin{align*}
||F[\vecrho] - F[\tilde{\vecrho}]  ||_M &= \sup_{t\in [0, T_0]}  ||F[\vecrho](t) - F[\tilde{\vecrho}]||_\X
\leq \sup_{t\in [0, T_0]}   \int_0^t ||S(t-s)(N_2[\vecrho](s) - N_2[\tilde{\vecrho}](s))||_\X\, ds\\
& \leq \sup_{t\in [0, T_0]}   \int_0^t ||N_2[\vecrho](s) - N_2[\tilde{\vecrho}](s)||_\X\, ds
\leq \sup_{t\in [0, T_0]}   \int_0^t ||N_2[\vecrho] - N_2[\tilde{\vecrho}]||_M\, ds\\
& \leq T_0\paren{3 \paren{2C_2 \vee 1} \paren{\sum_{\ell = 1}^L  \kappa_\ell} }||\vecrho(s) - \tilde{\vecrho}(s)||_M
\leq \frac{1}{2}||\vecrho(s) - \tilde{\vecrho}(s)||_M,
\end{align*}
as long as $T_0\leq \frac{1}{6 \paren{2C_2 \vee 1} \paren{\sum_{\ell = 1}^L  \kappa_\ell} }$.
\end{enumerate}
Therefore, by the contraction mapping theorem there exists a unique mild solution $\vecrho\in C([0, T_0], \X)$  to \cref{Eq:density_formula_PIDEs} satisfying \cref{eq:mildsoluPDE}.
\end{proof}

\begin{theorem}\label{thm:regularityPIDE}
Under the conditions of \cref{thm:localEU_PIDE}, the mild solution $\rho_j(x, t) \in C([0, T_0], C_{b, unif}(\R^d)) \cap C^1((0, T_0], C^2(\R^d))$ is a classical solution  to \cref{Eq:density_formula_PIDEs}. Furthermore, if $\rho_j(x, 0) \in C_b^1(\R^d)$, for all $j = 1, \cdots, J$, then  $\rho_j(x, t) \in   C^2_b(\R^d)$ for any $t\in (0, T_0] $.
\end{theorem}
\begin{proof}
Since $\vecrho\in C([0, T_0], \X)$ as we showed in \cref{thm:localEU_PIDE}, classical results for nonhomogenous Cauchy problems give that $\rho_j(x, t)\in  C^1((0, T_0], C^2(\R^d))$ (see Chapter 2.3 Theorem 7 in \cite{E:2010}). $\vecrho$ is hence a classical solution. We'll next show that under the condition  $\rho_j(x, 0) \in C_b^1(\R^d)$ then  $\rho_j(x, t) \in  C^2_b(\R^d)$ for $t \in (0,T_0]$, i.e. all the first and second partial derivatives in $x$ are bounded for $0 < t \leq T_0$.

Let us denote $\Phi(x, t)$ as the fundamental solution of $d$-dimensional heat equation. Note that for any $x\in\R^d$
\begin{equation}
\int_{\R^d}  |\partial_{x_i}\Phi(x-y, t)|\, dy = \sqrt{\frac{1}{\pi t}}.
\end{equation}

We establish two estimates for the first and second partial derivatives for $\rho_j(x, t)$, $t\in (0, T_0]$. We claim
\begin{equation}
||\partial_{x_i} \rho_j(x, t)||_{C_b(\R^d)} \leq  ||\rho_j(x, 0) ||_{C_b^1(\R^d)} + 6 \paren{C_2^2 \vee 1} \paren{\sum_{\ell = 1}^L  k_\ell}  \sqrt{\frac{T_0}{\pi}} \tag{\textbf{Inequality 1}}
\end{equation}
for any $0 < t\leq T_0$, any $i = 1, \cdots, d$, and $j = 1, \cdots, J$. Starting from \cref{eq:mildsoluPDE}, we have that
\begin{align}
|\partial_{x_i} \rho_j(x, t)| & =  | \partial_{x_i} \int_{\R^d} \Phi(x-y, t) \rho_j(y, 0)\, dy +  \partial_{x_i} \int_{0}^t \int_{\R^d} \Phi(x-y, t-s) \paren{N_2[\vecrho]}_j(y, s)\, dy\, ds|\nonumber\\
&
 \leq \int_{\R^d}  |\Phi(x-y, t)| |\partial_{x_i} \rho_j(y, 0)|\, dy + ||N_2[\vecrho]||_{M}\times \int_{0}^t \int_{\R^d} |\partial_{x_i} \Phi(x-y, t-s)|\, dy\, ds\nonumber\\
 &
= ||\rho_j(x, 0) ||_{C_b^1(\R^d)}+ ||N_2[\vecrho]||_{M}\times \int_{0}^t \sqrt{\frac{1}{\pi(t-s)}}\, ds\nonumber\\
 &
\leq  ||\rho_j(x, 0) ||_{C_b^1(\R^d)} + 3 \paren{C_2^2 \vee 1} \paren{\sum_{\ell = 1}^L  k_\ell}  \times 2 \sqrt{\frac{t}{\pi}},
\end{align}
where in the last inequality, we use the estimates for $ ||N_2[\vecrho]||_{M}$ in \cref{lem:reactionPIDE} and recall that we denote $C_2:=\sup_{ \vecrho\in M } ||\vecrho||_M $.

For the second derivatives, we claim
\begin{equation}
||\partial_{x_i, x_k} \rho_j(x, t)||_{C_b(\R^d)} \leq    ||\rho_j(x, 0) ||_{C_b^1(\R^d)}\sqrt{\frac{1}{\pi t}} +  6C_4(2C_2\vee 1)\paren{\sum_{\ell = 1}^L k_\ell} \sqrt{\frac{T_0}{\pi}} \tag{\textbf{Inequality 2}}
\end{equation}
for any $0 < t\leq T_0$, any $i, k = 1, \cdots, d$, and $j = 1, \cdots, J$, where we have denoted $C_4 =  \sup_{j = 1, \cdots, J}||\rho_j(x, 0) ||_{C_b^1(\R^d)} + 6 \paren{C_2^2 \vee 1} \paren{\sum_{\ell = 1}^L  k_\ell}  \sqrt{\frac{T_0}{\pi}} $. Using that $ ||\partial_{x_i} \rho_j(x, t)||_{C_b(\R^d)} \leq  C_4$ from (\textbf{Inequality 1}) we have
\begin{align}
&|\partial_{x_i}\paren{N_2[\vecrho]}_j(x, t)| \nonumber\\
& =  - \sum_{\ell = 1}^L \paren{
\frac{1}{\vec{\alpha}^{(\ell)}!}  \sum_{r = 1}^{\alpha_{\ell j}} \partial_{x_i}\ \int_{\tilde{\vec{x}} \in \mathbb{X}^{(\ell)}}   \delta_{x}(\tilde{x}_r^{(j)})  K_\ell(\tilde{\vec{x}}) \, \left( \Pi_{k = 1}^{J} \Pi_{s = 1}^{\alpha_{\ell k}}  \rho_{k}(\tilde{x}_{s}^{(k)}, t)\right) \, d\tilde{\vec{x}}
}  \nonumber\\
&
\quad+\sum_{\ell = 1}^L \paren{  \frac{1}{\vec{\alpha}^{(\ell)}!} \sum_{r = 1}^{\beta_{\ell j}}  \partial_{x_i}\ \int_{\tilde{\vec{x}} \in\mathbb{X}^{(\ell)}}  K_\ell(\tilde{\vec{x}}) \left( \int_{\vy \in \mathbb{Y}^{(\ell)}}   \delta_{x}(y_r^{(j)}) m_\ell(\vec{y}\, | \,\tilde{\vec{x}}) \,d \vec{y} \right) \left( \Pi_{k = 1}^J \Pi_{s = 1}^{\alpha_{\ell k}}  \rho_{k}(\tilde{x}_s^{(k)}, t)\right) \, d\tilde{\vec{x}}
}\nonumber\\
&
\leq \sum_{\ell = 1}^L \frac{\beta_{\ell j} + \alpha_{\ell j}}{\vec{\alpha}^{(\ell)}!} \times k_\ell\paren{ 2C_2 \vee 1} \sup_{k = 1, \cdots, J}||\partial_{x_m} \rho_{k}(x, t)||_{C_b(\R^d)} \qquad \text{( by  \cref{lem:diffNonlinearPIDE1}-\cref{lem:diffNonlinearPIDE2} )} \nonumber\\
&
\leq 3C_4(2C_2\vee 1)\paren{\sum_{\ell = 1}^L k_\ell}.
\end{align}
Using \cref{eq:mildsoluPDE}
\begin{align}
|\partial_{x_i, x_k} \rho_j(x, t)| & =  | \partial_{x_i, x_k} \int_{\R^d} \Phi(x-y, t) \rho_j(y, 0)\, dy +  \partial_{x_i, x_k} \int_{0}^t \int_{\R^d} \Phi(x-y, t-s) \paren{N_2[\vecrho]}_j(y, s)\, dy\, ds|\nonumber\\
&
 \leq \int_{\R^d}  |\partial_{x_i}\Phi(x-y, t)||\partial_{x_k} \rho_j(y, 0)|\, dy +\sup_{t\in(0, T_0]} ||\partial_{x_i}\paren{N_2[\vecrho]}_j(x, t)||_{C_b(\R^d)}\nonumber\\
 &
 \qquad \qquad \times \int_{0}^t \int_{\R^d} |\partial_{x_k} \Phi(x-y, t-s)|\, dy\, ds\nonumber\\
 &
\leq ||\rho_j(x, 0) ||_{C_b^1(\R^d)} \sqrt{\frac{1}{\pi t}} +  3C_4(2C_2\vee 1)\paren{\sum_{\ell = 1}^L \kappa_\ell}
\times \int_{0}^t \sqrt{\frac{1}{\pi(t-s)}}\, ds\nonumber\\
 &
\leq  ||\rho_j(x, 0) ||_{C_b^1(\R^d)}\sqrt{\frac{1}{\pi t}} +  3C_4(2C_2\vee 1)\paren{\sum_{\ell = 1}^L \kappa_\ell}
\times 2 \sqrt{\frac{t}{\pi}}.
\end{align}

Thus we conclude $\rho_j(x, t) \in C^2_b(\R^d)$ for any $t\in (0, T_0] $ as claimed and the theorem has been proven.
\end{proof}

For the Standard Model (\textbf{SM}), let  $N_1[\vecrho] = \paren{ \paren{N_1[\vecrho]}_1, \paren{N_1[\vecrho]}_2, \cdots, \paren{N_1[\vecrho]}_J }^T$ represent the nonlinear reaction term with
\begin{equation}
\paren{N_1[\vecrho]}_j =  \sum_{\ell = 1}^L \kappa_\ell (\beta_{\ell j} - \alpha_{\ell j}) \paren{ \Pi_{k = 1}^J \rho_{k}(x, t)^{\alpha_{\ell k}} }  ,\nonumber\\
\label{Eq:nonlinear_formula_PDEs}
\end{equation}
for $1\leq j\leq J$. It is standard to show that $N_1[\vecrho]$ is a smooth and locally Lipschitz function w.r.t. $\vecrho$. Analogously to the previous calculations for the \textbf{MFM}, we can obtain
\begin{theorem}\label{thm:localEU_PDE}
For all $C_0 > 0$, there exists a $T_0 > 0$ such that for $\vecrho_0\in\X$, with $||\vecrho_0||_\X \leq C_0$, there exists a unique mild solution $\vecrho\in C([0, T_0], \X)$  to \cref{Eq:density_formula_PDEs} with $\vecrho(0) = \vecrho_0$.
\end{theorem}

\begin{theorem}\label{thm:regularityPDE}
Under the condition of  \cref{thm:localEU_PDE},  then the mild solution $\rho_j(x, t) \in C([0, T_0], C_{b, unif}(\R^d)) \cap C^1((0, T_0], C^2(\R^d))$ is a classical solution  to \cref{Eq:density_formula_PDEs}. Furthermore when assuming $\rho_j(x, 0) \in C_b^1(\R^d)$, for all $j = 1, \cdots, J$, indeed $\rho_j(x, t) \in   C^2_b(\R^d)$ for any $t\in (0, T_0] $.
\end{theorem}


\section{On Global Well-posedness }\label{S:Global}

Global existence in time of the classical solution to the \textbf{SM} \cref{Eq:density_formula_PDEs} for general reaction systems is a difficult open problem.  We refer the interested reader to the recent review article \cite{P:2010} for survey of the current state of the art. The recent papers by \cite{FMT:2020,CGV:2019,S:2018} also deal with global well-posedness of reaction-diffusion systems under various combinations of growth and mass control assumptions.

For two main reasons, the setup of this paper is only partially covered by the existing literature. First, we are dealing with non-local systems of equations (the \textbf{MFM}) while the vast majority of the literature has concentrated on local systems of equations (such as the \textbf{SM}). Second, since our main interest in this paper is to examine conditions under which the \textbf{SM} is a special case of the \textbf{MFM}, we need to be able to assume that the reaction kernels converge to delta Dirac masses, see Assumption \ref{Assume:kernelTwo}, which would then require uniform bounds with respect to this approximation (see Theorem \ref{thm:convTwoModels}). The latter precludes us from being able to work with global boundedness assumptions on the reaction kernels (see also Remark \ref{R:NolenPaper} for a more detailed explanation of this).

In this section we demonstrate that global well-posedness can be proven to hold for our non-local \textbf{MFM} \cref{Eq:density_formula_PIDEs} in, at least, the cases of $A+B\rightleftarrows C+D$ and $A+B\rightleftarrows C$ (the latter under specific choices for the placement measures). The case of $A+B\rightleftarrows C+D$ can be addressed using the results of \cite{FMT:2020}, using the mass conservation property of the non-local system, see Lemma \ref{lem:FourSepciesCondition}. This approach, however, fails for the non-local $A+B\rightleftarrows C$ reaction. We address the non-local $A+B\rightleftarrows C$ reaction in Lemma \ref{lem:ReversibleCondition} by modifying an argument of \cite{P:2010} to deal with the non-local nature of the equations, which uses in an essential way that two of the equations have only linear growth. We stress here that the two methods are different in nature; the first is based on mass conservation properties,  while the second is based on finding a linear combination of the equations with linear growth.

It is an interesting open problem to address global well-posedness in a more unifying way generally for the \textbf{MFM} \cref{Eq:density_formula_PIDEs}. However, this is outside the scope of this paper, whose primary focus is elucidating how the commonly used \textbf{SM} approximates the rigorous large population limit of PBSRD systems given by the \textbf{MFM}.

We begin in \cref{lem:FourSepciesCondition} by addressing the $A+B\rightleftarrows C+D$ reaction network. In Remark \ref{R:NolenPaper} we make a number of comments on alternative approaches in the literature that might be used to establish global existence.

\begin{lemma}\label{lem:FourSepciesCondition}
For the $A+B\rightleftarrows C+D$ reaction, both the \textbf{SM} and the \textbf{MFM} are globally well posed, i.e. \cref{thm:regularityPDE} and \cref{thm:regularityPIDE} hold for all $T_0 < \infty$.
\end{lemma}
\begin{proof}
By the results of \cite{FMT:2020},  global well-posedness will follow if the following additional conditions hold
\begin{description}
\item[(A1) Local Lipschitz and Preservation of Positivity] For all $j = 1, \cdots, J$, $\paren{N[\vecrho]}_j$ is locally Lipschitz and  $\paren{N[\vecrho]}_j \geq 0$ for all $\vecrho \geq 0$ with $\rho_j = 0$,
\item[(A2) Mass Control] $\sum_{j = 1}^J w_j\times\paren{N[\vecrho]}_j \leq C_0 + C_1\sum_{j = 1}^J \rho_j$, for $\vecrho \geq 0$, some constants $C_0$ and $C_1$, and some set of $\{w_j\}_{j=1,\cdots,J}$ with $w_j > 0$,
\item [(A3) (Super)-Quadratic Growth] $|\paren{N[\vecrho]}_j| \leq C(1+||\vecrho||^{2+\epsilon})$, for $\vecrho \geq 0$, some constant $C, \epsilon > 0$ and all $j$,
\end{description}
where $N[\vecrho]$ represents the nonlinear reaction term of a general reaction-diffusion equation, including our \textbf{SM} and \textbf{MFM}. Here, in condition (A3), $||\, \cdot\,||$ represents the uniform norm in space, i.e. $||\vecrho|| = \sup_{j = 1, \cdots, J}\sup_{x\in\R^d} |\rho_j(x)|$. This is slightly different from what is assumed in \cite{FMT:2020}, but an examination of the proofs in \cite{FMT:2020} shows that the argument goes through in this norm.

Let us denote by $\rho_1(x, t), \cdots, \rho_4(x, t)$ the concentration at position $x$ and time $t$ for species A, B, C, D respectively. Let $\vecrho = (\rho_1, \rho_2, \rho_3, \rho_4)^T$. Without loss of generality, we may assume that once the forward reaction happens, A becomes C and B becomes D, whereas once the backward reaction happens, C becomes A and D becomes B. Then the reaction term $N[\vecrho] = \paren{\paren{N[\vecrho]}_1, \paren{N[\vecrho]}_2 , \paren{N[\vecrho]}_3, \paren{N[\vecrho]}_4  }^T$ for the \textbf{MFM} Model is
\begin{equation}
N[\vecrho](x, t) =
\begin{pmatrix}
-  \paren{\int_{\R^d} K_{1}(x, y) \rho_2(y, t) \, dy} \rho_1(x, t) + \paren{\int_{\R^d}K_{2}(x, w)\rho_4(w, t)\, dw} \rho_3(x, t)\\[2mm]
-  \paren{\int_{\R^d} K_{1}(y, x) \rho_1(y, t) \, dy} \rho_2(x, t) + \paren{\int_{\R^d}K_{2}(z, x)\rho_3(z, t)\, dw} \rho_4(x, t)\\[2mm]
+ \paren{\int_{\R^d} K_{1}(x, y) \rho_2(y, t) \, dy} \rho_1(x, t) - \paren{\int_{\R^d}K_{2}(x, w)\rho_4(w, t)\, dw} \rho_3(x, t)\\[2mm]
+  \paren{\int_{\R^d} K_{1}(y, x) \rho_1(y, t) \, dy} \rho_2(x, t) - \paren{\int_{\R^d}K_{2}(z, x)\rho_3(z, t)\, dw} \rho_4(x, t)\\[2mm]
\end{pmatrix}.
\end{equation}
Preservation of positivity is satisfied by the non-negativity of the reaction kernel $K_1(x, y)$ and $K_2(z, w)$. A local Lipschitz condition is shown in \cref{lem:reactionPIDE}, while  $\sum_{j = 1}^4 \paren{N[\vecrho]}_j = 0$ gives (A2). By symmetry, it is sufficient to show (A3) when $j = 1$.
\begin{align*}
|\paren{N[\vecrho]}_1| & = \abs{ -\paren{\int_{\R^d} K_{1}(x, y) \rho_2(y, t) \, dy} \rho_1(x, t) + \paren{\int_{\R^d}K_{2}(x, w)\rho_4(w, t)\, dw} \rho_3(x, t)},\\
&\leq \paren{\int_{\R^d} \hat{K}_{1}(x - y) |\rho_2(y, t)| \, dy} |\rho_1(x, t)| + \paren{\int_{\R^d}\hat{K}_{2}(x - w)|\rho_4(w, t)|\, dw} |\rho_3(x, t)|,\\
&\leq \paren{||\hat{K}_1||_{L^1}  + ||\hat{K}_2||_{L^1}} ||\vecrho||^2.
\end{align*}

Thus, conditions (A1)-(A3) are satisfied for the \textbf{MFM} Model. Note, all the constants involved do not depend on $\epsilon$. The same argument works for the \textbf{SM} model by choosing the reaction kernels to be delta functions.
\end{proof}

\begin{remark}\label{R:NolenPaper}
For completeness we mention here that apart from \cite{FMT:2020},  \cite{CGV:2019} and \cite{S:2018} have also addressed the global well-posedness question of general  reaction-diffusion systems with condition (A1), (A2), and (A3) with $C_1 = 0$ plus an additional entropy condition.

In addition, in principle, one could also use the regularized effects of the bimolecular reaction kernel in order to prove well-posedness of the $A+B\rightleftarrows C+D$ system, as in Section 7 in \cite{LLN:2019}. However, this method of proof requires boundedness of the reaction kernel whereas for the \textbf{SM}, the reaction kernel term is essentially  a Dirac delta function, which is rather singular. The idea of the method of proof in \cite{LLN:2019} is that under the conditions (A1), (A2) and (A3),  $||\rho_j(x, t)||_{L^1(\R^d)}$ is bounded for any time $t\in [0, \infty )$ for all $j = 1, \cdots, J$. Then the main argument of \cite{LLN:2019} is that by the prior $L^1$ bound on the solutions and the $L^\infty$ bound on the reaction kernels, the regularized (convoluted) quadratic growth term can be bounded by linear growth of the solutions, and thereby admits global existence. However, as $\epsilon\to0$, the $L^\infty$ bound of the reaction kernel  unfortunately is not uniform in the approximation parameter $\epsilon$. Therefore, given that our aim is to compare the two models, this method of proof is not immediately applicable to our case.

As such, we found that for our problem of interest using \cite{FMT:2020} was more straight-forward for the non-local $A+B\rightleftarrows C+D$ system.
\end{remark}

We now address global existence for the $A+B\rightleftarrows C$ reaction network.
\begin{lemma}\label{lem:ReversibleCondition}
For the case of the $A+B\rightleftarrows C$ reaction, by choosing the binding placement measure $m_1(z\,|\, x, y) = p\delta(z - x) + (1-p)\delta(z-y)$, for some $p \in[0, 1]$ and assuming the detailed balance condition \cref{eq:dbcondit} in the \textbf{MFM}, solutions to both the \textbf{SM} and the \textbf{MFM} exist globally, i.e. \cref{thm:regularityPDE} and \cref{thm:regularityPIDE} hold for all $T_0 < \infty$.
\end{lemma}
\begin{proof}
We only address the \textbf{MFM} model as the situation for the local model \textbf{SM} is immediately covered by the results of both \cite{P:2010} and \cite{FMT:2020}.

Let us denote by $\rho_1(x, t), \cdots, \rho_3(x, t)$ the concentration at position $x$ and time $t$ for species A, B, C respectively. Let $T$ be the maximal existence time for $L^\infty$ solutions. The \textbf{MFM} Model is
\begin{align}\label{eq:ABC1}
\partial_t\rho_1(x, t) &= \DA\lap \rho_1(x, t) -  \left(\int_{\R^d} K_{1}(x, y) \rho_2(y, t) \, dy\right) \rho_1(x, t) + \int_{\R^d}K_{2}(z)\left( \int_{\R^d} m_{2}(x,y|z) dy \right)\rho_3(z, t)\, dz \nonumber\\
\partial_t \rho_2(y, t)&=  \DB\lap \rho_2(y, t) - \left(\int_{\R^d} K_{1}(x, y) \rho_1(x, t) \, dx\right) \rho_2(y, t)+ \int_{\R^d}K_{2}(z)\left( \int_{\R^d} m_{2}(x,y|z) dx \right)\rho_3(z, t)\, dz \nonumber\\
\partial_t \rho_3(z, t)&=  \DC\lap \rho_3(z, t) -  K_{2}(z)\rho_3(z, t) + \int_{\R^{2d}} K_{1}(x, y)m_1(z | x, y) \rho_1(x, t) \rho_2(y, t) \, dx \, dy.\nonumber\\
\end{align}

Recall that the detailed balance condition gives $$\hat{K}_1(x-y)m_1(z | x, y) \times \Kd = K_2(z)m_2(x,y|z),$$
where $\Kd = k_2/k_1$ is the equilibrium dissociation constant of the reaction and $K_2(z) = k_2$ is a constant function. Let us denote $(f*g)(x) = \int_{\R^d} f(x-y)g(y)\, dy$ as the convolution. Utilizing the detailed balance condition and the explicit form $m_1(z\,|\, x, y) = \alpha\delta(z - x) + (1-\alpha)\delta(z-y)$, \cref{eq:ABC1} can be further simplified to
\begin{align}
\label{eq:ABCL1}
\partial_t\rho_1 &= \DA\lap \rho_1-  (\hat{K}_1*\rho_2)   \rho_1+ \alpha K_2\times \rho_3+  (1-\alpha)\Kd\times (\hat{K}_1*\rho_3),\\
\label{eq:ABCL2}
\partial_t \rho_2&=  \DB\lap \rho_2 - (\hat{K}_1*\rho_1) \rho_2+ \alpha\Kd\times ( \hat{K}_1*\rho_3)  + (1-\alpha)K_2\times\rho_3 ,\\
\label{eq:ABCL3}
\partial_t \rho_3&=  \DC\lap \rho_3-  K_{2}\rho_3+ \alpha(\hat{K}_{1}* \rho_2)\rho_1+ (1-\alpha)(\hat{K}_{1}* \rho_1) \rho_2 . \\ \nonumber
\end{align}
By summing up \cref{eq:ABCL1,eq:ABCL2,eq:ABCL3}, we see immediately
\begin{align}\label{eq:massTypeIneq}
\partial_t\rho_3-  \DC\lap \rho_3  &= -\alpha\paren{ \partial_t\rho_1 - \DA\lap \rho_1} - (1-\alpha)\paren{ \partial_t\rho_2- \DB\lap \rho_2} \nonumber\\
&\qquad + \alpha^2 K_2\times \rho_3 +(1-\alpha)^2K_2\times\rho_3 + 2\alpha(1-\alpha)\Kd\times (\hat{K}_1*\rho_3) - K_{2}\rho_3.
\end{align}

Let $C$ and $C_1$ be a generic constant that only depends on $p, \alpha, \Kd, ||\hat{K}_1||_{L^1}, K_2, \DA, \DB, \DC, \rho_1(x, 0), \rho_2(x, 0), \rho_3(x, 0)$. We'll always assume that for all $i = 1, 2, 3$, $||\rho_i(x, 0)||_{L^p} < \infty$ for all $1\leq p\leq \infty$ and $\rho_i(x, 0)\geq 0$, for all $x\in\R^d$. By the positivity preserving property of \cref{eq:ABC1}, we actually know that $\rho_i(x, t)\geq 0$, for all $t\in[0, T]$. Then \cref{eq:ABCL1} gives
\begin{equation}
\partial_t\rho_1 - \DA\lap \rho_1 \leq \alpha K_2\times \rho_3+  (1-\alpha)\Kd\times (\hat{K}_1*\rho_3),\nonumber
\end{equation}
from which we obtain that for all $p\in(1, \infty)$,
\begin{equation}
||\rho_1(t)||_{L^p} \leq ||\rho_1(0)||_{L^p} + C \int_0^t[ || \rho_3(s)||_{L^p}+  ||\hat{K}_1*\rho_3(s)||_{L^p}]\, ds.\nonumber
\end{equation}
Using Young's inequality, $||\hat{K}_1*\rho_3||_{L^p}\leq ||\hat{K}_1||_{L^1}||\rho_3||_{L^p}$, it further simplifies to
\begin{equation}\label{eq:ABCL1Lp}
||\rho_1(t)||_{L^p} \leq ||\rho_1(0)||_{L^p} + C \int_0^t || \rho_3(s)||_{L^p}\, ds.
\end{equation}
Similarly, \cref{eq:ABCL2} gives
\begin{equation}\label{eq:ABCL2Lp}
||\rho_2(t)||_{L^p} \leq ||\rho_2(0)||_{L^p} + C \int_0^t || \rho_3(s)||_{L^p}\, ds.
\end{equation}
Let $Q_t := \R^d \times \brac{0,t}$. \cref{lem:duality} (an extension of  Lemma 3.4 in \cite{P:2010}) gives that
\begin{equation}\label{eq:dualEstimateFor3}
||\rho_3||_{L^p(Q_t)}\leq C\paren{1 + ||\rho_1||_{L^p(Q_t)} +||\rho_2||_{L^p(Q_t)} }.
\end{equation}
This implies
\begin{align}
\paren{\int_0^t || \rho_3(s)||_{L^p}\, ds}^p &\leq t^{p-1} \int_{Q_t} |\rho_3|^p\leq C\paren{1 + ||\rho_1||_{L^p(Q_t)} +||\rho_2||_{L^p(Q_t)} }^p\nonumber\\
&\leq C\paren{1 + ||\rho_1||_{L^p(Q_t)}^p +||\rho_2||_{L^p(Q_t)}^p }\nonumber.
\end{align}
Therefore, \cref{eq:ABCL1Lp,eq:ABCL2Lp} can be combined to give
\begin{align}
||\rho_1(t)||_{L^p}^p + ||\rho_2(t)||_{L^p}^p &\leq C_1+ C \paren{\int_0^t || \rho_3(s)||_{L^p}\, ds}^p\nonumber\\
&\leq C_1+ C\paren{1 + \int_0^t ||\rho_1(s)||_{L^p}^p +||\rho_2(s)||_{L^p}^p \, ds}.\nonumber
\end{align}

By Gronwall's inequality, we know that $\rho_1, \rho_2\in L^p(Q_T)$, and furthermore $\rho_3 \in L^p(Q_T)$ by \cref{eq:dualEstimateFor3}, for all $p\in (1, \infty)$. Notice that since the reaction terms in \cref{eq:ABCL1,eq:ABCL2,eq:ABCL3} are all polynomial (with a convolution that does not effect the $L^p$ bound by Young's inequality), they are bounded in $L^p(Q_T)$ for all $p\in(1, \infty)$. Morrey's inequality in $\R^{d+1}$ together with the $L^p$-regularity theory for parabolic operators (see Chapter IV Section 3 in \cite{Ladyzenskaja}) implies that $\rho_1, \rho_2, \rho_3\in L^\infty(Q_T)$ , and therefore $T = \infty$ (See Lemma 1.1 in \cite{P:2010} ).
\end{proof}

\begin{lemma}\label{lem:duality}
For $\rho_1, \rho_2, \rho_3$ as used in \cref{lem:ReversibleCondition} satisfying
\begin{align}\label{eq:massTypeIneq2}
\partial_t\rho_3-  \DC\lap \rho_3 + \theta_1\rho_3 + \theta_2(\hat{K}_1*\rho_3) &= -\alpha \paren{ \partial_t\rho_1 - \DA\lap \rho_1} - (1-\alpha)\paren{ \partial_t\rho_2- \DB\lap \rho_2},
\end{align}
on $Q_T := \R^d\times[0, T]$ for some $\theta_1\in\R$ and $\theta_2<0$, it holds that
\begin{equation}
||\rho_3||_{L^p(Q_t)}\leq C\paren{1 + ||\rho_1||_{L^p(Q_t)} +||\rho_2||_{L^p(Q_t)} },\nonumber
\end{equation}
for all $t\in (0, T]$ and $p\in(1, \infty)$.
\end{lemma}

\begin{proof}
For any $t\in (0, T]$, let $\phi$ be the solution of the following dual problem
\begin{equation}\label{eq:dualProb}
-[\partial_t\phi + \DC \lap\phi] +\theta_1\phi + \theta_2(\hat{K}_1*\phi) = \Theta,
\end{equation}
with $\phi(t) = 0$, where $\Theta\geq 0$ and $\Theta\in C_0^\infty(Q_t)$. Let $q = p/(p-1)$ be the conjugate index of $p$. Then the solution to \cref{eq:dualProb} $\phi$ is smooth,  satisfies $\phi\geq 0$ and for all $t\in[0, T]$,

\begin{equation}\label{eq:dualProbLq}
||\phi_t||_{L^q(Q_t)} + || \lap\phi||_{L^q(Q_t)} + \sup_{s\in[0, t]} ||\phi(s)||_{L^q} + \sup_{s\in[0, t]} ||\hat{K}_1*\phi(s)||_{L^q}  \leq C||\Theta||_{L^q(Q_t)}.
\end{equation}

One can show that $\phi\in L^q(Q_t)$ by a fixed point argument,  $\phi\geq 0$ by comparison principle  and further get the estimates \cref{eq:dualProbLq} by investigating the mild solution form via Young's convolution inequality and $L^q$ estimates of heat potential  (see Chapter IV Section 3 in \cite{Ladyzenskaja}).

Multiply $\partial_t\rho_1-  \DA\lap \rho_1$ by $\phi$ and integrate by parts on $Q_t$, we find
\begin{align}\label{eq:intByParts}
\int_{Q_t} \paren{\partial_t\rho_1-  \DA\lap \rho_1}\phi \,dt\, dx &= \int_{\R^d} [\rho_1(t)\phi(t) - \rho_1(0)\phi(0)] \,dx - \int_{Q_t} \paren{\partial_t\phi +  \DA\lap \phi}\rho_1 \,dt\,dx \nonumber\\
& =  -\int_{\R^d}  \rho_1(0)\phi(0) \,dx + \int_{Q_t} \paren{\Theta   +  (\DC -  \DA)\lap \phi - \theta_1\phi - \theta_2(\hat{K}_1*\phi)}\rho_1 \,dt \, dx.
\end{align}
Multipling \cref{eq:massTypeIneq2} by $\phi$ and integrating by parts over $Q_t$ in a similar manner as \cref{eq:intByParts}, we find
\begin{align}\label{eq:dualEstimate}
\int_{Q_t}\Theta\times\rho_3 \,dt \,dx&=  \int_{\R^d}\rho_3(0)\phi(0) \,dx + \alpha  \int_{\R^d} \rho_1(0)\phi(0)\,dx  +(1-\alpha)\int_{\R^d} \rho_2(0)\phi(0) \,dx \nonumber\\
& \qquad  -\alpha\int_{Q_t} \paren{\Theta   +  (\DC -  \DA)\lap \phi - \theta_1\phi - \theta_2(\hat{K}_1*\phi)}\rho_1\,dt \,dx \nonumber\\
&  \qquad - (1-\alpha)\int_{Q_t} \paren{\Theta   +  (\DC -  \DB)\lap \phi - \theta_1\phi - \theta_2(\hat{K}_1*\phi)}\rho_2\,dt \,dx.
\end{align}
Notice that $\rho_1, \rho_2, \rho_3, \phi\geq 0 $, then \cref{eq:dualEstimate} combined with \cref{eq:dualProbLq} gives
\begin{equation}
\int_{Q_t}\Theta\times\rho_3 \leq C ||\Theta||_{L^q(Q_t)}\paren{1 + ||\rho_1||_{L^p(Q_t)} + ||\rho_2||_{L^p(Q_t)}}\,dt \,dx.\nonumber
\end{equation}
Since the choices of $\Theta$ are arbitrary, by duality argument, we have
\begin{equation}
||\rho_3||_{L^p(Q_t)}\leq C\paren{1 + ||\rho_1||_{L^p(Q_t)} +||\rho_2||_{L^p(Q_t)} }.\nonumber
\end{equation}

\end{proof}

\begin{appendices}
\section{Lemmas for Estimating Differences Between Two Models}\label{A:appendixLemDiff}

First, we start with a key estimate presented in \cref{lem:molifier}.
\begin{lemma}\label{lem:molifier}
For any non-negative kernel $\hat{K}^\epsilon(w)\geq 0$ with $w\in\R^d$, $\epsilon$ sufficiently small,  of the type
\begin{description}
\item[(A) ] radially symmetric $\hat{K}^\epsilon(w)$ only depends on $|w|$,
\item[(B) ] renormalized integrable $|| \hat{K}^\epsilon||_{ L^1(\R^d) } = 1$,
\item[(C)] $\int_{\R^d} \hat{K}^\epsilon(w)  |w|^2 \, dw = O(\epsilon^2)$,
\end{description}
and for any function $f, g \in C^2_{b}(\R^d)$, any real number $|\alpha| < \infty$, we get as $\epsilon\rightarrow 0$
\begin{equation}
\norm{ \int_{\R^d} \hat{K}^\epsilon(x-y) f( x + \alpha (y-x))g(y)\, dy- f(x)g(x)}_{L^\infty} = || f\times g ||_{C^2_{b}(\R^d)}O(\epsilon^2),
\end{equation}
where $O(\epsilon^2)$ depends on $\alpha$.
\end{lemma}

\begin{remark}
In particular, in \cref{lem:molifier}, one can choose $g\equiv 1$ to obtain as $\epsilon\rightarrow 0$
\begin{equation}
\norm{ \int_{\R^d} \hat{K}^\epsilon(x-y) f( x + \alpha (y-x))\, dy- f(x)}_{L^\infty} = || f ||_{C^2_{b}(\R^d)}O(\epsilon^2).
\end{equation}
\end{remark}

\begin{proof}[Proof of \cref{lem:molifier}]
Assume that $\alpha \not= 0$. The case when $\alpha = 0$ can be adapted from the following proof. Notice that for any $f \in C^2_{b}(\R^d)$, any $x, y\in\R^d$, we can apply Taylor's theorem to obtain
\begin{equation}\label{eq:taylorThm}
f(y) = f(x) + \la \nabla f(x), y-x \ra +  \la y-x, D^2 f(\xi) (y-x) \ra,
\end{equation}
where $\la \cdot, \cdot \ra$ is the inner product on $\R^d$, $D^2 f(\xi)$ is the Hessen matrix at $x$ and $\xi\in B_{|y-x|}(x)$. Since $f \in C^2_{b}(\R^d)$, we can alternatively rewrite \cref{eq:taylorThm} as
\begin{equation}\label{eq:taylorApprox}
|f(y) - f(x) - \la \nabla f(x), y-x \ra| \leq ||f||_{C_b^2(\R^d)} \times |y-x|^2.
\end{equation}
Applying Taylor's theorem to $f( x + \alpha (y-x))\times g(y)$, we find
\begin{align*}\label{eq:taylorFGApprox}
&|f( x + \alpha (y-x))g(y) - f(x)g(x) - \alpha\la \nabla f(x), (y-x) \ra g(x) - \la \nabla g(x), (y-x) \ra f(x)|\\
 &\leq \paren{\alpha^2\vee 1} ||f\times g||_{C_b^2(\R^d)} \times |y-x|^2.
\end{align*}
so that
\begin{align*}
&\norm{\int_{\R^d} \hat{K}^\epsilon(x-y) f( x + \alpha (y-x))g(y)\, dy- f(x)g(x)}_{L^\infty}  \\
&
= \sup_{x\in\R^d} \abs{ \int_{\R^d} \hat{K}^\epsilon(x-y) \left[f( x + \alpha (y-x))g(y) - f(x)g(x)\right]\, dy}\\
&
\leq  \sup_{x\in \R^d} \abs{\int_{\R^d} \hat{K}^\epsilon(x-y) \left[f( x + \alpha (y-x))g(y) - f(x)g(x) - \alpha\la \nabla f(x), (y-x) \ra g(x) -\la \nabla g(x),  (y-x) \ra f(x) \right] \, dy} \\
&
\qquad+   \sup_{x\in \R^d} \abs{\int_{\R^d} \hat{K}^\epsilon(x-y) \left(\alpha\la \nabla f(x), (y-x) \ra g(x) +\la \nabla g(x),  (y-x) \ra f(x) \right) \, dy} \\
&
\leq  \sup_{x\in \R^d} \int_{\R^d} \hat{K}^\epsilon(x-y) \paren{\alpha^2\vee 1} ||f\times g||_{C_b^2(\R^d)} \times |y-x|^2 \, dy  \qquad\text{( by property (A) of the reaction kernel )}\\
&
\leq  \paren{\alpha^2\vee 1} ||f\times g||_{C_b^2(\R^d)} \int_{\R^d} \hat{K}^\epsilon(w)  |w|^2 \, dw , \qquad\text{( by property (C) of the reaction kernel )}\\
&
\leq   ||f\times g||_{C_b^2(\R^d)} O(\epsilon^2),
\end{align*}
where we note that in the final expression the $O(\epsilon^2)$ term depends on $\alpha$.
\end{proof}

\begin{lemma}\label{lem:estimates1}
\begin{equation}\label{eq:estimates1}
\Lambda^{\ell, j}(\tau)   \leq 2C\paren{ \max_{j = 1, \cdots, J}||\rho_j(x, \tau) - \rho_j^\epsilon(x, \tau) ||_{L^\infty} + \paren{C_2 + \frac{C_3}{\sqrt{\tau}}}O(\epsilon^2)},
\end{equation}
where $C, C_2, C_3$ is defined in the proof of \cref{thm:convTwoModels} and independent of $\epsilon$.
\end{lemma}
\begin{proof}
We'll discuss this case by case. When $\alpha_{\ell j} = 0$, $\Lambda^{\ell, j}(\tau) = 0$, and \cref{eq:estimates1} is trivially true. When $\alpha_{\ell j} = 1$, note that $\vec{\alpha}^{(\ell)}! = 1$ for the allowable types of reactions. If $\mathscr{R}_\ell$ is a first order reaction, based on  \cref{Assume:kernelOne},  $K^\epsilon_\ell(x) = k_\ell$, $x\in\R^d$,  is a constant rate function. As long as we choose $k_\ell = \kappa_\ell$,
\begin{align*}
\Lambda^{\ell, j}(\tau)
& \leq  k_\ell \max_{j = 1, \cdots, J}||\rho_j(x, \tau) - \rho_j^\epsilon(x, \tau) ||_{L^\infty}.
\end{align*}
When $\alpha_{\ell j} = 1$ and $\mathscr{R}_\ell$ is of the type $S_i + S_j \rightarrow \cdots$, $i\not= j$, \cref{Assume:kernelTwo} gives
\begin{align}\label{eq:estimates1_AB}
\Lambda^{\ell, j}(\tau)
&
=  \norm{  \paren{  \int_{\R^{d}}   \hat{K}^\epsilon_\ell(x-y) \,\rho^\epsilon_{i}(y, \tau) \, dy } \rho^\epsilon_{j}(x, \tau) - \kappa_{\ell}\rho_{i}(x, \tau)\rho_{j}(x, \tau)}_{L^\infty}\nonumber\\
&
\leq \norm{  \paren{  \int_{\R^{d}}   \hat{K}^\epsilon_\ell(x-y) \,\rho^\epsilon_{i}(y, \tau) \, dy  - \kappa_{\ell}\rho^\epsilon_{i}(x, \tau)} \rho^\epsilon_{j}(x, \tau)}_{L^\infty} \nonumber\\
&
\qquad + ||\rho^\epsilon_{j}(x, \tau)\paren{\rho^\epsilon_{i}(x, \tau) - \rho_{i}(x, \tau)}||_{L^\infty} + ||\rho_{i}(x, \tau)\paren{\rho^\epsilon_{j}(x, \tau) - \rho_{j}(x, \tau)}||_{L^\infty}\nonumber\\
&
\leq ||\rho^\epsilon_{i}(x, \tau)||_{C^2_b(\R^d)}O(\epsilon^2)\times||\rho^\epsilon_{j}(x, \tau)||_{L^\infty} + ||\rho^\epsilon_{j}(x, \tau)||_{L^\infty} ||\rho^\epsilon_{i}(x, \tau) - \rho_{i}(x, \tau)||_{L^\infty}\nonumber\\
&
\qquad + ||\rho_{i}(x, \tau)||_{L^\infty} ||\rho^\epsilon_{j}(x, \tau) - \rho_{j}(x, \tau)||_{L^\infty}\nonumber\\
&
\leq C\paren{ ||\rho^\epsilon_{i}(x, \tau)||_{C^2_b(\R^d)}O(\epsilon^2) + 2\max_{j = 1, \cdots, J}||\rho^\epsilon_{j}(x, \tau) - \rho_{j}(x, \tau)||_{L^\infty} },
\end{align}
by choosing $k_\ell =\kappa_\ell$ and applying \cref{lem:molifier} and the regularity results of the solutions in \cref{thm:regularityPDE} and  \cref{thm:regularityPIDE}. Here in the last step,  recall that $C  = \paren{\max_{\ell = 1, \cdots, L} k_\ell}  \max_{j = 1, \cdots, J} \{\sup_{\epsilon > 0} \sup_{\tau\in[0, T_0]}||\rho^\epsilon_{j}(x, \tau)||_{L^\infty} \vee \sup_{\tau\in[0, T_0]} ||\rho_{j}(x, \tau)||_{L^\infty} \}$ was assumed independent of $\epsilon$.

When $\alpha_{\ell j} = 2$, $\mathscr{R}_\ell$ should be of the type $S_j + S_j \rightarrow \cdots$,  \cref{Assume:kernelTwo} gives
\begin{align}
\Lambda^{\ell, j}(\tau)
&
=  \norm{\paren{  \int_{\R^{d}}   \hat{K}^\epsilon_\ell(x-y) \,\rho^\epsilon_{j}(y, \tau) \, dy } \rho^\epsilon_{j}(x, \tau) - 2\kappa_{\ell}\rho_{j}(x, \tau)^2}_{L^\infty}.\nonumber
\end{align}
By choosing $k_\ell =2 \kappa_\ell$ in this case,  \cref{eq:estimates1} follows closely from  \cref{eq:estimates1_AB}.
\end{proof}


\begin{lemma}\label{lem:estimates2}
\begin{equation}\label{eq:estimates2}
\Theta^{\ell, j}(\tau)   \leq 4C\paren{\max_{j = 1, \cdots, J}\norm{\rho_j(x, \tau) - \rho_j^\epsilon(x, \tau) }_{L^\infty} + \paren{C_2 + \frac{C_3}{\sqrt{\tau}}}O(\epsilon^2)},
\end{equation}
where $C, C_2, C_3$ is defined in the proof of \cref{thm:convTwoModels} and independent of $\epsilon$.
\end{lemma}
\begin{proof}

We'll again discuss this case by case. When $\beta_{\ell j} = 0$ then $\Theta^{\ell, j}(\tau) = 0$, and \cref{eq:estimates2} is trivially true.
\begin{enumerate}[wide, labelwidth=!, labelindent=0pt]
\item When  $\mathscr{R}_\ell$ is of the type $S_i \rightarrow S_j$, $i\not= j$, by plugging in  \cref{Assume:measureOne2One} and  \cref{Assume:kernelOne}, we obtain
\begin{align}
\Theta^{\ell, j}(\tau) &= \norm{k_\ell\rho^\epsilon_{i}(x, \tau) - \kappa_{\ell} \rho_{i}(x, \tau)||_{L^\infty}\leq k_\ell ||   \rho^\epsilon_{i}(x, \tau) - \rho_{i}(x, \tau)}_{L^\infty}.
\end{align}

\item When $\mathscr{R}_\ell$ is of the type $S_i + S_k\rightarrow S_j$, $i\not= k$, by plugging in  \cref{Assume:measureTwo2One} and \cref{Assume:kernelTwo}, we obtain
\begin{align}\label{eq:estimates2_ABC}
\Theta^{\ell, j}(\tau)
&
= \norm{\int_{\R^{2d}}  \hat{K}^\epsilon_\ell(x_1- x_2) m^\epsilon_\ell( x \, | x_1, x_2)  \rho^\epsilon_{i}(x_1, \tau) \rho^\epsilon_{k}(x_2, \tau)   \, dx_1\,dx_2- \kappa_{\ell} \rho_{i}(x, \tau) \rho_{k}(x, \tau)}_{L^\infty}\nonumber\\
&
\leq \norm{ \int_{\R^{2d}}  \hat{K}^\epsilon_\ell(x_1- x_2) m^\epsilon_\ell( x \, | x_1, x_2)  \rho^\epsilon_{i}(x_1, \tau) \rho^\epsilon_{k}(x_2, \tau)   \, dx_1\,dx_2- \kappa_{\ell} \rho^\epsilon_{i}(x, \tau) \rho^\epsilon_{k}(x, \tau)}_{L^\infty}\nonumber\\
&
\qquad + \kappa_{\ell}|| \rho^\epsilon_{i}(x, \tau) ||_{L^\infty}|| \rho_{k}(x, \tau) -  \rho^\epsilon_{k}(x, \tau)||_{L^\infty} + \kappa_{\ell}|| \rho^\epsilon_{i}(x, \tau) -  \rho_{i}(x, \tau)||_{L^\infty} || \rho_{k}(x, \tau)||_{L^\infty}.
\end{align}
Without loss of generality, we assume that $\alpha_i>0$ for all $1\leq i\leq I$ in the following, then by choosing $k_\ell = \kappa_\ell$, we have that
\begin{align*}
\bigg| \int_{\R^{2d}}  &\hat{K}^\epsilon_\ell(x_1- x_2) m^\epsilon_\ell( x \, | x_1, x_2)  \rho^\epsilon_{i}(x_1, \tau) \rho^\epsilon_{k}(x_2, \tau)   \, dx_1\,dx_2- \kappa_{\ell} \rho_{i}^\epsilon (x, \tau) \rho_{k}^{\epsilon}(x, \tau)\bigg|\nonumber\\
&
= \abs{\int_{\R^{2d}} \hat{K}^\epsilon_\ell(x_1- x_2) \sum_{i=1}^{I}p_i\delta\left(x-(\alpha_i x_1 +(1-\alpha_i)x_2)\right) \rho_i^\epsilon(x_1, \tau) \rho_k^\epsilon(x_2, \tau) \, dx_1 \, dx_2 - \kappa_\ell\rho_i^\epsilon(x, \tau)\rho_k^\epsilon(x, \tau)}\nonumber\\
&
\leq \sum_{i=1}^{I}p_i \abs{\int_{\R^{d}} \frac{1}{\alpha_i^d} \hat{K}^\epsilon_\ell\left(\frac{1}{\alpha_i}(x-x_2) \right) \rho_i^\epsilon\left(\frac{1}{\alpha_i}\left(x-(1-\alpha_i)x_2\right), \tau\right) \rho_k^\epsilon(x_2, \tau) \, dx_2 - k_1\rho_i^\epsilon(x, \tau)\rho_k^\epsilon(x, \tau)}\nonumber\\
&= \sum_{i=1}^{I}p_i \abs{\int_{\R^{d}} \frac{1}{\alpha_i^d}\hat{K}^\epsilon_\ell\left(\frac{1}{\alpha_i}(x-x_2) \right) \rho_i^\epsilon\left( x + \frac{(1-\alpha_i)}{\alpha_i}(x-x_2) , \tau\right) \rho_k^\epsilon(x_2, \tau) \, dx_2 - \kappa_\ell\rho_i^\epsilon(x, \tau)\rho_k^\epsilon(x, \tau)}\nonumber\\
&= ||\rho_i^\epsilon(x, \tau)\rho_k^\epsilon(x, \tau)  ||_{C^2_b(\R^d)}O(\epsilon^2),
\end{align*}
based on \cref{lem:molifier} and the regularity results  in  \cref{thm:regularityPDE} and  \cref{thm:regularityPIDE}. Substituting this back into \cref{eq:estimates2_ABC}, we obtain
\begin{align}
\Theta^{\ell, j}(\tau)
&
\leq k_{\ell}|| \rho^\epsilon_{i}(x, \tau) ||_{L^\infty}|| \rho_{k}(x, \tau) -  \rho^\epsilon_{k}(x, \tau)||_{L^\infty} + k_{\ell}|| \rho^\epsilon_{i}(x, \tau) -  \rho_{i}(x, \tau)||_{L^\infty} || \rho_{k}(x, \tau)||_{L^\infty} \nonumber\\
&
\qquad + ||\rho_i^\epsilon(x, \tau)\rho_k^\epsilon(x, \tau)  ||_{C^2_b(\R^d)}O(\epsilon^2)\nonumber\\
&\leq 2C\paren{\max_{j = 1, \cdots, J}||\rho_j(x, \tau) - \rho_j^\epsilon(x, \tau) ||_{L^\infty} + ||\rho_i^\epsilon(x, \tau)\rho_k^\epsilon(x, \tau)  ||_{C^2_b(\R^d)}O(\epsilon^2)}
\end{align}
where $C  = \paren{\max_{\ell = 1, \cdots, L} k_\ell}  \max_{j = 1, \cdots, J} \{\sup_{\epsilon > 0} \sup_{\tau\in[0, T_0]}||\rho^\epsilon_{j}(x, \tau)||_{L^\infty} \vee \sup_{\tau\in[0, T_0]} ||\rho_{j}(x, \tau)||_{L^\infty} \}$ is independent of $\epsilon$.

\item When $\mathscr{R}_\ell$ is of the type $S_i + S_i\rightarrow S_j$,  by plugging in \cref{Assume:measureTwo2One} and \cref{Assume:kernelTwo}, we obtain
\begin{equation*}
\Theta^{\ell, j}(\tau)
= \norm{ \frac{1}{2} \int_{\R^{2d}}  \hat{K}^\epsilon_\ell(x_1- x_2) m^\epsilon_\ell( x \, | x_1, x_2)  \rho^\epsilon_{i}(x_1, \tau) \rho^\epsilon_{i}(x_2, \tau)   \, dx_1\,dx_2- \kappa_{\ell} \rho_{i}(x, \tau) \rho_{i}(x, \tau)}_{L^\infty}.\nonumber\\
\end{equation*}
By choosing $k_\ell = 2\kappa_\ell$ and following the same arguments as the previous case, we shall obtain
\begin{equation*}
\Theta^{\ell, j}(\tau)
\leq 2C\paren{\max_{j = 1, \cdots, J}||\rho_j(x, \tau) - \rho_j^\epsilon(x, \tau) ||_{L^\infty} + ||\rho_i^\epsilon(x, \tau)\rho_k^\epsilon(x, \tau)  ||_{C^2_b(\R^d)}O(\epsilon^2)}.
\end{equation*}

\item When $\mathscr{R}_\ell$ is of the type $S_i \rightarrow S_j + S_k$, by plugging in  \cref{Assume:measureOne2Two} and  \cref{Assume:kernelOne}, using symmetry and choosing $k_\ell = \kappa_\ell$ we obtain
\begin{align}
\Theta^{\ell, j}(\tau) &= \beta_{\ell j} \norm{  \paren{   \int_{\R^d}  \hat{K}^\epsilon_\ell(z) \left( \int_{\R^d}  m^\epsilon_\ell(x, y | \,z) \,d y \right)  \rho^\epsilon_{i}(z, \tau) \, dz }- \kappa_{\ell} \rho_{i}(x, \tau)}_{L^\infty}\nonumber\\
&
= \beta_{\ell j}\kappa_\ell \norm{  \paren{   \int_{\R^d}  \left( \int_{\R^d}  \rho^\epsilon(|x-y|) \sum_{j=1}^{J}q_j\times\delta\left(z-(\beta_j x +(1-\beta_j)y)\right) dy  \right)  \rho^\epsilon_{i}(z, \tau) \, dz }-   \rho_{i}(x, \tau)}_{L^\infty}\nonumber\\
&
 =  \beta_{\ell j}\kappa_\ell \norm{\sum_{j=1}^{J}q_j\times\int_{\R^d}\rho^\epsilon(|x-y|) \rho_i^\epsilon(\beta_j x +(1-\beta_j)y, \tau)\, dy - \rho_i(x, \tau)}_{L^\infty}\nonumber\\
&
 \leq  \beta_{\ell j}\kappa_\ell \norm{\sum_{j=1}^{J}q_j\times\int_{\R^d}\rho^\epsilon(|x-y|) \left(\rho_i^\epsilon(\beta_j x +(1-\beta_j)y, \tau)- \rho_i(\beta_j x +(1-\beta_j)y, \tau)\right)\, dy  }_{L^\infty}\nonumber\\
 &
 \qquad+ \beta_{\ell j}\kappa_\ell \norm{\sum_{j=1}^{J}q_j\times\int_{\R^d}\rho^\epsilon(|x-y|) \rho_i(\beta_j x +(1-\beta_j)y, \tau)\, dy - \rho_i(x, \tau)}_{L^\infty}\nonumber\\
&
\leq \beta_{\ell j}\kappa_\ell \paren{ ||\rho_i^\epsilon(x, \tau)- \rho_i(x, \tau)||_{L^\infty}+||\rho_i(x, \tau) ||_{C^2_b(\R^d)}O(\epsilon^2)}
\end{align}
where the last term here again uses \cref{lem:molifier} and the regularity of solutions in  \cref{thm:regularityPDE} and  \cref{thm:regularityPIDE}.

\item When $\mathscr{R}_\ell$ is of the type $S_i + S_k \rightarrow S_j + S_r$, $i\not= k $, by plugging in  \cref{Assume:measureTwo2Two} and \cref{Assume:kernelTwo}, using symmetry and choosing $k_\ell = \kappa_\ell$ we obtain
\begin{align}\label{eq:estimates2_ABCD}
\Theta^{\ell, j}(\tau) &= \beta_{\ell j} \norm{ \paren{  \int_{\R^{2d}}  K^\epsilon_\ell(z, w) \left( \int_{\R^d}  m^\epsilon_\ell(x, y | z, w) \,dy \right)  \rho^\epsilon_{i}(z, \tau)  \rho^\epsilon_{k}(w, \tau) \, dz\, dw} - \kappa_{\ell}\rho_{i}(x, \tau)  \rho_{k}(x, \tau)}_{L^\infty}\nonumber\\
&
=  \beta_{\ell j} p \, \norm{  \paren{  \int_{\R^d}  K^\epsilon_\ell(x, y)  \rho^\epsilon_{i}(x, \tau)  \rho^\epsilon_{k}(y, \tau) \, dy} - \kappa_{\ell}\rho_{i}(x, \tau)  \rho_{k}(x, \tau)}_{L^\infty}\nonumber\\
&
\qquad +  \beta_{\ell j} (1-p) \norm{ \paren{  \int_{\R^d}  K^\epsilon_\ell(y, x)  \rho^\epsilon_{i}(y, \tau)  \rho^\epsilon_{k}(x, \tau) \, dy} - \kappa_{\ell}\rho_{i}(x, \tau)  \rho_{k}(x, \tau)}_{L^\infty}\nonumber\\
&
\leq \beta_{\ell j}2C\paren{\sup_{i = 1, \cdots, J}||\rho^\epsilon_i(x, \tau) ||_{C^2_b(\R^d)}O(\epsilon^2) + \max_{j = 1, \cdots, J}||\rho^\epsilon_{j}(x, \tau) - \rho_{j}(x, \tau)||_{L^\infty} },
\end{align}
where the last inequality comes from  \cref{eq:estimates1_AB}.

\item When $\mathscr{R}_\ell$ is of the type $S_i + S_i \rightarrow S_j + S_r$,  by  choosing $k_\ell = 2\kappa_\ell$ instead, we obtain the same estimates as \cref{eq:estimates2_ABCD}.
\end{enumerate}
This concludes the proof of the lemma.
\end{proof}

\section{Lemmas for Estimating Derivatives of the Nonlinear Term of MFM}\label{A:appendixLemDerivative}

\begin{lemma}\label{lem:diffNonlinearPIDE1}
For $\mathscr{R}_\ell$ to be a first or second order reactions,  $$|\partial_{x_i} \int_{\tilde{\vec{x}} \in \mathbb{X}^{(\ell)}}   \delta_{x}(\tilde{x}_1^{(j)})  K_\ell(\tilde{\vec{x}}) \, \left( \Pi_{k = 1}^{J} \Pi_{s = 1}^{\alpha_{\ell k}}  \rho_{k}(\tilde{x}_{s}^{(k)}, t)\right) \, d\tilde{\vec{x}}|\leq k_\ell \paren{2C_2\vee 1} \sup_{k = 1, \cdots, J}||\partial_{x_i} \rho_{k}(x, t)||_{C_b(\R^d)}.$$
\end{lemma}

\begin{proof}
When $\mathscr{R}_\ell$ is first order of the type $S_j \to \cdots$,
\begin{align*}
|\partial_{x_i} \int_{\tilde{\vec{x}} \in \mathbb{X}^{(\ell)}}   \delta_{x}(\tilde{x}_1^{(j)})  K_\ell(\tilde{\vec{x}}) \, \left( \Pi_{k = 1}^{J} \Pi_{s = 1}^{\alpha_{\ell k}}  \rho_{k}(\tilde{x}_{s}^{(k)}, t)\right) \, d\tilde{\vec{x}}| &= k_\ell |\partial_{x_i}    \rho_{j}(x, t)|,\\
& \leq k_\ell \sup_{k = 1, \cdots, J}||\partial_{x_i} \rho_{k}(x, t)||_{C_b(\R^d)}.
\end{align*}
When $\mathscr{R}_\ell$ is second order of the type $S_k + S_j\to\cdots$,
\begin{align*}
&|\partial_{x_i} \int_{\tilde{\vec{x}} \in \mathbb{X}^{(\ell)}}   \delta_{x}(\tilde{x}_1^{(j)})  K_\ell(\tilde{\vec{x}}) \, \left( \Pi_{k = 1}^{J} \Pi_{s = 1}^{\alpha_{\ell k}}  \rho_{k}(\tilde{x}_{s}^{(k)}, t)\right) \, d\tilde{\vec{x}}| \\
&
= |\partial_{x_i} \paren{\int_{\R^d} \hat{K}_\ell(x - y) \, \rho_{k}(y, t) \, dy \times \rho_j(x, t)}| ,\\
&
= |\paren{\int_{\R^d} \hat{K}_\ell(y) \, \partial_{x_i}\rho_{k}(x-y, t) \, dy \times \rho_j(x, t)}| +  | \paren{\int_{\R^d} \hat{K}_\ell(y) \, \rho_{k}(x-y, t) \, dy \times \partial_{x_i} \rho_j(x, t)}|,\\
&
 \leq 2k_\ell C_2\sup_{k = 1, \cdots, J}||\partial_{x_i} \rho_{k}(x, t)||_{C_b(\R^d)}.
\end{align*}
\end{proof}

\begin{lemma}\label{lem:diffNonlinearPIDE2}
For $\mathscr{R}_\ell$ to be a reaction producing one or two species,
\begin{align*}
&|\partial_{x_m} \int_{\tilde{\vec{x}} \in\mathbb{X}^{(\ell)}}  K_\ell(\tilde{\vec{x}}) \left( \int_{\vy \in \mathbb{Y}^{(\ell)}}   \delta_{x}(y_1^{(j)}) m_\ell(\vec{y}\, | \,\tilde{\vec{x}}) \,d \vec{y} \right) \left( \Pi_{k = 1}^J \Pi_{s = 1}^{\alpha_{\ell k}}  \rho_{k}(\tilde{x}_s^{(k)}, t)\right) \, d\tilde{\vec{x}}|\\
& \leq k_\ell\paren{ 2C_2 \vee 1} \sup_{k = 1, \cdots, J}||\partial_{x_m} \rho_{k}(x, t)||_{C_b(\R^d)}.
 \end{align*}
\end{lemma}
\begin{proof}
We discuss this case by case for the different types of reactions.
\begin{enumerate}[wide, labelwidth=!, labelindent=0pt]
\item When $\mathscr{R}_\ell$ is of the type $S_k \to S_j$,
\begin{align*}
&|\partial_{x_m} \int_{\tilde{\vec{x}} \in\mathbb{X}^{(\ell)}}  K_\ell(\tilde{\vec{x}}) \left( \int_{\vy \in \mathbb{Y}^{(\ell)}}   \delta_{x}(y_1^{(j)}) m_\ell(\vec{y}\, | \,\tilde{\vec{x}}) \,d \vec{y} \right) \left( \Pi_{k = 1}^J \Pi_{s = 1}^{\alpha_{\ell k}}  \rho_{k}(\tilde{x}_s^{(k)}, t)\right) \, d\tilde{\vec{x}}|\\
& = k_\ell |\partial_{x_m}  \rho_k(x, t)| \leq k_\ell \sup_{k = 1, \cdots, J}||\partial_{x_m} \rho_{k}(x, t)||_{C_b(\R^d)}.
 \end{align*}

\item When $\mathscr{R}_\ell$ is of the type $S_k + S_r \to S_j$, without loss of generality, let us assume $\alpha_i > 0$ for all $i = 1, \cdots, I$. Then,
\begin{align*}
&|\partial_{x_m} \int_{\tilde{\vec{x}} \in\mathbb{X}^{(\ell)}}  K_\ell(\tilde{\vec{x}}) \left( \int_{\vy \in \mathbb{Y}^{(\ell)}}   \delta_{x}(y_1^{(j)}) m_\ell(\vec{y}\, | \,\tilde{\vec{x}}) \,d \vec{y} \right) \left( \Pi_{k = 1}^J \Pi_{s = 1}^{\alpha_{\ell k}}  \rho_{k}(\tilde{x}_s^{(k)}, t)\right) \, d\tilde{\vec{x}}|\\
&
 = |\partial_{x_m} \int_{\R^{2d}} \hat{K}_\ell(y- z)    m_\ell(x |y, z) \rho_{k}(y, t)  \rho_r(z, t)\, dy\, dz|.\\
&
= |\partial_{x_m} \int_{\R^{2d}} \hat{K}_\ell(y- z)  \sum_{i=1}^{I}p_i \times \delta\left(x-(\alpha_i y+(1-\alpha_i)z)\right) \rho_{k}(y, t) \rho_r(z, t)\, dy\, dz|.\\
&
=|\sum_{i=1}^{I}p_i \times\partial_{x_m} \int_{\R^{d}} \frac{1}{\alpha_i^d}\hat{K}_\ell\paren{\frac{x-z}{\alpha_i}}\rho_{k}\paren{\frac{x - (1-\alpha_i) z}{\alpha_i}, t}  \rho_r(z, t)\,  dz|, \\
&
=|\sum_{i=1}^{I}p_i \times \int_{\R^{d}} \hat{K}_\ell(w) \partial_{x_m}\paren{ \rho_{k}\paren{x + (1 - \alpha_i)w, t} \rho_r(x - \alpha_i w, t) } \,  dw|\\
&
\leq k_\ell\times 2\sup_{k = 1, \cdots, J}||\partial_{x_m} \rho_{k}(x, t)||_{C_b(\R^d)}||\vecrho||_M
\leq 2k_\ell C_2\sup_{k = 1, \cdots, J}||\partial_{x_m} \rho_{k}(x, t)||_{C_b(\R^d)}.
 \end{align*}

\item When $\mathscr{R}_\ell$ is of the type $S_k + S_r \to S_j + S_i$,
\begin{align*}
&|\partial_{x_m} \int_{\tilde{\vec{x}} \in\mathbb{X}^{(\ell)}}  K_\ell(\tilde{\vec{x}}) \left( \int_{\vy \in \mathbb{Y}^{(\ell)}}   \delta_{x}(y_1^{(j)}) m_\ell(\vec{y}\, | \,\tilde{\vec{x}}) \,d \vec{y} \right) \left( \Pi_{k = 1}^J \Pi_{s = 1}^{\alpha_{\ell k}}  \rho_{k}(\tilde{x}_s^{(k)}, t)\right) \, d\tilde{\vec{x}}|\\
&
=|\partial_{x_m} \int_{\R^{2d}}  K_\ell(z, w) \left( \int_{\R^d}   m_\ell(x, y\, | \, z, w) \,dy \right)   \rho_{k}(z, t) \rho_{r}(w, t)\, dz\, dw|\\
&
\leq p\times|\partial_{x_m} \int_{\R^{d}}  \hat{K}_\ell(x - y)    \rho_{k}(x, t) \rho_{r}(y, t)\, dy| + (1-p)\times|\partial_{x_m} \int_{\R^{d}}  \hat{K}_\ell(y - x)    \rho_{k}(y, t) \rho_{r}(x, t)\, dy| \\
&
= p\times|\int_{\R^{d}}  \hat{K}_\ell(y)   \partial_{x_m} \paren{ \rho_{k}(x, t) \rho_{r}(x-y, t)}\, dy| + (1-p)\times|\int_{\R^{d}}  \hat{K}_\ell(-y)   \partial_{x_m} \paren{ \rho_{k}(x-y, t) \rho_{r}(x, t)}\, dy| \\
&
\leq k_\ell\times 2\sup_{k = 1, \cdots, J}||\partial_{x_m} \rho_{k}(x, t)||_{C_b(\R^d)}||\vecrho||_M
\leq 2k_\ell C_2\sup_{k = 1, \cdots, J}||\partial_{x_m} \rho_{k}(x, t)||_{C_b(\R^d)}.
\end{align*}

\item When $\mathscr{R}_\ell$ is of the type $S_k \to S_j + S_r$, without loss of generality, let us assume $\alpha_i > 0$ for all $i = 1, \cdots, I$. Then,
\begin{align*}
&|\partial_{x_m} \int_{\tilde{\vec{x}} \in\mathbb{X}^{(\ell)}}  K_\ell(\tilde{\vec{x}}) \left( \int_{\vy \in \mathbb{Y}^{(\ell)}}   \delta_{x}(y_1^{(j)}) m_\ell(\vec{y}\, | \,\tilde{\vec{x}}) \,d \vec{y} \right) \left( \Pi_{k = 1}^J \Pi_{s = 1}^{\alpha_{\ell k}}  \rho_{k}(\tilde{x}_s^{(k)}, t)\right) \, d\tilde{\vec{x}}|\\
&
=|\partial_{x_m} \int_{\R^{d}}  K_\ell(z) \left( \int_{\R^d}   m_\ell(x, y\, | \, z) \,dy \right)   \rho_{k}(z, t) \, dz |\\
&
=|\partial_{x_m} \int_{\R^{d}}  K_\ell(z) \left( \int_{\R^d}   \rho(|x-y|) \sum_{i=1}^{I}p_i\times\delta\left(z-(\alpha_i x +(1-\alpha_i)y)\right)\,dy \right)   \rho_{k}(z, t) \, dz |\\
&
=| \sum_{i=1}^{I}p_ik_\ell\times\partial_{x_m} \int_{\R^{d}} \rho(|x-y|)\rho_{k}(\alpha_i x +(1-\alpha_i)y, t) \, dy |\\
&
=| \sum_{i=1}^{I}p_ik_\ell\times\int_{\R^{d}} \rho(|w|) \partial_{x_m} \rho_{k}(x -(1-\alpha_i)w, t) \, dw |
\leq k_\ell \sup_{k = 1, \cdots, J}||\partial_{x_m} \rho_{k}(x, t)||_{C_b(\R^d)}.
\end{align*}
\end{enumerate}
This concludes the proof of the lemma.

\end{proof}


\end{appendices}



\begin{thebibliography}{99}

\bibitem[AP95]{AP:1995} Ambrosetti, Antonio, and Giovanni Prodi. A primer of nonlinear analysis. No. 34. Cambridge University Press, 1995.

\bibitem[CGV19]{CGV:2019} Caputo, M. Cristina, Thierry Goudon, and Alexis F. Vasseur. "Solutions of the 4-species quadratic reaction-diffusion system are bounded and $C^\infty$-smooth, in any space dimension." Analysis \& PDE 12.7 (2019): 1773-1804.

\bibitem[CFF19]{CFF:2019} Crevat, Joachim, Grégory Faye, and Francis Filbet. "Rigorous Derivation of the Nonlocal Reaction-Diffusion Fitzhugh--Nagumo System." SIAM Journal on Mathematical Analysis 51.1 (2019): 346-373.

\bibitem[D76a]{DoiSecondQuantA}
M.~Doi, \emph{Second quantization representation for classical many-particle
  system}, J. Phys. A: Math. Gen. \textbf{9} (1976), no.~9, pp. 1465--1477.

\bibitem[D76b]{DoiSecondQuantB}
M.~Doi, \emph{Stochastic theory of diffusion-controlled reaction}, J. Phys. A:
  Math. Gen. \textbf{9} (1976), no.~9, pp. 1479--1495.

\bibitem[ECM07]{ECM:2007} Erban, Radek, Jonathan Chapman, and Philip Maini. "A practical guide to stochastic simulations of reaction-diffusion processes." arXiv preprint arXiv:0704.1908 (2007).

\bibitem[E10]{E:2010} Evans, Lawrence C. Partial differential equations. Vol. 19. American Mathematical Soc., 2010.

\bibitem[FMT20]{FMT:2020}Fellner, Klemens, Jeff Morgan, and Bao Quoc Tang. "Global classical solutions to quadratic systems with mass control in arbitrary dimensions." Annales de l'Institut Henri Poincaré C, Analyse non linéaire. Vol. 37. No. 2. Elsevier Masson, 2020.

\bibitem[JG89]{JG:1989} Furter, J., and Michael Grinfeld. "Local vs. non-local interactions in population dynamics." Journal of Mathematical Biology 27.1 (1989): 65-80.

\bibitem[GB200]{GB:2000}Gibson, M. A., and J. Bruck. ``Efficient Exact Stochastic Simulation of Chemical Systems with Many Species and Many Channels.'' The Journal of Chemical Physics 104 (2000): 1876-1899.

\bibitem[G00]{G:2000} Gillespie, Daniel T. "The chemical Langevin equation." The Journal of Chemical Physics 113.1 (2000): 297-306.

\bibitem[HGG07]{HGG:2007} Hesthaven, Jan S., Sigal Gottlieb, and David Gottlieb. Spectral methods for time-dependent problems. Vol. 21. Cambridge University Press, 2007.

\bibitem[IMS20]{IMS:2020} Isaacson, Samuel A., Jingwei Ma, and Konstantinos Spiliopoulos. "Mean Field Limits of Particle-Based Stochastic Reaction-Diffusion Models." arXiv preprint arXiv:2003.11868 (2020).

\bibitem[GB96]{GB:1996} Gourley, S. A., and N. F. Britton. "A predator-prey reaction-diffusion system with nonlocal effects." Journal of Mathematical Biology 34.3 (1996): 297-333.

\bibitem[I13]{I:2013} Isaacson, Samuel A. "A convergent reaction-diffusion master equation." The Journal of chemical physics 139.5 (2013): 054101.

\bibitem[IZ18]{IZ:2018} Isaacson, Samuel A., and Ying Zhang. "An unstructured mesh convergent reaction–diffusion master equation for reversible reactions." Journal of Computational Physics 374 (2018): 954-983.

\bibitem[K20]{K:2020} Kostr\'{e}, M., C. Sch\"{u}tte, F.  No\'{e}, and M. J. del Razo. "Coupling particle-based reaction-diffusion simluations with reservoirs mediated by reaction-diffusion PDEs." arXiv preprint (2020): 2006.00003v1.

\bibitem[LSU68]{Ladyzenskaja}
O.A. Lady\v{z}enskaja, V.A. Solonnikov and N.N. Ural'ceva. Linear and Quasi-linear equations of parabolic type, {\em Transactions of Mahtematical Monographs, American Mathematical Society}, Volume 23, (1968).

\bibitem[LLN19]{LLN:2019} Lim, Tau Shean, Yulong Lu, and James Nolen. "Quantitative Propagation of Chaos in the bimolecular chemical reaction-diffusion model." arXiv preprint arXiv:1906.01051 (2019).
APA	

\bibitem[MNKS09]{MNKS:2009}
J.~Mu{\~n}oz-Garc{\'\i}a, Z.~Neufeld, B.~N. Kholodenko, and H.~M. Sauro,
  \emph{Positional information generated by spatially distributed signaling
  cascades}, PLoS Comp. Biol. \textbf{5} (2009), no.~3, e1000330.

\bibitem[NTY17]{NTY:2017} Ninomiya, Hirokazu, Yoshitaro Tanaka, and Hiroko Yamamoto. "Reaction, diffusion and non-local interaction." Journal of Mathematical Biology 75.5 (2017): 1203-1233.

\bibitem[NTS08]{NTS:2008}
S.~Neves, P.~Tsokas, A.~Sarkar, E.~Grace, P.~Rangamani, S.~Taubenfeld,
  C.~Alberini, J.~Schaff, R.~Blitzer, I.~Moraru, and R.~Iyengar, \emph{Cell
  shape and negative links in regulatory motifs together control spatial
  information flow in signaling networks}, Cell \textbf{133} (2008), no.~4,
  pp. 666--680.

\bibitem[PGB19]{PGB:2019} Pal, Swadesh, S. Ghorai, and Malay Banerjee. "Effect of kernels on spatio-temporal patterns of a non-local prey-predator model." Mathematical biosciences 310 (2019): 96-107.

\bibitem[P10]{P:2010} Pierre, Michel. "Global existence in reaction-diffusion systems with control of mass: a survey." Milan Journal of Mathematics 78.2 (2010): 417-455.

\bibitem[S18]{S:2018} Souplet, Philippe. "Global existence for reaction–diffusion systems with dissipation of mass and quadratic growth." Journal of Evolution Equations 18.4 (2018): 1713-1720.

\bibitem[SU17]{SU:2017} Schneider, Guido, and Hannes Uecker. Nonlinear PDEs. Vol. 182. American Mathematical Soc., 2017.

\bibitem[TS67]{TeramotoDoiModel1967}
E~Teramoto and N~Shigesada, \emph{Theory of bimolecular reaction processes in
  liquids}, Prog. Theor. Phys. \textbf{37} (1967), no.~1, pp. 29--51.

\bibitem[SVB13]{SVA:2013} Segal, B. L., V. A. Volpert, and Alvin Bayliss. "Pattern formation in a model of competing populations with nonlocal interactions." Physica D: Nonlinear Phenomena 253 (2013): 12-22.

\bibitem[SSS20]{SSS:2020} Shi, Qingyan, Junping Shi, and Yongli Song. "Effect of spatial average on the spatiotemporal pattern formation of reaction-diffusion systems." arXiv preprint arXiv:2001.11960 (2020).

\bibitem[ZI20]{ZI:2020} Zhang, Y., L. Clemens, J. Goyette, J. Allard, O. Dushek, and S. A. Isaacson. ``The Influence of Molecular Reach and Diffusivity on the Efficacy of Membrane-Confined Reactions.'' Biophysical J. 117 (2020): 1189-1201.

\bibitem[BMP96]{baras1996microscopic} F. Baras, M. M. Mansour, and J. Pearson, Microscopic simulation of chemical bistability in homogeneous systems, The Journal of chemical physics, 105 (1996), pp. 8257–8261.

\bibitem[BPM90]{baras1990microscopic} F. Baras, J. Pearson, and M. M. Mansour, Microscopic simulation of chemical oscillations in homogeneous systems, The Journal of chemical physics, 93 (1990), pp. 5747–5750.

\bibitem[DYK18]{donev2018efficient} A. Donev, C.-Y. Yang, and C. Kim, Efficient reactive brownian dynamics, The Journal of chemical physics, 148 (2018), p. 034103.

\bibitem[KNBGD17]{kim2017stochastic} C. Kim, A. Nonaka, J. B. Bell, A. L. Garcia, and A. Donev, Stochastic simulation of reaction-diffusion systems: A fluctuating-hydrodynamics approach, The Journal of chemi- cal physics, 146 (2017), p. 124110.

\end{thebibliography}
\end{document}